\documentclass{amsart}

\usepackage{amsmath,url,graphicx,amssymb,enumerate, amsthm, soul, verbatim, stix, ifsym, pinlabel}
\usepackage[pagebackref,colorlinks,citecolor=blue,linkcolor=blue,urlcolor=blue,filecolor=blue]{hyperref}
\usepackage{microtype}
\usepackage[all]{xy}
\usepackage{dsfont}
\usepackage{aliascnt}
\usepackage{etoolbox}
\usepackage{multicol}
\usepackage{tikz}
\usepackage{tikz-cd}
\usetikzlibrary{matrix,arrows,decorations.pathmorphing}
\usetikzlibrary{arrows,decorations.pathmorphing,backgrounds,fit,positioning,shapes.symbols,chains,calc}
\usepackage{verbatim}
\usepackage{pinlabel}

\usepackage[textsize = footnotesize]{todonotes}
\usepackage{mathtools}

\usepackage[hmargin=3cm,vmargin=3cm]{geometry}

\DeclareMathOperator{\sd}{sd}

\newcommand{\m}{\to}

\providecommand{\Ab}{\ensuremath\mathsf{Ab}}

\providecommand{\sgn}{\ensuremath\mathsf{sgn}}

\providecommand{\Star}{\ensuremath\mathsf{Star}}
\providecommand{\Link}{\ensuremath\mathsf{Link}}
\providecommand{\Linkhat}{\ensuremath\widehat{ \mathsf{Link}}}
\providecommand{\Vr}{\ensuremath\mathsf{Vert}}
\providecommand{\Cone}{\ensuremath\mathsf{Cone}}

\providecommand{\N}{\ensuremath\mathbb N_0}

\providecommand{\Z}{\ensuremath\mathbb Z}

\providecommand{\Q}{\ensuremath\mathbb Q}

\providecommand{\F}{\ensuremath\mathbb F}

\providecommand{\bT}{\ensuremath\mathbb T}

\providecommand{\cM}{\ensuremath\mathcal M}

\providecommand{\cT}{\ensuremath\mathcal T}

\DeclareMathOperator{\im}{im}

\DeclareMathOperator{\GL}{GL}

\DeclareMathOperator{\SL}{SL}
\DeclareMathOperator{\rank}{rank}

\DeclareMathOperator{\Sym}{Sym}

\DeclareMathOperator{\Lk}{Lk}
\DeclareMathOperator{\Ind}{Ind}

\DeclareMathOperator{\St}{St}

\newcommand{\B}{\operatorname{B}}
\newcommand{\BA}{\operatorname{BA}}
\newcommand{\DD}{\operatorname{DD}}
\newcommand{\TA}{\operatorname{TA}}

\newcommand{\BAA}{\operatorname{BAA}}

\definecolor{grey}{gray}{.5}

\newcommand{\inter}{\operatorname{int}}
\newcommand{\ls}{\left\{}
\newcommand{\rs}{\right\}}

\numberwithin{thmcounter}{section}
\newaliascnt{thmauto}{thmcounter}

\newaliascnt{Defauto}{thmcounter}

\newaliascnt{exauto}{thmcounter}

\newaliascnt{quauto}{thmcounter}

\newaliascnt{lemauto}{thmcounter}

\newaliascnt{propauto}{thmcounter}

\newaliascnt{corauto}{thmcounter}

\newaliascnt{remauto}{thmcounter}

\newaliascnt{convauto}{thmcounter}

\newaliascnt{obauto}{thmcounter}

\newaliascnt{conauto}{thmcounter}

\newaliascnt{claimauto}{thmcounter}

\newtheorem{atheorem}{Theorem}

\newtheorem*{ThmA'}{Theorem A'}
\newtheorem*{ThmB'}{Theorem B'}
\newtheorem*{ThmC'}{Theorem C'}

\newtheorem{theorem}[thmauto]{Theorem}
\newtheorem{lemma}[lemauto]{Lemma}
\newtheorem{proposition}[propauto]{Proposition}
\newtheorem{corollary}[corauto]{Corollary}
\newtheorem{conjecture}[conauto]{Conjecture}
\theoremstyle{definition}
\newtheorem{definition}[Defauto]{Definition}
\newtheorem{example}[exauto]{Example}

\newtheorem{remark}[remauto]{Remark}
\newtheorem{convention}[convauto]{Convention}
\newtheorem{claim}[claimauto]{Claim}
\newtheorem{observation}[obauto]{Observation}

\author{Benjamin Br\"uck}
\address{Institut f{\"u}r Mathematische Logik und Grundlagenforschung, Universit\"at M{\"u}nster, Germany}
\email{benjamin.brueck@uni-muenster.de}
\thanks{Benjamin Br\"uck was partially supported by the Deutsche Forschungsgemeinschaft (DFG, German Research Foundation) -- Project-ID 427320536 -- SFB 1442, as well as by Germany’s Excellence Strategy EXC 2044 -- 390685587, Mathematics Münster: Dynamics–Geometry–Structure.}

\author{Jeremy Miller}\thanks{Jeremy Miller was supported in part by NSF grant DMS-2202943 and a Simons Foundation collaboration grant}
\address{Department of Mathematics, Purdue University, USA}
\email{jeremykmiller@purdue.edu}

\author{Peter Patzt}
\thanks{Peter Patzt was supported in part by NSF grant DMS-2405310, a Simons Foundation collaboration grant, and by the Danish National Research Foundation through the Copenhagen Centre
for Geometry and Topology (DNRF151) and the European Research Council under the European Union’s
Seventh Framework Programme ERC Grant agreement ERC StG 716424 - CASe, PI Karim Adiprasito.}
\address{Department of Mathematics, University of Oklahoma, USA}
\email{ppatzt@ou.edu}

\author{Robin J. Sroka}
\thanks{Robin J.\ Sroka was supported by the European Research Council (ERC grant agreement No.772960) and the Danish National Research Foundation (DNRF92, DNRF151) as a PhD student at the University of Copenhagen, by NSERC Discovery Grant A4000 in connection with a Postdoctoral Fellowship at McMaster University, by the Swedish Research Council under grant no.\ 2016-06596 while in residence at Institut Mittag-Leffler in Djursholm, Sweden during the semester \emph{Higher algebraic structures in algebra, topology and geometry}, and by the Deutsche Forschungsgemeinschaft (DFG, German Research Foundation) -- Project-ID 427320536 -- SFB 1442, as well as by Germany’s Excellence Strategy EXC 2044 390685587, Mathematics Münster: Dynamics–Geometry–Structure as a Postdoctoral Research Associate at the University of Münster.}
\address{Department of Mathematics \& Statistics, McMaster University, Canada}
\curraddr{Mathematisches Institut, Universität Münster, Germany}
\email{robinjsroka@uni-muenster.de}

\author{Jennifer C. H. Wilson} \thanks{Jennifer Wilson was supported in part by NSF CAREER grant DMS-2142709 and NSF grant DMS-1906123.} 
\address{Department of Mathematics,  University of Michigan, USA}
\email{jchw@umich.edu}

\title{On the codimension-two cohomology of $\SL_n(\Z)$}

\begin{document}

\begin{abstract}
Borel--Serre proved that $\SL_n(\Z)$ is a virtual duality group of dimension $n \choose 2$ and the Steinberg module $\St_n(\Q)$ is its  dualizing module. This module is the top-dimensional homology group of the Tits building associated to $\SL_n(\Q)$.
We determine the ``relations among the relations'' of this Steinberg module. That is, we construct an explicit partial resolution of length two of the $\SL_n(\Z)$-module $\St_n(\Q)$. We use this partial resolution to show the codimension-$2$ rational cohomology group $H^{{n \choose 2} -2}(\SL_n(\Z);\Q)$ of $\SL_n(\Z)$ vanishes for $n \geq 3$. This resolves a case of a conjecture of Church--Farb--Putman. We also produce lower bounds for the codimension-$1$ cohomology of certain congruence subgroups of $\SL_n(\Z)$.
\end{abstract}
\maketitle

\tableofcontents


\section{Introduction}

\subsection{Steinberg modules and Borel--Serre duality}

Although the Steinberg module was initially introduced as an object of study in representation theory, the work of Borel--Serre \cite{BoSe} showed its importance to the study of cohomology of arithmetic groups.
In this paper, we are interested in the arithmetic group $\SL_n(\Z)$ and its congruence subgroups. We use their relationship to the Steinberg module for $\SL_n(\Q)$ to obtain new insights about the high-dimensional cohomology of these groups.

We begin by recalling the relevant definitions.
Let $\F$ be a field. The \emph{Tits building} associated to $\SL_n(\F)$, denoted $\cT_n(\F)$, is the geometric realisation of the poset of proper nonzero subspaces of $\F^n$. It is $(n-2)$-spherical by the Solomon--Tits Theorem \cite{Sol:Steinbergcharacterfinite} and its one potentially nonvanishing reduced homology group is called the Steinberg module \[\St_n(\F) := \widetilde H_{n-2}(\cT_n(\F)).\] The group $\SL_n(\F)$ acts on the Tits building and hence the Steinberg module is a representation of $\SL_n(\F)$. The results of Borel--Serre \cite{BoSe} show that $\SL_n(\Z)$ is a virtual duality group of dimension $n \choose 2$ and that the Steinberg module $\St_n(\Q)$ is the virtual dualizing module. Thus, for any finite index subgroup $\Gamma \subseteq \SL_n(\Z)$, we have $H^{k}(\Gamma;\Q) = 0$ for $k >{n \choose 2}$ and 
\begin{equation}
\label{Borel-Serre}
H^{{n \choose 2}-i}(\Gamma;\Q) \cong H_i(\Gamma;\St_n(\Q) \otimes \Q).
\end{equation}
 If $\Gamma$ is torsion-free, then $H^{{n \choose 2}-i}(\Gamma) \cong H_i(\Gamma;\St_n(\Q))$. 
We call the cohomology group $H^{{n \choose 2}-i}(\Gamma)$ the \emph{codimension-$i$} cohomology of $\Gamma$.

\subsection{Resolutions of Steinberg modules}
Borel--Serre duality is useful because it translates questions about the high-degree cohomology of $\SL_n(\Z)$ and its finite index subgroups to questions about their low-degree homology, at the cost of working with twisted coefficients. One can compute this group homology with twisted coefficients by constructing a projective  resolution of the coefficient module. 
The main achievement of this work is the construction of a partial resolution of $\St_n(\Q)$, 
$$ \cM_2 \longrightarrow \cM_1 \longrightarrow \cM_0 \longrightarrow \St_n(\Q) \longrightarrow 0,$$
where the $\SL_n(\Z)$-modules $\cM_i$ for $i=0,1,2$ have generating sets that allow for an easy description of the $\SL_n(\Z)$-action (see below). 

This extends work of Solomon--Tits \cite{Sol:Steinbergcharacterfinite}, Ash--Rudolph \cite{AR}, and Bykovski\u\i \, \cite{Byk}:
Given a basis $\beta=\{ \vec v_0,\ldots,\vec v_{n-1} \}$ of $\Q^n$, let $A_{\beta}$ be the full subcomplex of $\cT_n(\Q)$ of all subspaces that are spans of nonempty proper subsets of $\{ \vec v_0,\ldots,\vec v_{n-1} \}$. This subcomplex is called an \emph{apartment}. It is homeomorphic to $S^{n-2}$ and this sphere has a canonical fundamental class $[A_{\beta}]$ (well-defined up to sign).  See \autoref{ApartmentCircle}. 
\begin{figure}[h!]
\labellist
\Large \hair 0pt
\pinlabel {\small $v_0 $ } [l] at   0 18.5
\pinlabel {\small  $v_1 $ } [l] at  -3 0 
\pinlabel {\small  $v_2 $ } [r] at  21.5 0
\pinlabel {\small  $v_0 \oplus v_1$ } [r] at  0 9 
\pinlabel {\small   $v_1 \oplus v_2$ } [c] at  9 -2
\pinlabel {\small  $v_0 \oplus v_2$ } [l] at  10 9 
\endlabellist
\begin{center}{\includegraphics[scale=3]{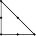}}\end{center} 
\caption{The apartment $A_{\{ \vec v_0, \vec v_1, \vec v_{2} \}}$ in $\cT_3(\Q)$}
\label{ApartmentCircle}
\end{figure}
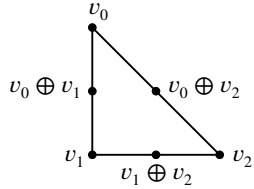 
 By the Solomon--Tits Theorem, the images of all these homology classes form a generating set for the Steinberg module $\St_n(\Q) = \widetilde H_{n-2}(\cT_n(\Q))$. 
Ash--Rudolph \cite{AR} showed that in fact, a generating set is given by the \emph{integral apartment classes} (also known as \emph{modular symbols}), i.e.~the images of $[A_{\beta}]$, where $\beta=\{ \vec v_0,\ldots,\vec v_{n-1} \}$ is a basis of $\Z^n$.
Bykovski\u\i \, \cite{Byk} extended this to a presentation.
Now our resolution computes the two-syzygies (the relations among the relations) of $\St_n(\Q)$.

Our partial resolution admits the following ``combinatorial'' description: We define the groups $\cM_i$ as quotients of free abelian groups, generated by formal symbols $[[\vec v_0,\ldots,\vec v_k]]$, where  $\vec v_0,\ldots,\vec v_k$ are certain sets of vectors in $\Z^n$. The action of $\SL_n(\Z)$ on $\Z^n$ induces an action on the sets of these formal symbols, given by $\phi \cdot [[\vec v_0,\ldots,\vec v_k]] = [[\phi(\vec v_0),\ldots,\phi(\vec v_k)]]$.

\noindent {\bf Generators:}
Let $\cM_0$ be the quotient of the free abelian group
 $$ \left\langle [[\vec v_0,\ldots,\vec v_{n-1}]] \; \middle| \; \text{ $\vec v_0,\ldots,\vec v_{n-1}$ a basis of $\Z^n$} \right\rangle_{\Z}$$ by the relations: \begin{enumerate}
\item $[[\vec v_0,\ldots,\vec v_{n-1}]]=\sgn(\sigma)[[\vec v_{\sigma(0)},\ldots,\vec v_{\sigma({n-1})}]]$  for all permutations $\sigma \in \Sym(n)$,
\item $[[\vec v_0, \ldots,\vec v_{n-1}]]=[[\pm \vec v_0,\ldots,\pm \vec v_{n-1}]]$, (with the $n$ signs each chosen independently). 
\end{enumerate}

\noindent {\bf Relations:} 
Let $\cM_1$ be the quotient of the free abelian group
 $$ \left\langle [[\vec v_0,\ldots,\vec v_n]] \; \middle| \;  \begin{array}{l}\text{there exist indices $i,j,k$ with } \\ \bullet \text{ $\vec v_0,\ldots,\vec v_{i-1}, \vec v_{i+1}, \ldots, \vec v_n$ is a basis of $\Z^n$,  } \\ \bullet \text{  $\vec v_i=\pm \vec v_j \pm \vec v_k$ or $\vec v_i=\pm \vec v_j \pm \vec v_k\pm \vec v_l$ for $l \not = i,j,k$} \end{array} \right\rangle_{\Z}$$ 
 by the relations
 \begin{enumerate}
\item $[[\vec v_0,\ldots,\vec v_n]]=\sgn(\sigma)[[\vec v_{\sigma(0)},\ldots,\vec v_{\sigma(n)}]]$  for all permutations $\sigma \in \Sym(n+1)$, 
\item $[[ \vec v_0,\ldots,\vec v_n]]=[[\pm \vec v_0, \ldots,\pm \vec v_n]]$  (signs chosen independently).
\end{enumerate}

\noindent {\bf Relations among the relations:} 
Let $\cM_2$ be the quotient of the free abelian group
 \[
\sbox0{\ensuremath{\left. [[\vec v_{0},\ldots,\vec v_{n+1}]]\; \middle|  \;\begin{array}{l}\text{there exist distinct indices $i,j,k,l,m$ with  } \\  \bullet \text{ $\vec v_{0},\ldots, \vec v_{i-1}, \vec v_{i+1}, \ldots, \vec v_{j-1}, \vec v_{j+1},,\ldots, \vec v_n$ is a basis of $\Z^n$,  } \\ \bullet \text{ $\vec v_i=\pm \vec v_k \pm \vec v_l$} \\ \bullet \text{ $\vec v_j = \pm \vec v_m \pm \vec v_l$ or $\vec v_j=\pm \vec v_m \pm \vec v_p$ for $p \not = i,j,k,l,m$}  \end{array} \right.}}
\mathopen{\resizebox{1.2\width}{1.1\ht0}{$\Bigg\langle$}}
\raisebox{2pt}{\usebox{0}}
\mathclose{\resizebox{1.2\width}{1.1\ht0}{$\Bigg\rangle$}}_{\Z}
\] 
by the relations
\begin{enumerate}
\item $[[\vec v_{0},\ldots,\vec v_{n+1}]]=\sgn(\sigma)[[\vec v_{\sigma(0)},\ldots,\vec v_{\sigma(n+1)}]]$ for all permutations $\sigma \in \Sym(n+2)$,
\item $[[\vec v_{0},\ldots,\vec v_{n+1}]]=[[\pm \vec v_{0},\ldots,\pm \vec v_{n+1}]]$  (signs chosen independently).
\end{enumerate}

\noindent {\bf Maps in the resolution:} 
Let $\delta_1\colon  \cM_1 \m \cM_0$ and $\delta_2 \colon  \cM_2 \m \cM_1$ be the maps 
\begin{align*}
\delta_1\colon   [[\vec v_0, \ldots,\vec v_n]] &\longmapsto \sum_i (-1)^i \; [[\vec v_0, \ldots,\vec v_{i-1}, \vec v_{i+1}, \ldots, \vec v_n]].   \\ 
\delta_2\colon   [[\vec v_0, \ldots,\vec v_{n+1}]] &\longmapsto \sum_i (-1)^{i+1} \; [[\vec v_0, \ldots,\vec v_{i-1}, \vec v_{i+1}, \ldots, \vec v_{n+1}]].   
\end{align*}
For these maps, we define the symbols $[[\vec v_0,\ldots,\vec v_{i-1}, \vec v_{i+1}, \ldots, \vec v_n]]$ and $[[\vec v_0,\ldots,\vec v_{i-1}, \vec v_{i+1}, \ldots, \vec v_{n+1}]]$ to be zero if the vectors do not span $\Z^n$. 

The map $\epsilon \colon  \cM_0 \m \St_n(\Q)$ is the ``integral apartment class map'' mentioned above. More precisely, it is defined as follows. If $[[\vec v_0,\ldots,\vec v_{n-1}]]$ is a generator of $\cM_0$, then $\beta = \{ \vec v_0,\ldots,\vec v_{n-1} \}$ is a basis of $\Z^n$ that comes with an order that is well-defined up  to the action of the alternating group. This order determines a sign of the corresponding apartment class $[A_{\beta}]$. Define $\epsilon$ via the formula:
\begin{align*} 
\epsilon \colon  \cM_0 & \longrightarrow \St_n(\Q) \\ 
 [[\vec v_0,\ldots,\vec v_{n-1}]] &\longmapsto [A_{\beta}]. 
\end{align*}

\begin{atheorem}\label{TheoremA}

The sequence \[\cM_2 \overset{\delta_2}{\longrightarrow} \cM_1 \overset{\delta_1}{\longrightarrow} \cM_0 \overset{\epsilon}{\longrightarrow} \St_n(\Q) \longrightarrow 0\] is exact.

\end{atheorem}

Exactness of \[\cM_0 \overset{\epsilon}{\longrightarrow} \St_n(\Q) \longrightarrow 0\] is due to Ash--Rudolph \cite{AR} and exactness of \[\cM_1 \overset{\delta_1}{\longrightarrow} \cM_0 \overset{\epsilon}{\longrightarrow} \St_n(\Q) \longrightarrow 0\] follows from Bykovski\u\i \ \cite{Byk}. See Church--Putman \cite{CP} for an alternate proof.

\subsection{Applications to the cohomology of $\SL_n(\Z)$} Using \autoref{TheoremA}, we show that the codimension-$2$ rational homology of $\SL_n(\Z)$ vanishes for large $n$. 

\begin{atheorem} \label{TheoremB}
For $n \geq 3$, $H^{{n \choose 2}-2}(\SL_n(\Z);\Q) \cong 0$.
\end{atheorem}

A standard transfer argument implies that $H^i(\GL_n(\Z);\Q )$ is a summand of $H^i(\SL_n(\Z);\Q )$. Thus $H^{{n \choose 2}-2}(\GL_n(\Z);\Q) \cong 0$ for $n \geq 3$. \autoref{TheoremB} resolves the codimension-$2$ case of a conjecture of Church--Farb--Putman \cite[Conjecture 2]{CFPconj}.

\begin{conjecture}[Church--Farb--Putman] \label{CFPconjecture}
For $n \geq i+2$, $H^{{n \choose 2}-i}(\SL_n(\Z);\Q) \cong 0$.
\end{conjecture}

For codimension $i=0$, this conjecture is true and due to Lee--Szczarba \cite{LS}. Vanishing in codimension-$0$ follows easily from Ash--Rudolph's \cite{AR} generating set for $\St_n(\Q)$. For codimension-$1$, the conjecture was established by Church--Putman \cite{CP} and follows from the Bykovski\u\i \  presentation \cite{Byk} of $\St_n(\Q)$. Similarly, \autoref{TheoremB} follows readily from our result determining the relations among the relations in the Steinberg module, \autoref{TheoremA}. 

The rational cohomology of $\SL_n(\Z)$ has been completely computed for $n \leq 7$ (Soul\'e \cite{Soule} for $n=3$, Lee--Szczarba \cite{LSK45} for $n=4$, and Elbaz-Vincent--Gangl--Soul\'e \cite{PerfFormModGrp} for $n=5$, $6$, and $7$). These calculations verify \autoref{CFPconjecture} for $n\leq 7$ and also show that the vanishing range predicted by \autoref{CFPconjecture} is not sharp for $n=3$, $5$, or $7$. This failure of sharpness is reflected in the fact that \autoref{TheoremB} implies that the codimension-$2$ rational cohomology vanishes for $n \geq 3$ while the codimension-$2$ case of \autoref{CFPconjecture} only concerns vanishing for $n \geq 4$.

\subsection{Applications to the cohomology of congruence subgroups}

The principal level $p$-congruence subgroup of $\SL_n(\Z)$, denoted $\Gamma_n(p)$, is defined to be the kernel of the mod-$p$ reduction map $$\SL_n(\Z) \longrightarrow \SL_n(\Z/p). $$ Using \autoref{TheoremA}, we obtain a combinatorial chain complex computing $H_1(\Gamma_n(p);\St_n(\Q)) \cong H^{{n \choose 2}-1}(\Gamma_n(p))$ (see \autoref{RelHomologyGamma}). In the case $p=3$ or $5$, we use this to obtain the following numerical estimate on the size of the codimension-$1$ homology. For a field $\F$  let $Gr_k^m(\F)$ denote the Grassmannian of $k$-planes in $\F^m$.

\begin{atheorem}\label{TheoremC}
For $p=3$ or $5$, $\displaystyle \dim_{\Q} H^{{n \choose 2}-1}(\Gamma_n(p);\Q) \geq p^{n-2 \choose 2} |Gr_2^{n}(\F_p)| \left(\frac{p-1}{2}\right)^{n-2}. $
\end{atheorem}

See \cite[Corollary 1.2]{MNP} for an upper bound of a similar flavour in the case $p=3$. 

\subsection{Proof structure and paper outline} 

Following Lee--Szczarba \cite{LS}, Church--Farb--Putman \cite{CFP}, and Church--Putman \cite{CP}, we will construct our resolution of $\St_n(\Q)$ by proving that certain simplicial complexes are highly-connected. The complexes relevant to our paper are called $\BAA_n$. These complexes are related to Maazen's complex of partial bases \cite{Maazen79, Maazen} with added augmentations in the sense of Church--Putman \cite{CP}. The augmentations are inspired by the Voronoi tessellation of the symmetric spaces associated to the groups $\SL_n(\Z)$. In \autoref{Sec2}, we define $\BAA_n$ and some variants. In the following sections, we adapt an argument of Church--Putman \cite{CP} to prove $\BAA_n$ is highly-connected, and in fact Cohen--Macaulay of dimension $n+1$. Because of the added complexity needed to study the relations among the relations, we use computer calculations for one step in the proof. In \autoref{Sec4}, we construct a retraction map that is used in the connectivity argument and is based on the Euclidean algorithm. The last step of the construction of this retraction uses that certain finite subcomplexes of $\BAA_4$ are highly-connected. This is proved in \autoref{Sec3} using computer calculations. In \autoref{Sec5} and \autoref{sec_connectivity_BAA}, we complete the proof that $\BAA_n$ is highly-connected. In \autoref{Sec6}, we recall a spectral sequence due to Quillen \cite{Quillen-Poset} concerning maps of posets. We use this spectral sequence in \autoref{Sec7} to study the codimension-$2$ cohomology of $\SL_n(\Z)$ and in \autoref{Sec8} to study the codimension-$1$ cohomology of congruence subgroups.

\subsection{Code for the computer calculations}

The code that was used to perform the computer calculations described in  \autoref{Sec3} is publicly available under \url{https://github.com/benjaminbrueck/codim2_cohomology_SLnZ}. Comments on runtime and verifiability of the results can be found in \autoref{sec_details_on_computer_calculations}.

\subsection{Acknowledgements}

We thank Alexander Kupers, Andrew Putman, Nathalie Wahl, and Dan Yasaki for helpful conversations and Lukas K\"uhne and Joshua Maglione for comments on the presentation of the python code. We thank our anonymous referee for their feedback. 

\section{Definitions} \label{Sec2}

Following Church--Putman \cite{CP} (building on ideas of Church--Farb--Putman \cite{CFP} and Lee--Szczarba \cite{LS}), we will construct our partial resolution of Steinberg modules using highly-connected complexes. In this section, we define the relevant complexes. 

\label{sec_definitions}

\subsection{Definition of $\B_n$ and $\BA_n$}

We begin by recalling a variant of Maazen's complex of partial bases $\B_n$ \cite{Maazen79, Maazen}. Church--Farb--Putman \cite{CFP} observed that high connectivity of this complex can be used to construct generators for Steinberg modules. We will then recall the definition of a large complex of augmented partial bases, denoted $\BA_n$, that was introduced by Church--Putman to study relations in Steinberg modules. 

\begin{definition}
Let $\Lambda$ be a PID. A vector $\vec v \in \Lambda^n$ is called \emph{primitive} if it spans a summand.
\end{definition}

\noindent Recall that a vector $\vec v \in \Lambda^n$ is primitive if and only if the greatest common divisor of its entries is a unit. If $\Lambda$ is a field, every $\vec v \in \Lambda^n\backslash \ls 0 \rs$ is primitive.

\begin{convention}
Throughout this text, we take $\Lambda$ to be either $\Z$ or $\F_p$. Given a primitive vector $\vec v$, the equivalence class $\pm \vec v$ is denoted by $v$. Given an equivalence class $v$, we let  $\vec v$ denote an (arbitrary) choice of representative of $v$.  We refer to equivalence classes $v$  as $\pm$-\emph{vectors}. If $\Lambda = \Z$, we also call $v$ a \emph{line}, since in this case there is a bijection between rank-$1$ summands (lines) in $\Z^n$ and equivalence classes of primitive vectors.

For $\vec v_0, \ldots, \vec v_k \in \Lambda^n$, we write $\langle \vec v_0, \ldots, \vec v_k \rangle_\Lambda$ for the $\Lambda$-span of $\vec v_0, \ldots, \vec v_k$. If $\Lambda = \Z$, we shorten this notation to $\langle \vec v_0, \ldots, \vec v_k \rangle \coloneqq \langle \vec v_0, \ldots, \vec v_k \rangle_\Z$.
\end{convention}

\begin{definition} \label{def:BA-simplices} 
Let $\Lambda$ be $\Z$ or $\F_p$.
  Let $V_n^\pm(\Lambda)$ be the set $$V_n^\pm(\Lambda) : = \{ v  \;  |  \; \vec v \in \Lambda^n  \text{  is primitive}\}.$$
  A subset $$\sigma = \{v_0, \dots, v_k\} \subset V_n^\pm(\Lambda)$$ of $(k+1)$ $\pm$-vectors is called
  \begin{enumerate}
  \item a \emph{standard} simplex, if $\langle \vec v_0, \ldots, \vec v_k \rangle_{\Lambda}$ is a rank-$(k+1)$ summand of $\Lambda^{n}$ and if $k=n-1$, the determinant of $[ \vec v_0 \cdots \vec v_{n-1}]$ is $\pm 1$; \label{2.3i}
  
  \item a \emph{2-additive} simplex, if (possibly after re-indexing) 
	\begin{equation*}
		\vec v_0 = \pm \vec v_1 \pm  \vec v_2
	\end{equation*}	  
  for some choice of signs and $\sigma \setminus \{v_0\}$ is a standard simplex.
  \end{enumerate}
\end{definition}

Note that the condition in \autoref{def:BA-simplices} \autoref{2.3i} that the determinant of $[ \vec v_0 \cdots \vec v_{n-1}]$ be $\pm 1$ is always true in the case $\Lambda = \Z$ and is only an extra condition in the case $\Lambda=\F_p$.

\begin{definition}
Let $\Lambda$ be $\Z$ or $\F_p$ and $n \in \N$. The simplicial complexes $\B_n^\pm(\Lambda)$ and $\BA_n^\pm(\Lambda)$ have $V^\pm_n(\Lambda)$ as their vertex set, and
  \begin{enumerate}
  \item the simplices of $\B_n^\pm(\Lambda)$ are all standard simplices;
  \item the simplices of $\BA_n^\pm(\Lambda)$ are all either standard simplices or 2-additive simplices.
  \end{enumerate}
\end{definition}

\subsection{Definition of $\BAA_n$}
 
 We now introduce a  larger complex denoted  $\BAA_n$. This  complex captures relations among the relations in Steinberg modules. The second ``A'' indicates that we add even more augmentations.

\begin{definition} \label{def:BAA-simplices}
Let $\Lambda$ be $\Z$ or $\F_p$.  A subset $$\sigma = \{v_0, \dots, v_k\} \subset V_n^\pm(\Lambda)$$ of $(k+1)$ $\pm$-vectors is called
  \begin{enumerate}
  \item a \emph{3-additive} simplex, if (possibly after re-indexing)
    $$\vec v_0 = \pm \vec v_1 \pm \vec v_2 \pm \vec v_3,$$
    for some choice of signs and $\sigma \setminus \{v_0\}$ is a standard simplex;

  \item a \emph{double-triple} simplex, if (possibly after re-indexing)
\begin{align*} \vec v_0 = \pm \vec v_2 \pm \vec v_3, \qquad   
   \vec v_1 = \pm \vec v_2 \pm \vec v_4,
    \end{align*}
    for some choice of signs and $\sigma \setminus \{v_0, v_1\}$ is a standard simplex;
  \item a \emph{double-double} simplex, if (possibly after re-indexing)
\begin{align*}  \vec v_0 = \pm \vec v_2 \pm \vec v_3, \qquad  
   \vec v_1 = \pm \vec v_4 \pm \vec v_5,
    \end{align*}
     for some choice of signs and $\sigma \setminus \{v_0, v_1\}$ is a standard simplex.
  \end{enumerate}
\end{definition}

We remark that the name ``double-triple" reflects that, after performing a change of basis and re-indexing, a double-triple simplex is represented by vectors of the form
$$ \vec v_0 = \pm \vec v_2 \pm \vec v_3 ,\quad \vec v_1 = \pm \vec v_2 \pm \vec v_3 \pm \vec v_4 ,\quad \vec v_2 ,\quad \vec v_3,\quad \ldots, \quad \vec v_{k}$$ 
for some choice of signs and a partial basis  $\vec v_2 , \vec v_3, \ldots, \vec v_{k}$ of $\Lambda^n$. See also \autoref{obs_list_facet_types}.

\begin{definition}
 Let $\Lambda$ be $\Z$ or $\F_p$ and $n \in \N$.  The simplicial complex $\BAA_n^\pm(\Lambda)$ has $V_n^\pm(\Lambda)$ as its vertex set. The simplices of $\BAA_n$ are precisely the ones introduced in \autoref{def:BA-simplices} and \autoref{def:BAA-simplices}.
\end{definition}

\begin{convention}
When $\Lambda = \Z$, we also write $\B_n$, $\BA_n$ and $\BAA_n$ for $\B_n^\pm(\Z)$, $\BA_n^\pm(\Z)$ and $\BAA_n^\pm(\Z)$, respectively.
\end{convention}

\subsection{Definition of $\Linkhat$, $\B^m_n$, $\BA^m_n$, $\BAA_n^m$ and $\Linkhat^{<}$} In this subsection, we specialise to the case $\Lambda=\Z$. We will define some subcomplexes of links of simplices. Throughout this section, let $\vec e_1, \ldots, \vec e_k$ denote the standard basis elements of $\Z^k$ and $e_1, \ldots, e_k$ the corresponding lines.
\begin{definition}
Let $n \in \N$ and let $X_n$ denote the complex $\B_n$, $\BA_n$ or $\BAA_n$. Consider a simplex $\sigma = \{w_0, \dots, w_k\}$ of $X$. Then $\Linkhat_{X_n}(\sigma)$ is defined to be the full subcomplex of $\Link_{X_n}(\sigma)$ on the vertex set
  $$\{ v \in \Link_{X_n}(\sigma) \; | \; \vec v \notin  \langle \vec w_0, \dots, \vec w_k \rangle\}.$$
\end{definition}

\begin{definition}
  Let $m,n \in \N$ and let $X_{m+n}$ denote the complex $\B_{m+n}$, $\BA_{m+n}$ or $\BAA_{m+n}$. Consider the standard simplex $\Delta^{m} = \{e_1, \dots, e_m\}$ contained in $X_{m+n}$. We set
  \begin{equation*}
  X^m_n \coloneqq \Linkhat_{X_{m+n}}(\Delta^{m}).
\end{equation*}
  When $X_{m+n}$ is $\B_{m+n}$, $\BA_{m+n}$ or $\BAA_{m+n}$, respectively, we write $\B_n^m$, $\BA_n^m$ or $\BAA_n^m$, respectively, for $X^m_n$. 
\end{definition}

The majority of the paper will be devoted to proving the following theorem. It is our main technical tool and the main theorems follow fairly quickly from it.

\begin{theorem} \label{BAAnmConn}
Let $n \geq 1$. Then $\BAA_n^m$ is $n$-connected. 
\end{theorem}

For the cases where $n+m \leq 2$, this immediately follows from results of Church--Putman: The complex $\BAA_1^0 = \B_1$ is a single point given by the unique line spanning $\Z$; the complex $\BAA_1^1 = \BA_1^1$ is isomorphic to the Cayley graph of $\Z$ with respect to the generating set $\{ e_1 \}$, so it is a line \cite[Proof of Theorem C', base case]{CP}; the complex $\BAA_2^0 = \BA_2$ is contractible as well by \cite[Remark 1.4]{CP}.

In the present article, we prove that \autoref{BAAnmConn} also hold if $n+m > 2$. In this case, the following stronger statement is true.

\begin{theorem} \label{BAAnmCM}
Let $n \geq 1$ and $m+n \geq 3$. Then $\BAA_n^m$ is Cohen--Macaulay of dimension $n+1$. 
\end{theorem}

Recall that a simplicial complex is called Cohen--Macaulay of dimension $d$ if it is $d$-dimensional, $(d-1)$-connected, and links of $p$-simplices are $(d-1-p)$-connected. 
In fact, to deduce the main theorems, it will be sufficient to prove the connectivity result \autoref{BAAnmConn} for the case $m=0$. The complexes $\BAA_n^m$ are ``relative versions'' of this complex that naturally show up in our inductive proof. The Cohen--Macaulay property is not directly needed for this induction or the main theorems; it however follows rather easily from the steps of our proof.

We need to consider the following subcomplex of links. 

\begin{definition}
  Let $m,n \in \N$ and let $X^m_n$ denote the complex $\B_n^m$, $\BA_n^m$ or $\BAA_n^m$. Consider a simplex $\sigma = \{w_0, \dots, w_k\}$ of $X^m_n$. Then $\Linkhat_{X^m_n}(\sigma)$ is defined to be the full subcomplex of $\Link_{X^m_n}(\sigma)$ on the vertex set
  $$\{ v \in \Link_{X_n^m}(\sigma) \; | \;  \vec v \notin \langle \vec e_1, \dots, \vec e_m, \vec w_0, \dots, \vec w_k \rangle\}.$$
\end{definition}

\begin{definition} \label{LinkhadDef}
  Let $R \in \mathbb{Z}_{\geq 1}$,  let $X^m_n$ denote the complex $\B_n^m$, $\BA_n^m$ or $\BAA_n^m$ and consider a simplex $\sigma = \{w_0, \dots, w_k\}$ of $X^m_n$. We write $\Linkhat^{< R}_{X^m_n}(\sigma)$ for the full subcomplex of $\Linkhat_{X^m_n}(\sigma)$ on the vertex set
  $$\{ v \in \Linkhat_{X^m_n}(\sigma) \;  | \; \vec v = c_1 \vec e_1 + \dots + c_{m+n} \vec e_{m+n} \text{ with } |c_{m+n}|  < R  \}.$$
  We will use the notation $\Linkhat^{<}_{X^m_n}(\sigma) = \Linkhat^{< R}_{X^m_n}(\sigma)$ with $R$ equal to the absolute value of the maximum nonzero last coordinate of the vectors in $\sigma$.
  
\end{definition} 

\section{Constructing the retraction}
\label{Sec4}
In this section, we present the main technical result that enables us to show that $\BAA_n$ is spherical of dimension $n+1$. To prove it, we build on ideas of Church--Putman \cite{CP} and Maazen \cite{Maazen79}.

\begin{theorem}\label{retraction}
  Let $n\geq 2$, $m\geq 0$ and  $w = \langle \vec w \rangle \in \BAA^m_n$ be a vertex. Assume the last coordinate of $\vec w \in \Z^{m+n}$ is nonzero. Then, the inclusion
  $$i\colon \Linkhat_{\BAA_n^m}^<(w) \hookrightarrow \Linkhat_{\BAA_n^m}(w)$$
  admits a \emph{topological} retraction
  $$r\colon \Linkhat_{\BAA_n^m}(w) \twoheadrightarrow \Linkhat^<_{\BAA_n^m}(w).$$
\end{theorem}

The definition of the retraction map occurring in \autoref{retraction} is inspired by work of Church--Putman \cite[Section 4]{CP} and Maazen \cite[Chapter III]{Maazen79}. On vertices, the retraction is given by using the Euclidean algorithm to ``reduce'' the last coordinate of vertices in the domain ``modulo $R$'', where $R > 0$ is the last coordinate of a fixed vector $\vec w \in \Z^{m+n}$ (compare with \autoref{DefRetB}). Church--Putman \cite[Section 4.1]{CP} demonstrated that this map can be used to prove that the complex of partial frames $\B_n$ is spherical (compare with \autoref{Bretraction}). However, the method does not directly apply to the complex of augmented partial frames $\BA_n$. To show that $\BA_n$ is spherical, Church--Putman \cite{CP} need to modify the definition of the retraction. The reason for this additional difficulty in the paper of Church--Putman comes from the following algebraic fact. For an integer $z$, let us denote by $(z \mod R) \in \{0, \dots, R-1\}$ the remainder of division of $z$ by $R > 0$. Let $R > 0$ and $a,b \geq 0$ be nonnegative integers. Then
 $$(a \mod R) + (b \mod R) = \begin{cases} (a+b \mod R) \text{ or }\\ (a+b \mod R) + R. \end{cases}$$
It is a consequence of this fact that the \emph{simplicial} retraction maps defined for $\B_n$ do not extend to \emph{simplicial} retraction maps for $\BA_n$. The problem is that, because of the second case in the equation above, the image of a $2$-additive simplex might not always span a simplex \cite[p.\ 1020]{CP}. To circumvent this problem, Church--Putman subdivide all problematic $2$-additive simplices, which they call \emph{carrying} simplices, in the domain of the retraction. They do this by adding a single vertex at the barycentre of every carrying simplex and extending this subdivision to the whole complex.  Then, they specify the value that their ``modified'' retraction takes at these newly introduced vertices and prove that the resulting map is a \emph{topological} retraction (compare with \autoref{BAretraction}).
In our construction of the retraction map for $\BAA_n$, i.e.\ \autoref{retraction}, we need to deal additionally with $3$-additive simplices. For these simplices, the following algebraic fact is the main source of trouble. Let $R > 0$ and $a,b,c \geq 0$ be nonnegative integers. Consider the integer $a+b+c$ or $a+b-c$. Then
\begin{enumerate}
  \item $(a \mod R) + (b \mod R) + (c \mod R) = \begin{cases} (a+b+c \mod R),\\ (a+b+c \mod R) + R \text{, or }\\ (a+b+c \mod R) + 2R \end{cases}$ for the sum $a+b+c$ and
  \item $(a \mod R) + (b \mod R) - (c \mod R) = \begin{cases} (a+b-c \mod R),\\ (a+b-c \mod R) + R \text{, or }\\ (a+b-c \mod R) - R \end{cases}$ for the sum $a+b-c$.
\end{enumerate}

Similarly to the difficulty for $\BA_n$, the problem is that, because of the second and third case in both item i) and ii), the image of a $3$-additive simplex might not always span a simplex. To circumvent this, we subdivide these problematic \emph{carrying} $3$-additive simplices by adding a new vertex at their barycentre and, in analogy with Church--Putman, construct a \emph{topological} retraction map for $\BAA_n$. However, since double-double and double-triple simplices might contain multiple problematic $2$-additive and $3$-additive facets (codimension-1 faces) we face novel difficulties. We not only need to explain how $2$-additive and $3$-additive simplices are subdivided but also need to describe how higher dimensional simplices can be subdivided in a \emph{compatible} fashion. For the most complicated case, we use computer calculations to show the existence of a retraction and do not make the corresponding subdivisions explicit (see \autoref{lem:extending-over-dt-internal-case} et seq.\ and \autoref{Sec3}).

We now start working towards the proof of \autoref{retraction} by introducing and fixing some notation. In the next subsection, we discuss the results that Church--Putman obtained for $\BA_n$ in greater detail. In each of the following subsections, we explain how the retraction maps can be defined on and extended over double-double, $3$-additive, and double-triple simplices, respectively.

\begin{convention} \label{bar-convention}
Throughout this section, we work in the setting of \autoref{retraction}. We fix the natural numbers $n\geq 2$, $m\geq 0$. For each line $v$ in $\Z^{m+n}$, we let $\bar v$ denote a choice of primitive vector in $v$ with nonnegative last coordinate. Note that the vector $\bar v$ is uniquely defined unless its last coordinate is zero. The line $w$ occurring in the statement of \autoref{retraction} is fixed throughout this section and $R$ always denotes the last coordinate of $\bar w$, which by assumption satisfies $R>0$. 
\end{convention}

The following notions will be frequently used for $\Linkhat_{\BAA^m_n}(w)$ in this section, and for $\BAA^m_n$ in the subsequent section.

\begin{definition} \label{relative-simplex-types}
  Let $\sigma$ be a simplex of $\Linkhat_{\BAA^m_n}(w)$ or $\BAA^m_n$. We say that $\sigma$ is a \emph{standard}, \emph{2-additive}, \emph{3-additive}, \emph{double-double} or \emph{double-triple} simplex of $\Linkhat_{\BAA^m_n}(w)$ or $\BAA^m_n$ if the \emph{underlying simplex}
  $$\sigma \ast \{e_1, \dots, e_m, w\} \text{ or } \sigma \ast \{e_1, \dots, e_m\}$$
  in $\BAA_{m+n}$ is a simplex of this type.
\end{definition}

\begin{example}
\label{ex_simplex_types}
Let $\vec v_1, \vec v_2 \in \Z^{m+n}$ such that $\{\vec v_1, \vec v_2, \vec e_1,\ldots, \vec e_m, \vec w\} $ is a partial basis. Then $\{ v_1, \langle \vec v_1 + \vec w \rangle\}$ is a standard simplex in $\BAA_{m+n}$ and  $\BAA^m_n$, but it is a 2-additive simplex in $\Linkhat_{\BAA^m_n}(w)$. Similarly, $\{ v_1, v_2, \langle \vec v_1 + \vec v_2  + \vec e_1 \rangle\}$ is a 3-additive simplex in $\BAA^m_n$ (and in $\Linkhat_{\BAA^m_n}(w)$) and $\{ v_1, \langle \vec v_1 + \vec e_1 \rangle , \langle \vec v_1 + \vec w \rangle\}$ is a double-triple simplex in $\Linkhat_{\BAA^m_n}(w)$.
\end{example}

Any simplex $\sigma$ that is not a standard simplex contains a unique minimal face that determines its type. This is the content of the next definition.

\begin{definition} \label{minimal-simplices-and-additive-core}
  Let $\sigma$ be a simplex of $\Linkhat_{\BAA^m_n}(w)$ or $\BAA^m_n$.
  \begin{enumerate}
  \item The simplex $\sigma$ is called a \emph{minimal} simplex of $2$-additive, $3$-additive, double-double, or double-triple type (in $\Linkhat_{\BAA^m_n}(w)$ or $\BAA^m_n$) if $\sigma$ is of this type and $\sigma$ does not contain a proper face  also  of this type.
  \item The \emph{additive core} of a nonstandard simplex $\sigma$ is the unique minimal face of $\sigma$ with the same type as $\sigma$.
  \end{enumerate}
\end{definition}

\begin{example}
\label{ex_minimal_simplex}
Let again $\vec v_1, \vec v_2 \in \Z^{m+n}$ such that $\{\vec v_1, \vec v_2, \vec e_1,\ldots, \vec e_m, \vec w\} $ is a partial basis. The simplices $\{ v_1, v_2, \langle \vec v_1 + \vec v_2 \rangle \}$ and $\{ v_1, \langle \vec v_1 + \vec w \rangle\}$ are minimal in $\Linkhat_{\BAA^m_n}(w)$. In particular, these simplices form their own additive cores.
The simplex $\{ v_1, v_2,  \langle \vec v_1 + \vec w \rangle\}$ is not minimal in $\Linkhat_{\BAA^m_n}(w)$. Its additive core is $\{ v_1, \langle \vec v_1 + \vec w \rangle\}$.
\end{example}

The next definition is parallel to \cite[Definition 4.9]{CP}. 

\begin{definition} \label{external-wrelated-internal-simplices}
  Let $\sigma$ be a simplex of $\Linkhat_{\BAA^m_n}(w)$ or $\BAA^m_n$.
  \begin{enumerate}
  \item We say that $\sigma$ is \emph{external} if the additive core of the underlying simplex in $\BAA_{m+n}$ contains $e_i$ for some $1 \leq i \leq m$.
  \item We say that $\sigma$ is \emph{$w$-related} if $\sigma$ is a simplex in $\Linkhat_{\BAA^m_n}(w)$ and the additive core of the underlying simplex in $\BAA_{m+n}$ contains $w$.
  \item We say that $\sigma$ is \emph{internal} if the additive core of the underlying simplex in $\BAA_{m+n}$ is contained in $\sigma$.
  \end{enumerate}
\end{definition}

Note that an internal simplex is neither external nor $w$-related.

\begin{example}
Among the simplices in \autoref{ex_simplex_types} and \autoref{ex_minimal_simplex}, in $\Linkhat_{\BAA^m_n}(w)$,
\begin{itemize}
  \begin{minipage}{0.35\linewidth}  
\item $\{ v_1,  \langle \vec v_1 + \vec w \rangle\}$ is $w$-related,
\item $\{ v_1, v_2, \langle \vec v_1 + \vec v_2  + \vec e_1 \rangle\}$ is external,
\end{minipage} \quad
  \begin{minipage}{0.5\linewidth}  
\item $\{ v_1, \langle \vec v_1 + \vec e_1 \rangle , \langle \vec v_1 + \vec w \rangle\}$ is external and $w$-related,
\item $\{ v_1, v_2,  \langle \vec v_1 + \vec v_2 \rangle\}$ is internal.
\end{minipage}
\end{itemize}
\end{example}

\subsection{Definition on vertices, standard and 2-additive simplices}

We start by defining the retraction maps on the set of vertices  $\Vr(\Linkhat_{\BAA_n^m}(w))$ of the simplicial complex $\Linkhat_{\BAA_n^m}(w)$.

\begin{definition} \label{DefRetB} Let $\Vr(X)$ denote the vertex set of the simplicial complex $X$. Then, we define
\begin{align*} r\colon  \Vr(\Linkhat_{\BAA_n^m}(w)) & \longrightarrow \Vr(\Linkhat^<_{\BAA_n^m}(w)) \\
v & \longmapsto \langle \bar v - a \bar w \rangle
\end{align*} 
  where $a \in \Z$ is chosen so that $\bar v - a \bar w$ has last coordinate in the interval $[0,R)$.
\end{definition}

The constant $a \in \Z$ in \autoref{DefRetB} is determined by the Euclidean algorithm. We note that, although the vector $\bar v$ is not uniquely determined by $v$ if its last coordinate is zero, the line $r(v)$ is still uniquely determined because $r(v) = v$. More generally, we observe that $r(v) = v$ if the last coordinate of $\bar v$ is contained in $[0, R)$.

\begin{convention}
  Consider $v \in \Vr(\Linkhat_{\BAA_n^m}(w))$, then the line $r(v) \in \Vr(\Linkhat^<_{\BAA_n^m}(w))$ is spanned by a vector $\overline{r(v)}$. Recall that this vector is \emph{not well-defined} if the last coordinate of $\overline{r(v)}$ is zero (see \autoref{bar-convention}). We use the following notational convention for $\overline{r(v)}$. Let $\bar v$ be a vector representing $v$ that has nonnegative last coordinate $aR+b$ where $a \in \Z$ and $b \in [0, R)$. In this situation, $\overline{r(v)}$ always denotes the vector $\bar v - a \bar w$.
\end{convention}

Before we discuss the effect of this map on standard and $2$-additive simplices, we record some facts that will help us to calculate the value of $r$ on certain vertices. The following observation is elementary but useful.

\begin{observation} \label{ObservationColumnOperations}
  Let $\{\bar v_0, \bar v_1, \bar v_2, \ldots, \bar v_p, \bar w\}$ be a partial basis for $\Z^{m+n}$. Then, if we replace any element $\bar v_i$ by  $\bar v_i + a \bar w$ for any $a \in \Z$, the result is still a partial basis for $\Z^{m+n}$ and spans the same summand. In particular, $\bar v_i + a \bar w$ is necessarily primitive.
\end{observation}

The next lemma describes some properties of the map $r$. Its proof is easy and left to the reader.

\begin{lemma} \label{lem_properties_r}
  Let $v \in \Vr(\Linkhat_{\BAA_n^m}(w))$ and let $\epsilon_1, \epsilon_2 \in \{-1,+1\}$ be two signs. Then, the map $r$ introduced in \autoref{DefRetB} has the following properties.
  \begin{enumerate}
  \item \label{Observation-r-externally-additive} If $u \in \Vr(\Linkhat_{\BAA_n^m}(w)) \sqcup \{e_1, \dots, e_m\}$ is a line such that the last coordinate of $\bar u$ is zero, then
    $$r(\langle \epsilon_1 \bar v + \epsilon_2 \bar u \rangle) = \langle \epsilon_1 \overline{r(v)} + \epsilon_2 \bar u \rangle.$$
  \item \label{r(v+-w)} It holds that
    $$r(\langle \epsilon_1 \bar v + \epsilon_2 \bar w \rangle) = r(v), \text{ if } \epsilon_1 = \epsilon_2,$$
    and, if $\epsilon_1 \neq \epsilon_2$, then
    \begin{equation*}
      r(\langle \epsilon_1 \bar v + \epsilon_2 \bar w \rangle) =  
      \begin{cases}
	r(v),  \text{ if the last coordinate of } \bar v \text{ is in } \{0\} \sqcup [R, \infty),\\
	\langle \epsilon_1 \overline{r(v)} + \epsilon_2 \bar w \rangle = \langle \bar w - \overline{r(v)} \rangle = \langle \bar w - \bar v \rangle, \text{ if the last coordinate of } \bar v \text{ is in } (0,R).
      \end{cases}
    \end{equation*}
  \item \label{Observation-r-internally-additive} Given two vertices $v_1,v_2 \in \Vr(\Linkhat_{\BAA_n^m}(w))$. Let $a_i R + b_i$ for $a_i \geq 0$ and $b_i \in [0, R)$ denote the last coordinate of $\bar v_i$. Then, 
    \begin{equation*}
      r( \langle \bar v_1+\bar v_2\rangle )  =
      \begin{cases}
        \langle \overline{r(v_1)} + \overline{r(v_2)}\rangle, \text{ if } b_1 + b_2 \in [0,R), \\
         \langle \overline{r(v_1)} + \overline{r(v_2)} -\bar w \rangle, \text{ if } b_1 + b_2 \in [R,2R).
      \end{cases}
    \end{equation*}
  \end{enumerate}
\end{lemma}

We now discuss the effect of the retraction on standard simplices. This has been studied by Church--Putman in Section 4.1 of \cite{CP} as part of their proof that $\B^m_n$ is a Cohen--Macaulay complex of dimension $n-1$ \cite[Theorem 4.2]{CP}.

\begin{definition}
  Let $\Linkhat_{\B_n^m}(w)$ and $\Linkhat_{\B_n^m}^{<}(w)$ denote the subcomplexes of $\Linkhat_{\BAA_n^m}(w)$ and $\Linkhat_{\BAA_n^m}^{<}(w)$, respectively, that consists of all standard simplices in the sense of \autoref{relative-simplex-types}.
\end{definition}

Note that the vertex sets of $\Linkhat_{\B_n^m}(w)$ and $\Linkhat_{\BAA_n^m}(w)$ are equal. The following result is a case of Church--Putman {\cite[Lemma 4.5]{CP}}.

\begin{proposition}[{\cite[Lemma 4.5]{CP}}] \label{Bretraction}
  Let $n, m \geq 0$. Then, \autoref{DefRetB} induces a simplicial map
  $$r\colon  \Linkhat_{\B_n^m}(w) \twoheadrightarrow \Linkhat^<_{\B_n^m}(w) \hookrightarrow \Linkhat_{\BAA_n^m}^{<}(w)$$
  that restricts to the inclusion on the subcomplex $\Linkhat^<_{\B_n^m}(w)$ of $\Linkhat_{\B_n^m}(w)$. In particular, $\Linkhat^<_{\B_n^m}(w)$ is a simplicial retract of $\Linkhat_{\B_n^m}(w)$.
\end{proposition}

We now explain how Church--Putman extended the \emph{simplicial} retraction of $\Linkhat_{\B_n^m}(w)$ onto $\Linkhat^<_{\B_n^m}(w)$ over $2$-additive simplices to a \emph{topological} retraction between the following two simplicial complexes. 

\begin{definition}
  Let $\Linkhat_{\BA_n^m}(w)$ and $\Linkhat_{\BA_n^m}^{<}(w)$ denote the subcomplexes of $\Linkhat_{\BAA_n^m}(w)$ and $\Linkhat_{\BAA_n^m}^{<}(w)$ respectively that consist of all standard and $2$-additive simplices in the sense of \autoref{relative-simplex-types}.
\end{definition}

The following definition captures the reason why \autoref{DefRetB} does not induce a \emph{simplicial} retraction $r\colon  \Linkhat_{\BA_n^m}(w) \twoheadrightarrow \Linkhat_{\BA_n^m}^{<}(w)$ as one might initially hope.

\begin{definition} \label{carrying-2-additive}
  Let $\sigma = \tau_1 \ast \tau_2$ be a $2$-additive simplex in $\Linkhat_{\BAA_n^m}(w)$, where $\tau_1$ is a minimal $2$-additive simplex and $\tau_2$ is a standard simplex. $\sigma$ is called \emph{carrying} if one of the following equivalent conditions holds
  \begin{enumerate}
  \item The set $r(\tau_1)$ does not span a simplex in $\Linkhat_{\BA_n^m}^{<}(w)$. (Note the $\BA^m_n$ subscript).
  \item $\tau_1 = \{v_0, v_1, v_2 = \langle \bar v_0 + \bar v_1 \rangle\}$ is internally $2$-additive with $b_0 + b_1 \in [R,2R)$, where $a_i R + b_i$ with $a_i \geq 0$ and $b_i \in [0, R)$ is the last coordinate of $\bar v_i$. 
    \end{enumerate}
\end{definition}
We remark that in Condition ii), $v_2$ is the unique vertex in $\tau_1$ with last coordinate of $\bar v_i$  maximal and $r(\tau_1) = \{r(v_0), r(v_1), r(v_2) = \langle \overline{r(v_0)} + \overline{r(v_1)} - \bar w \rangle \}$ by Part iii) of \autoref{lem_properties_r}. The equivalence of i) and ii) in \autoref{carrying-2-additive} follows from \cite[§4.4.\ Claim 1-4 and the discussion on p.\ 1022]{CP}.

\begin{remark}
  In the first condition of \autoref{carrying-2-additive}, we highlighted the $\BA^m_n$ subscript because $r(\tau_1)$ does form a $w$-related $3$-additive simplex in $\Linkhat_{\BAA_n^m}^{<}(w)$, as visible in the second condition.
\end{remark}

\begin{example}
  Let $\bar w=\begin{bmatrix} 1\\ 0\\ 0\\ 10 \end{bmatrix}$, $\bar v_0 = \begin{bmatrix} 0\\1 \\ 0\\ 9\end{bmatrix}$, and $\bar v_1=\begin{bmatrix} 0\\ 0\\ 1\\ 6\end{bmatrix}$. Then $\{v_0,v_1, v_2 = \langle \bar v_0+ \bar v_1 \rangle \}$
    is a $2$-additive simplex in $\Linkhat_{\BAA^1_3}(w)$. However,
    $$r(\{v_0,v_1, v_2\}) = \{v_0,v_1, r(v_2) = \langle \bar v_1 + \bar v_2 - \bar w \rangle \}$$
    is not a simplex of $\Linkhat^<_{\BA^1_3}(w)$ and therefore $\{v_0, v_1, v_2\}$ is an example of a carrying $2$-additive simplex.
\end{example}

To circumvent this problem and, instead, construct a \emph{topological} retraction
$$r\colon  \Linkhat_{\BA_n^m}(w) \twoheadrightarrow \Linkhat_{\BA_n^m}^{<}(w),$$
Church--Putman modify the definition of $r$ on all carrying $2$-additive simplices. To do this, they pass to the following subdivision of $\Linkhat_{\BA_n^m}(w)$.

\begin{definition} \label{definition-sd-BA}
  Let $\sd(\Linkhat_{\BA_n^m}(w))$ denote the coarsest subdivision of $\Linkhat_{\BA_n^m}(w)$, where every carrying minimal $2$-additive simplex $\tau_1 = \{v_0,v_1, v_2 = \langle \bar v_0 + \bar v_1 \rangle\}$ is subdivided by inserting a single vertex $t(\tau_1)$ at the barycentre of $\tau_1$.
\end{definition}

Concretely, $\sd(\Linkhat_{\BA_n^m}(w))$ in \autoref{definition-sd-BA} is constructed as follows: Let $\sigma = \tau_1 \ast \tau_2$ be a $2$-additive simplex of $\Linkhat_{\BA_n^m}(w)$, where $\tau_1 = \{v_0,v_1, v_2 = \langle \bar v_0 + \bar v_1 \rangle\}$ is a carrying minimal $2$-additive simplex and $\tau_2$ is standard. Then, when passing from $\Linkhat_{\BA_n^m}(w)$ to $\sd(\Linkhat_{\BA_n^m}(w))$, each such simplex $\sigma$ is replaced by $\sd(\sigma)$, its subdivision into the three simplices $\{v_0, \ldots, \hat{v}_i, \ldots, v_2,  t(\tau_1) \} \ast \tau_2$ for $i=0,1,2$. Here the notation $\hat{v}_i$ means $v_i$ is omitted. 
Note that $\Linkhat_{\B_n^m}(w)$ and $\Linkhat_{\BA^m_n}^{<}(w)$ are subcomplexes of $\sd(\Linkhat_{\BA_n^m}(w))$.

The following is the main technical result of Church--Putman \cite{CP}, and the key input for their proof that $\BA^m_n$ is a Cohen--Macaulay complex of dimension $n$ \cite[Theorem C']{CP}.

\begin{proposition}[{\cite[Proposition 4.17.]{CP}}] \label{BAretraction}
  Let $n \geq 2$ and $m \geq 0$. Then, the simplicial map constructed in \autoref{Bretraction}
  $$r\colon  \Linkhat_{\B_n^m}(w) \twoheadrightarrow \Linkhat^<_{\B_n^m}(w) \hookrightarrow \Linkhat_{\BAA_n^m}^{<}(w)$$
  extends to a simplicial map
  $$r\colon  \sd(\Linkhat_{\BA_n^m}(w)) \twoheadrightarrow \Linkhat^<_{\BA_n^m}(w) \hookrightarrow \Linkhat_{\BAA_n^m}^{<}(w)$$
  that restricts to the inclusion on the subcomplex $\Linkhat^<_{\BA_n^m}(w)$ of $\sd(\Linkhat_{\BA_n^m}(w))$. The value of $r$ on the barycentre $t(\tau_1)$ of a carrying minimal $2$-additive simplex $\tau_1 = \{v_0, v_1, v_2 = \langle \bar v_0 + \bar v_1 \rangle \}$ is defined by the formula
  $$r(t(\tau_1)) = \langle \overline{r(v_l)} - \bar w \rangle$$
  where $l \in \{0,1\}$ is arbitrarily chosen, i.e.\ $v_l$ is one of the two lines in $\tau_1$ with the property that the last coordinate of $\bar v_l$ is not maximal. In particular, it follows that $\Linkhat^<_{\BA_n^m}(w)$ is a \emph{topological} retract of $\Linkhat_{\BA_n^m}(w) \cong \sd(\Linkhat_{\BA_n^m}(w))$.
\end{proposition}

 This completes our discussion of the definition of $r$ on vertices, and standard and $2$-additive simplices. We close this subsection by presenting a proof of the following lemma. It will frequently be used to reduce the question of whether the map $r$ extends over a simplex $\sigma = \tau_1 \ast \tau_2$ with additive core $\tau_1$ to the question whether $r$ extends over the additive core $\tau_1$. To shorten notation, we write $\langle \nu \rangle \coloneqq \langle \vec v \mid \langle \vec v \rangle \in \nu \rangle$ for the $\Z$-linear span of a set of lines $\nu$ in $\Z^{m+n}$.

\begin{lemma}
\label{lem_key_extend_from_additive_core}
  Let $\sigma = \tau_1 \ast \tau_2$ be a simplex of $\Linkhat_{\BAA_n^m}(w)$ such that the additive core of $\sigma$ is contained in $\tau_1$ and $\tau_2$ is a standard simplex. Let $\nu$ be a set of lines in $\Z^{m+n}$ such that $\langle \nu \rangle \subseteq \langle \tau_1 \cup \ls e_1, \ldots, e_m, w \rs\rangle$.
If $\nu$ spans a simplex in $\Linkhat_{\BAA_n^m}^<(w)$, then $\nu \ast r(\tau_2)$ spans a simplex of the same type.
\end{lemma}
\begin{proof}
  Since $\nu$ is a simplex in $\Linkhat_{\BAA_n^m}^<(w)$, there exists a (not necessarily unique) maximal standard simplex $\nu' \subseteq \nu$ that is contained in $\nu$. Observe that $\langle \nu' \rangle \oplus \langle \ls e_1, \dots, e_m, w \rs\rangle = \langle \nu \cup \ls e_1, \dots, e_m, w \rs\rangle$ is a direct summand of $\Z^{m+n}$. Since $\tau_1 \ast \tau_2$ is a simplex in $\Linkhat_{\BAA_n^m}(w)$ and the additive core of it is contained in $\tau_1$, it follows that $\langle \tau_1 \cup \ls e_1, \ldots, e_m, w \rs \rangle \oplus \langle \tau_2 \rangle$ is a direct summand of $\Z^{m+n}$.  The assumption that $\langle \nu \rangle \subseteq \langle \tau_1 \cup \ls e_1, \ldots, e_m, w \rs\rangle$ implies that $\langle \nu' \rangle \oplus \langle \ls e_1, \dots, e_m, w \rs\rangle \subseteq \langle \tau_1 \cup \ls e_1, \dots, e_m, w \rs\rangle$. We conclude that $\langle \nu' \rangle \oplus \langle \ls e_1, \dots, e_m, w \rs \rangle  \oplus \langle \tau_2 \rangle$ is a direct summand of $\Z^{m+n}$ as well, using e.g.\ \cite[Lemma 2.6]{CP}. \autoref{Bretraction} implies that $r(\tau_2)$ is a standard simplex and \autoref{ObservationColumnOperations} yields
  \begin{equation*}
    \langle \ls e_1, \dots, e_m, w\rs \rangle \oplus \langle \tau_2 \rangle = \langle \ls e_1, \dots, e_m, w \rs \rangle \oplus \langle r(\tau_2) \rangle.
  \end{equation*}
  Hence, $\langle \nu' \rangle \oplus \langle \ls e_1, \dots, e_m, w \rs \rangle  \oplus \langle r(\tau_2) \rangle$ is a direct summand of $\Z^{m+n}$. It follows that $\nu' \ast r(\tau_2)$ is a standard simplex in $\Linkhat_{\BAA_n^m}^<(w)$. The fact that $\nu$ is a simplex in $\Linkhat_{\BAA_n^m}^<(w)$ means that the vertices in $\nu \setminus \nu'$ can in an appropriate way be written as sums of the vectors spanning the lines $\nu' \cup \ls e_1, \ldots, e_m, w \rs$. Therefore, $\nu \ast r(\tau_2)$ spans a simplex of the same type as $\nu$ in $\Linkhat_{\BAA_n^m}^<(w)$.
\end{proof}

\subsection{Extending over double-double simplices}

The goal of this subsection is to extend the map $$r\colon  \sd(\Linkhat_{\BA_n^m}(w)) \m \Linkhat^{<}_{\BAA_n^m}(w)$$ defined in \autoref{BAretraction} over all double-double simplices. For this, we need to study minimal double-double simplices in the sense of \autoref{minimal-simplices-and-additive-core}.

\begin{observation}
  A minimal double-double simplex $\tau_1$ in $\Linkhat_{\BAA^m_n}(w)$ is the join $\tau_{1,1} \ast \tau_{1,2}$ of two minimal $2$-additive simplices in $\Linkhat_{\BAA^m_n}(w)$. In particular, any facet of $\tau_1$ is $2$-additive.
\end{observation}

If one of the two minimal $2$-additive simplices in a double-double simplex $\sigma$ in $\Linkhat_{\BAA_n^m}(w)$ is carrying, then the set $r(\sigma)$ might or might not span a simplex in $\Linkhat_{\BAA_n^m}(w)$. This is illustrated in the next example.

\begin{example}
  Consider a minimal double-double simplex $\tau_1$ in $\Linkhat_{\BAA^m_n}(w)$ of the form
  $$\tau_1 = \{v_0, v_1, v_2 = \langle \bar v_0 + \bar v_1 \rangle, v_3, \langle \bar v_3 + \epsilon \cdot \bar w \rangle\}$$ for $\epsilon \in \{+1,-1\}$. Assume that $\{v_0, v_1, v_2 = \langle \bar v_0 + \bar v_1 \rangle\}$ is carrying. If $\epsilon = +1$, then \autoref{lem_properties_r} implies that  $r(\tau_1) = \{r(v_0), r(v_1), \langle \overline{r(v_0)} + \overline{r(v_1)} - \bar w \rangle, r(v_3)\}$, which spans a $w$-related $3$-additive simplex in $\Linkhat_{\BAA^m_n}(w)$. If $\epsilon = -1$ and the last coordinate of $\bar v_3$ is contained in $(0, R)$, then \autoref{lem_properties_r} implies that $r(\tau_1) = \{r(v_0), r(v_1), \langle \overline{r(v_0)} + \overline{r(v_1)} - \bar w \rangle, r(v_3), \langle \bar w - \overline{r(v_3)} \rangle \}$, which does not define a simplex in $\Linkhat_{\BAA^m_n}(w)$.
\end{example}

Because we decided to construct the retraction maps $r$ for $\BAA^m_n$ occurring in \autoref{retraction} as extensions of the retraction maps that Church--Putman defined for $\BA^m_n$ (compare with \autoref{BAretraction}), we nevertheless subdivide \emph{every} minimal double-double simplex that contains a carrying $2$-additive face. This leads us to the following definition.

\begin{definition}
  Let $\sigma = \tau_1 \ast \tau_2$ be a double-double simplex of $\Linkhat_{\BAA_n^m}(w)$, where $\tau_1 = \tau_{1,1} \ast \tau_{1,2}$ is a minimal double-double simplex and $\tau_2$ is a standard simplex. Then $\sigma$ is called \emph{carrying} if one of the following equivalent conditions holds.
  \begin{enumerate}
  \item $\tau_1$ has a carrying facet.
  \item One of the two $2$-additive simplices $\tau_{1,1}$ or $\tau_{1,2}$ is carrying in the sense of \autoref{carrying-2-additive}.
  \end{enumerate}
\end{definition}

Since any carrying $2$-additive simplex has been subdivided in $\sd(\Linkhat_{\BA_n^m}(w))$, we need to subdivide every carrying double-double simplex in a compatible fashion. This is done in the next definition.

\begin{definition} \label{definition-DD}
  Let $\Linkhat_{\DD_n^m}(w)$ and $\Linkhat_{\DD_n^m}^{<}(w)$ denote the subcomplexes of $\Linkhat_{\BAA_n^m}(w)$ and $\Linkhat_{\BAA_n^m}^{<}(w)$, respectively, consisting of all simplices that are standard, $2$-additive, or of type double-double in the sense of \autoref{relative-simplex-types}. Let $\sd(\Linkhat_{\DD_n^m}(w))$ denote the coarsest subdivision of $\Linkhat_{\DD_n^m}(w)$ that contains $\sd(\Linkhat_{\BA_n^m}(w)) $ as a subcomplex.
\end{definition}

Concretely, $\sd(\Linkhat_{\DD_n^m}(w))$ in \autoref{definition-DD} is constructed as follows: Let $\sigma = \tau_1 \ast \tau_2$ be a double-double simplex of $\Linkhat_{\DD_n^m}(w)$, where $\tau_1 = \tau_{1,1} \ast \tau_{1,2}$ is a carrying minimal double-double simplex and $\tau_2$ is standard. Then, when passing from $\Linkhat_{\DD_n^m}(w)$ to $\sd(\Linkhat_{\DD_n^m}(w))$, each such simplex $\sigma$ is replaced by the simplicial join $$\sd(\sigma) = \sd(\tau_{1,1}) \ast \sd(\tau_{1,2}) \ast \tau_2,$$ where $\sd(\tau_{1,i})$ for $i \in \{1,2\}$ denotes the subdivision of the $2$-additive simplex $\tau_{1,i}$ (see \autoref{definition-sd-BA}) if it is carrying, and $\sd(\tau_{1,i}) = \tau_{1,i}$ if it is not carrying. 
Note that $\Linkhat_{\DD_n^m}^{<}(w)$ and $\sd(\Linkhat_{\BA_n^m}(w))$ are subcomplexes of $\sd(\Linkhat_{\DD_n^m}(w))$.

The main result of this subsection is the following proposition.

\begin{proposition} \label{DDretraction}
  The simplicial map constructed in \autoref{BAretraction}
  $$r\colon  \sd(\Linkhat_{\BA_n^m}(w))  \to \Linkhat^{<}_{\BAA_n^m}(w)$$
  extends to a simplicial map
  $$r\colon  \sd(\Linkhat_{\DD_n^m}(w))  \to \Linkhat^{<}_{\BAA_n^m}(w)$$
  that restricts to the inclusion $$\Linkhat^{<}_{\DD_n^m}(w) \hookrightarrow \Linkhat_{\BAA_n^m}^{<}(w)$$ on the subcomplex $\Linkhat^{<}_{\DD_n^m}(w)$ of $\sd(\Linkhat_{\DD_n^m}(w))$.
\end{proposition}

\begin{proof} Our goal is to check that $r$ is simplicial on all (possibly subdivided) double-double simplices.
  Let $\sigma = \tau_1 \ast \tau_2$ be a double-double simplex of $\Linkhat_{\DD_n^m}(w)$, where $\tau_1 = \tau_{1,1} \ast \tau_{1,2}$ is a minimal double-double simplex and $\tau_2$ is standard. We need to argue that $r$ extends over its subdivision $\sd(\sigma) = \sd(\tau_1) \ast \tau_2$. We will show that if $\alpha \subseteq \sd(\tau_1) = \sd(\tau_{1,1}) \ast \sd(\tau_{1,2})$ is a simplex, then $r(\alpha)$ is a simplex in $\Linkhat^{<}_{\BAA_n^m}(w)$. An application of \autoref{lem_key_extend_from_additive_core} for $\sigma = \tau_1 \ast \tau_2$ and $\nu = r(\alpha)$ then yields that $r(\alpha \ast \tau_2) = r(\alpha) \ast r(\tau_2)$ is a simplex of $\Linkhat_{\BAA^m_n}^{<}(w)$ as well and the claim follows. The use of \autoref{lem_key_extend_from_additive_core} is justified because the definition of $r$ on carrying $2$-additive simplices (compare with \autoref{BAretraction}) implies that $\nu = r(\alpha) \subset \langle \tau_1, e_1, \dots, e_m, w \rangle$.
  
  Firstly, assume that $\tau_1$ is not carrying. Then, $\sd(\tau_{1,1}) = \tau_{1,1}$ and $\sd(\tau_{1,2}) = \tau_{1,2}$, i.e.\ neither of the two $2$-additive simplices $\tau_{1,1}$ and $\tau_{1,2}$ is subdivided. We will show that $r$ extends over $\tau_1 = \alpha$. Because $\tau_1$ is a double-double simplex, it is impossible that $\tau_{1,1}$ and $\tau_{1,2}$ are both $w$-related or that $\tau_{1,1}$ and $\tau_{1,2}$ are both externally $2$-additive involving the same $e_j$. The following is verified in the proof of \cite[Section 4.4, Claim 2-4]{CP}. If $\tau_{1,i}$ is $\dots$
  \begin{itemize}
  \item $w$-related $2$-additive, then $r(\tau_{1,i})$ is a $w$-related $2$-additive or standard simplex,
  \item externally $2$-additive involving $e_j$, then $r(\tau_{1,i})$ is externally $2$-additive involving $e_j$,
  \item internally $2$-additive, then $r(\tau_{1,i})$ is internally $2$-additive.
  \end{itemize}
  This implies that it also is impossible that the simplices $r(\tau_{1,1})$ and $r(\tau_{1,2})$ are both $w$-related $2$-additive or that $r(\tau_{1,1})$ and $r(\tau_{1,2})$ are both externally $2$-additive involving the same $e_j$. We now compute and compare the two summands $\langle r(\tau_{1,i}) \rangle$ of $\Z^{m+n}$ obtained for $i \in \{1,2\}$. Let $\eta_{1,i} \subset \tau_{1,i}$ be a maximal standard simplex for $i \in \{1,2\}$. Then $\eta_1 = \eta_{1,1} \ast \eta_{1,2}$ is a maximal standard simplex in $\tau_1$. By \autoref{ObservationColumnOperations}, it holds that $r(\eta_1) = r(\eta_{1,1}) \ast r(\eta_{1,2})$ is a standard simplex in $\Linkhat^{<}_{\BAA_n^m}(w)$ of the same dimension as $\eta_1$. In particular, $\langle r(\eta_{1,1}) \rangle \oplus \langle r(\eta_{1,2}) \rangle \oplus w \oplus \langle \bar e_1, \dots, \bar e_m \rangle$ is a direct summand of $\Z^{m+n}$. The summand $\langle r(\tau_{1,i}) \rangle$ is equal to $\dots$
  \begin{itemize}
  \item $\langle r(\eta_{1,i}) \rangle \oplus w$ if $r(\tau_{1,i})$ is $w$-related $2$-additive,
  \item $\langle r(\eta_{1,i}) \rangle \oplus e_j$ if $r(\tau_{1,i})$ is externally additive involving $e_j$, and
  \item $\langle r(\eta_{1,i}) \rangle$ if $r(\tau_{1,i})$ is standard or internally $2$-additive.
  \end{itemize}
  Hence, the previous conclusion implies that the two summands $\langle r(\tau_{1,1}) \rangle$ and $\langle r(\tau_{1,2}) \rangle$ of $\Z^{m+n}$ intersect trivially. Since at least one of the two simplices $r(\tau_{1,1})$ and $r(\tau_{1,2})$ is $2$-additive and the other one is either a standard simplex or $2$-additive as well, we conclude that $r(\tau_1) = r(\tau_{1,1}) \ast r(\tau_{1,2})$ spans a $2$-additive or double-double simplex.
  
  Secondly, assume that $\tau_1$ is carrying such that $\tau_{1,1}$ is a carrying $2$-additive simplex and $\tau_{1,2}$ is carrying or not. Then, $\tau_{1,1} = \{v_0, v_1, v_2 = \langle \bar v_0 + \bar v_1 \rangle\}$ and $r(\sd(\tau_{1,1}))$ consists of the following three simplices where we write $\{l,l'\} = \{0,1\}$,
  \begin{itemize}
  \item $\{r(v_l), r(v_{l'}), r(t(\tau_{1,1})) = \langle \overline{r(v_l)} - \bar w \rangle\}$, which is $w$-related $2$-additive,
  \item $\{r(v_l), r(t(\tau_{1,1})) = \langle \overline{r(v_l)} - \bar w \rangle, r(v_2) = \langle \overline{r(v_l)} + \overline{r(v_{l'})} - \bar w \rangle\}$, which is $w$-related $2$-additive,
  \item $\{r(t(\tau_{1,1})) = \langle \overline{r(v_l)} - \bar w \rangle, r(v_{l'}), r(v_2) = \langle \overline{r(v_l)} + \overline{r(v_{l'})} - \bar w \rangle\}$,which is internally $2$-additive.
  \end{itemize}
  Let $\alpha_1 \subset \sd(\tau_{1,1})$ and $\alpha_2 \subset \sd(\tau_{1,2})$ be simplices of maximal dimension. We show that $r$ extends over $\alpha = \alpha_1 \ast \alpha_2$. If it is not the case that both $r(\alpha_1)$ and $r(\alpha_2)$ are $w$-related $2$-additive, we can argue as in the first part to see that the two summands $\langle r(\alpha_1) \rangle$ and $\langle r(\alpha_2) \rangle$ of $\Z^{m+n}$ intersect trivially and conclude that $r(\alpha) = r(\alpha_1) \ast r(\alpha_2)$ spans a $2$-additive or double-double simplex. If both $r(\alpha_1)$ and $r(\alpha_2)$ are $w$-related $2$-additive, then they are of the form $\{v, \langle \bar v \pm \bar w \rangle, v' \}$ and $\{u, \langle \bar{u} \pm \bar w \rangle, u'\}$ where $\{v, v', u, u'\}$ is a standard simplex in $\Linkhat_{\BAA^m_n}^<(w)$ or of the form $\{v, \langle \bar v \pm \bar w \rangle, v' \}$ and $\{u, \langle \bar{u} \pm \bar w \rangle\}$ where $\{v, v', u\}$ is a standard simplex in $\Linkhat_{\BAA^m_n}^<(w)$. In both cases it follows that $r(\alpha_1 \ast \alpha_2) = r(\alpha_1) \ast r(\alpha_2)$ is a $w$-related double-triple simplex. \qedhere
\end{proof}

\subsection{Extending over 3-additive simplices}

The goal of this subsection is to extend the map $$r \colon  \sd(\Linkhat_{\DD_n^m}(w)) \m \Linkhat^{<}_{\BAA_n^m}(w)$$ defined in the previous subsection over all $3$-additive simplices. For this, we need to study minimal $3$-additive simplices in the sense of \autoref{minimal-simplices-and-additive-core}.

\begin{observation} \label{minimal-3-additive}
  A $3$-additive simplex of $\Linkhat_{\BAA_n^m}(w)$ is minimal if all of its facets are standard.
\end{observation}

As in the $2$-additive case, studied by Church--Putman \cite{CP}, the difficulty is to extend $r$ over carrying simplices; that is $3$-additive simplices in $\Linkhat_{\BAA_n^m}(w)$ whose image under $r$ is not a simplex in the target $\Linkhat^{<}_{\BAA_n^m}(w)$.

\begin{definition}
  Let $\sigma = \tau_1 \ast \tau_2$ be a $3$-additive simplex in $\Linkhat_{\BAA_n^m}(w)$, where $\tau_1$ is a minimal $3$-additive simplex and $\tau_2$ is a standard simplex. $\sigma$ is called \emph{carrying} if the set $r(\tau_1)$ does not span a simplex of $\Linkhat^{<}_{\BAA_n^m}(w)$.
\end{definition}

As part of our discussion in this subsection, we will find the following characterisation of carrying $3$-additive simplices.

\begin{lemma} \label{carrying-3-additive-simplices}
  Let $\sigma = \tau_1 \ast \tau_2$ be a $3$-additive simplex of $\Linkhat_{\BAA_n^m}(w)$ such that $\tau_1$ is minimal $3$-additive and $\tau_2$ is a standard simplex. For any vertex $v_i = \langle \bar v_i \rangle$, write the last coordinate of $\bar{v_i}$ as $a_i R + b_i$ with $a_i \geq 0$ and $0 \leq b_i < R$. Then $\sigma$ is carrying if and only if $\tau_1$ is of one of the following two types for some $\epsilon \in \{-1,+1\}$:
  \begin{enumerate}
  \item $\tau_1 = \{v_0 , v_1, v_2 = \langle \bar v_0 + \bar v_1 + \epsilon \bar e_i \rangle\}$ is minimal externally $3$-additive and $b_0 + b_1 \not\in [0, R)$.
  \item $\tau_1 = \{v_0 , v_1, v_2, v_3 = \langle \bar v_0 + \bar v_1 + \epsilon \bar v_2 \rangle\}$ is minimal internally $3$-additive and $b_0 + b_1 + \epsilon b_2 \not\in [0, R)$.
  \end{enumerate}
\end{lemma}
This lemma follows from \autoref{TA-non-carrying}, \autoref{TA-externally} and \autoref{TA-internally}, which are proved below.

To extend the map $r$ over these carrying $3$-additive simplices, we need to subdivide them. This leads us to the definition of the complex $\sd(\Linkhat_{\TA_n^m}(w))$ that will serve as the new domain of the map $r$ when extending over $3$-additive simplices.

\begin{definition} \label{DefTA}
  Let $\Linkhat_{\TA_n^m}(w)$ and $\Linkhat_{\TA_n^m}^{<}(w)$ denote the subcomplexes of $\Linkhat_{\BAA_n^m}(w)$ and $\Linkhat_{\BAA_n^m}^{<}(w)$, respectively, consisting of all simplices that are standard, $2$-additive, double-double, or $3$-additive. Let $\sd(\Linkhat_{\TA_n^m}(w))$ denote the coarsest subdivision of $\Linkhat_{\TA_n^m}(w)$ that contains $\sd(\Linkhat_{\DD_n^m}(w))$ as a subcomplex and that subdivides every carrying minimal $3$-additive simplex $\tau_1$ by inserting a single vertex $t(\tau_1)$ at its barycentre.
\end{definition}

Using \autoref{carrying-3-additive-simplices}, this means that $\sd(\Linkhat_{\DD_n^m}(w))$ in \autoref{DefTA} is constructed as follows: Let $\sigma = \tau_1 \ast \tau_2$ be a $3$-additive simplex of $\Linkhat_{\TA_n^m}(w)$, where $\tau_1$ is a carrying minimal $3$-additive simplex and $\tau_2$ is standard. Then, when passing from $\Linkhat_{\TA_n^m}(w)$ to $\sd(\Linkhat_{\TA_n^m}(w))$, each such simplex $\sigma$ is replaced as follows.
\begin{itemize}
\item If $\tau_1 = \{v_0, v_1, v_2\}$ is a carrying minimal \emph{externally} $3$-additive simplex, we replace $\sigma$ by $\sd(\sigma) = \sd(\tau_1) \ast \tau_2$, the subdivision of $\sigma$ into the  three simplices $\{v_0, \ldots, \hat{v}_i, \ldots, v_2, t(\tau_1)\} \ast \tau_2$ for $i=0,1,2$. 
\item If $\tau_1 = \{v_0, v_1, v_2, v_3\}$ is a carrying minimal \emph{internally} $3$-additive simplex, we replace $\sigma$ by $\sd(\sigma) = \sd(\tau_1) \ast \tau_2$, the subdivision of $\sigma$ into the  four simplices $\{v_0, \ldots, \hat{v}_i, \ldots, v_3, t(\tau_1)\} \ast \tau_2$ for $i=0,1,2,3$. 
\end{itemize}
In addition, the subdivisions described in \autoref{definition-DD} are performed on the subcomplex $\Linkhat_{\DD_n^m}(w)$ of $\Linkhat_{\TA_n^m}(w)$. Note that $\sd(\Linkhat_{\DD_n^m}(w))$ and $\Linkhat_{\TA_n^m}^{<}(w)$ are subcomplexes of $\sd(\Linkhat_{\TA_n^m}(w))$.

The main result of this subsection is the following proposition.

\begin{proposition} \label{TAretraction}
  The simplicial map constructed in \autoref{DDretraction}
  $$r\colon  \sd(\Linkhat_{\DD_n^m}(w))  \m \Linkhat^{<}_{\BAA_n^m}(w)$$
  extends to a simplicial map
  $$r\colon  \sd(\Linkhat_{\TA_n^m}(w))  \m \Linkhat^{<}_{\BAA_n^m}(w)$$
  that restricts to the inclusion $$\Linkhat^{<}_{\TA_n^m}(w) \hookrightarrow \Linkhat_{\BAA_n^m}^{<}(w)$$ on the subcomplex $\Linkhat^{<}_{\TA_n^m}(w)$ of $\sd(\Linkhat_{\TA_n^m}(w))$.
\end{proposition}

The proof of this proposition and the definition of the extension of $r$ is split into several lemmas, which we present below. 
We start by proving that $r$ extends over all $3$-additive simplices that cannot be possibly carrying (compare \autoref{carrying-3-additive-simplices}).

\begin{lemma} \label{TA-non-carrying}
  The map $r$ in \autoref{TAretraction} extends over all $3$-additive simplices $\sigma = \tau_1 \ast \tau_2$ of $\Linkhat_{\BAA_n^m}(w)$, where $\tau_2$ is a standard simplex and $\tau_1$ is a minimal $3$-additive simplex that is not internally $3$-additive or externally $3$-additive of dimension two. In these cases, the set $r(\sigma) = r(\tau_1) \ast r(\tau_2)$ is a simplex of $\Linkhat^{<}_{\BAA_n^m}(w)$, so in particular, $\sigma$ is not carrying.
\end{lemma}

\begin{proof}
  Any minimal $3$-additive simplex $\tau_1 = \{v_0, \dots, v_{\dim(\tau_1)}\}$ satisfies $1 \leq \dim(\tau_1) \leq 3$. The underlying simplex of $\tau_1$ in $\BAA_{m+n}$ 
  is a subset $\{v_0, v_1, v_2, v_3\} \subseteq \{e_1, \dots, e_m, w, v_0, v_1, \dots, v_{\dim( \tau_1)}\}$, where $\{e_1, \dots, e_m, w, v_1, \dots, v_{\dim (\tau_1)}\}$ is a standard simplex and $v_0 = \langle \bar v_1 + \epsilon_2 \bar v_2 + \epsilon_3 \bar v_3 \rangle$ for some choice of signs $\epsilon_2, \epsilon_3 \in \{-1,+1\}$. We consider the possible minimal $3$-additive simplices $\tau_1$, one after the other, to prove this lemma.
  
  Firstly, assume that $\dim (\tau_1) = 1$ and write $\tau_1 = \{v_0, v_1\}$. Then there are two cases.
  
  \emph{Case (a)}: If $v_2 = e_i, v_3 = e_j$ for some $1 \leq i \neq j \leq m$, then  $ v_0 = \langle \bar v_1 + \epsilon_2 \bar e_i + \epsilon_3 \bar e_j \rangle$. Hence, $r(v_0)= \langle \overline{ r(v_1) } + \epsilon_2 \bar e_i + \epsilon_3 \bar e_j \rangle$  by \autoref{lem_properties_r} and it follows that $r(\tau_1)$ is a $3$-additive edge in $\Linkhat^{<}_{\BAA_n^m}(w)$ as well. Hence by \autoref{lem_key_extend_from_additive_core}, $r(\sigma) = r(\tau_1) \ast r(\tau_2)$ is a $3$-additive simplex.
  
  \emph{Case (b)}: If $v_2 = w, v_3 = e_i$ for some $1 \leq i \leq m$, then $v_0  = \langle \bar v_1 + \epsilon_2 \bar w + \epsilon_3 \bar e_i \rangle$ and, by \autoref{lem_properties_r}, it holds that $r(v_0) = \langle \epsilon \overline{ r(\langle \bar v_1 + \epsilon_2 \bar w \rangle) } + \epsilon_3 \bar e_i \rangle$ where $\epsilon = -1$ if the last coordinate of $\bar v_1 + \epsilon_2 \bar w$ is negative, and $\epsilon = +1$ otherwise. By \autoref{lem_properties_r} we furthermore have that $r(\langle \bar v_1 + \epsilon_2 \bar w \rangle) \in \{ r(v_1), \langle \bar w - \overline{ r(v_1) } \rangle \}$. Note that $r(\langle \bar v_1 + \epsilon_2 \bar w \rangle) = \langle \bar w - \overline{ r(v_1) } \rangle$ requires that $\epsilon_2 = -1$ and that $\epsilon = -1$. Resolving the signs, it follows that $r(\tau_1) = \{r(v_0), r(v_1)\}$ is either an externally 2-additive edge $\{\langle \overline{r(v_1)} + \epsilon_3 \bar e_i \rangle, r(v_1) \}$ or an externally $w$-related 3-additive edge $\{\langle \overline{r(v_1)} - \bar w  + \epsilon_3 \bar e_i \rangle, r(v_1)\}$. Hence by \autoref{lem_key_extend_from_additive_core}, $r(\sigma) = r(\tau_1) \ast r(\tau_2)$ is a $2$-additive or $3$-additive simplex in $\Linkhat^{<}_{\BAA_n^m}(w)$.

  Secondly, assume that $\dim (\tau_1) = 2$ and write $\tau_1 = \{v_0, v_1, v_2\}$. Assume further that $\tau_1$ is $w$-related, i.e.~$v_3 = w$. Then, $$r(v_0) = r(\langle \bar v_1 + \epsilon_2 \bar v_2 + \epsilon_3 \bar w \rangle)  \in \{ r(\langle \bar v_1 + \epsilon_2 \bar v_2 \rangle), \langle \bar w - \overline{ r(\langle \bar v_1 + \epsilon_2 \bar v_2 \rangle)} \rangle \}$$ by \autoref{lem_properties_r}. Note that the value of $r(v_0)$ depends on the last coordinate of $\bar v_1 + \epsilon_2 \bar v_2$, which might be negative, and the sign $\epsilon_3$ (compare \autoref{lem_properties_r}).
  There are different cases that can occur, depending on how $\overline{ r(\langle \bar v_1 + \epsilon_2 \bar v_2 \rangle)}$  compares to $r(v_1)$ and $r(v_2)$. We use the internally $2$-additive simplex $\{v_1, v_2, v_0' = \langle \bar v_1 + \epsilon_2 \bar v_2 \rangle\}$ to list these cases.
  
  \emph{Case (a)}: If $\{v_1, v_2, v_0' = \langle \bar v_1 + \epsilon_2 \bar v_2 \rangle\}$ is not carrying, then its image under $r$ is internally $2$-additive \cite[Section 4.4, Claim 4]{CP}. It follows that $r(\tau_1)$ is internally $2$-additive if $r(v_0) =  r(v_0')$, and $w$-related $3$-additive if $ r(v_0) =  \langle \bar w - \overline{r(v_0')} \rangle$. Hence by \autoref{lem_key_extend_from_additive_core}, $r(\sigma) = r(\tau_1) \ast r(\tau_2)$ is a $2$-additive or $3$-additive simplex.
  
  \emph{Case (b)}: If $\{v_1, v_2, v_0' = \langle \bar v_1 + \epsilon_2 \bar v_2 \rangle\}$ is carrying, then it contains a unique vertex whose coordinate is maximal in absolute value (see \autoref{carrying-2-additive} et seq.). We consider two subcases.
  
  \emph{Case (b.1)}: If the absolute value of the last coordinate of $v_1$ (or similarly $v_2$) is maximal among $\{v_1, v_2, v_0'\}$, then $\overline {r(v_1)} = \overline{ r(v_2) } + \overline { r(v_0') } - \bar w$ (see \autoref{carrying-2-additive} et seq.). In the case where $ r(v_0) =  r(v_0')$, it follows that $f(\tau_1)$ is $w$-related $3$-additive. In the case $r(v_0) = \langle \bar w - \overline{r(v_0')} \rangle$, it follows that $f(\tau_1)$ is $2$-additive. Hence by \autoref{lem_key_extend_from_additive_core}, $r(\sigma) = r(\tau_1) \ast r(\tau_2)$ is a $3$-additive or $2$-additive simplex.
  
  \emph{Case (b.2)}: If on the other hand, the absolute value of the last coordinate of $v_0'$ is maximal among $\{v_1, v_2, v_0'\}$, then we must have $\epsilon_2 = +1$ and $\overline {r(v_0')} = \overline{ r(v_1) } + \overline{ r(v_2) } - \bar w$. In the case where $ r(v_0) = r(v_0')$, it follows that $f(\tau_1)$ is $w$-related $3$-additive. Hence by \autoref{lem_key_extend_from_additive_core}, $r(\sigma) = r(\tau_1) \ast r(\tau_2)$ is a $3$-additive simplex. The case where $r(v_0) = \langle \bar w - \overline{ r(v_0')} \rangle$ cannot occur, because this only happens if the last coordinate of $\overline{ v_0'} = \bar v_1 + \bar v_2$ in $(0, R)$ (see \autoref{lem_properties_r}), which is impossible under the assumption that the last coordinate of $\overline{ v_0'} $ is the maximum of the carrying simplex $\{v_1, v_2, v_0'\}$.

  Thirdly and lastly, the remaining two possibilities are those where $\tau_1$ has dimension two and is externally $3$-additive and the one where $\tau_1$ has dimension three, which is equivalently to it being internally 3-additive. These are the cases we excluded in this lemma.
\end{proof}

We now deal with minimal externally $3$-additive simplices of type $\tau_1 = \{v_0 = \langle \bar v_1 \pm \bar v_2 \pm \bar e_i \rangle, v_1, v_2\}$.

\begin{lemma} \label{TA-externally}
  The map $r$ in \autoref{TAretraction} extends over all externally $3$-additive simplices $\sigma = \tau_1 \ast \tau_2$ of $\Linkhat_{\BAA_n^m}(w)$, where $\tau_2$ is a standard simplex and $\tau_1$ is minimal externally $3$-additive of dimension two.
\end{lemma} 

More precisely, in the proof of \autoref{TA-externally} we check that the map
$$r\colon  \sd(\Linkhat_{\DD_n^m}(w) ) \m \Linkhat^{<}_{\BAA_n^m}(w)$$
in \autoref{TAretraction} extends over the simplex $\sigma = \tau_1 \ast \tau_2$ if it is not carrying, and over the subdivision $\sd (\sigma)$ described in \autoref{DefTA} if it is carrying. The carrying case occurs if and only if $\tau_1$ is as case $i)$ of \autoref{carrying-3-additive-simplices}; we then define $r(t(\tau_1)) = r(\langle \bar v_1 + \bar v_2 \rangle) = \langle \overline {r(v_1)} + \overline{r(v_2)} - \bar w \rangle$, where $v_1,v_2 \in \tau_1$ are the two unique vertices whose last coordinate is not maximal in absolute value.

\begin{proof}
	There is an ordering of the vertices of $\tau_1$ such that $\tau_1 = \ls v_0, v_1, v_2 \rs$, where
\begin{equation*}
	\bar v_0 = \bar v_1 +\bar v_2 \pm \bar e_i
\end{equation*}	
for some  $1 \leq i \leq m$, an appropriate choice of sign and where $v_0$ is a (possibly not unique) vertex of $\tau_1$ whose last coordinate is maximal in absolute value.
	
Let $R_i$ denote the last coordinate of $\bar v_i$ and write $R_i=R a_i+b_i$ with $b_i \in [0,R)$. Note that $R_0 = R_1+R_2$. There are two cases: either $b_0 = b_1+b_2 \in [0, R)$ or $b_0 + R = b_1 + b_2 \not\in [0, R)$. In the following we use that $r(v_i)=\langle \bar v_i -a_i \bar w\rangle$ and $r(\langle \bar u \pm \bar e_i \rangle)=\langle \overline{r(u)} \pm \bar e_i \rangle$ for all lines $u$.
  
Firstly, assume that $b_0=b_1+b_2 \in [0, R)$. Then \[ r(\tau_1) = \{ \langle \overline{r(v_1)} + \overline{r(v_2)} \pm \bar e_i \rangle, r(v_1) , r(v_2) \} \] forms an externally $3$-additive simplex. Hence it follows from \autoref{lem_key_extend_from_additive_core} that $r(\sigma) = r(\tau_1) \ast r(\tau_2)$ is a $3$-additive simplex in $\Linkhat^{<}_{\BAA_n^m}(w)$.
  
Secondly, assume that $b_0+ R = b_1 + b_2 \not\in [0, R)$. Then \[ \overline{r(v_0)}  =  \langle (\bar v_1 - a_1 \bar w ) + ( \bar v_2 - a_2 \bar w ) \pm \bar e_i -\bar w \rangle =  \langle \overline{r(v_1)}+\overline{r(v_2)} \pm e_i- \bar w \rangle.\] and $r(\tau_1)$ does not form a simplex in $\Linkhat^{<}_{\BAA_n^m}(w)$. Hence, $\tau_1$ is a carrying minimal externally $3$-additive simplex (compare \autoref{carrying-3-additive-simplices}) and, in $\sd(\Linkhat_{\TA_n^m}(w))$, the simplex $\sigma = \tau_1 \ast \tau_2$ has been subdivided as $\sd(\sigma) = \sd(\tau_1) \ast \tau_2$ into three simplices (compare \autoref{DefTA}) $$\alpha_i \ast \tau_2 = \{v_0, \ldots, \hat{v}_i, \ldots, v_2, t(\tau_1)\} \ast \tau_2 \quad \text{for $i=0,1,2$.}$$
Observe that $b_0+R = b_1 + b_2 \in [R, 2R)$ implies that all $\bar v_i$ have nonzero last coordinate and hence that $v_0$ is the unique vertex in $\tau_1$ whose last coordinate is maximal in absolute value. To see that $r$ extends over $\sd(\sigma)$ by defining $r(t(\tau_1)) = r(\langle \bar v_1 + \bar v_2 \rangle) = \langle \overline {r(v_1)} + \overline {r(v_2)} - \bar w \rangle$, we first observe that the three sets $r(\alpha_i)$ span $2$-simplices in $\Linkhat^{<}_{\BAA_n^m}(w)$. Indeed,
\begin{itemize}
\item $r(\alpha_2) = \{{\color{cyan} \langle \overline{r(v_1)}+ \overline{r(v_2)} \pm \bar e_i-\bar w \rangle}, r(v_1), {\color{cyan} \langle \overline {r(v_1)} + \overline {r(v_2)} - \bar w \rangle } \}$ is externally $2$-additive\footnote{For better readability, we highlight the vertices that are contained in the additive core of the simplex in light blue.},
\item $r(\alpha_1) = \{{\color{cyan} \langle \overline{r(v_1)}+ \overline{r(v_2)} \pm \bar e_i-\bar w \rangle}, r(v_2), {\color{cyan} \langle \overline {r(v_1)} + \overline {r(v_2)} - \bar w \rangle } \}$ is externally $2$-additive, and
\item $r(\alpha_0) = \{{\color{cyan} r(v_1) ,r(v_2), \langle \overline {r(v_1)} + \overline {r(v_2)} - \bar w \rangle} \}$ is a $w$-related $3$-additive in $\Linkhat^{<}_{\BAA_n^m}(w)$.
\end{itemize}
Then, we invoke \autoref{lem_key_extend_from_additive_core} for $\sigma = \tau_1 \ast \tau_2$ to conclude that $r(\alpha_i \ast \tau_2) = r(\alpha_i) \ast r(\tau_2)$ spans a simplex of the same type. \qedhere
\end{proof}

We are left with proving that we can extend over internally $3$-additive simplices. This is done in the next lemma, whose proof also yields a description of the possible values that the vertices of a carrying internally $3$-additive simplex can take under $r$.

\begin{lemma} \label{TA-internally}
  The map $r$ in \autoref{TAretraction} extends over all internally $3$-additive simplices $\sigma$ of $\Linkhat_{\BAA_n^m}(w)$.
\end{lemma}

More precisely, in \autoref{TA-internally} we consider an internally $3$-additive simplex $\sigma = \tau_1 \ast \tau_2$ of $\Linkhat_{\BAA_n^m}(w)$ such that $\tau_1 = \{v_0,v_1,v_2, v_3\}$ is a minimal internally $3$-additive simplex and $\tau_2$ is a standard simplex. We may assume that the last coordinate of $\bar{v}_3$ is maximal (perhaps not uniquely). Then, the proof of \autoref{TA-internally} establishes the following sequence of claims: Possibly after reordering, we have that
  $$\bar{v}_3 = \bar v_0 + \bar v_1 + \bar v_2 \text{ or } \bar{v}_3 = \bar v_0 + \bar v_1 - \bar v_2.$$
  Letting $R_i = a_i R + b_i$ for $a_i \geq 0$ and $b_i \in [0,R)$ denote the last coordinate of $\bar v_i$, one of the following is true:
  \begin{enumerate}
  \begin{minipage}{0.45\linewidth}  
    \item $\overline{ r(v_3)} = \overline{ r(v_0)} + \overline{ r(v_1)} + \overline{ r(v_2)} - 2 \bar w$ \\and $b_0+b_1 \in [R,2R)$,
    \item $\overline{ r(v_3)} = \overline{ r(v_0)} + \overline{ r(v_1)} + \overline{ r(v_2)} - \bar w$,
    \item $\overline{ r(v_3)} = \overline{ r(v_0)} + \overline{ r(v_1)} + \overline{ r(v_2)}$,
    \end{minipage}
     \begin{minipage}{0.45\linewidth}  
    \item $\overline{ r(v_3)} = \overline{ r(v_0)} + \overline{ r(v_1)} - \overline{ r(v_2)} - \bar w$,
    \item $\overline{ r(v_3)} = \overline{ r(v_0)} + \overline{ r(v_1)} - \overline{ r(v_2)}$, or
    \item $\overline{ r(v_3)} = \overline{ r(v_0)} + \overline{ r(v_1)} - \overline{ r(v_2)} + \bar w$ \\ and $b_0+b_1 \in [0,R)$.
    \end{minipage}
  \end{enumerate}
In case $iii)$ and case $v)$, it holds that $b_0+b_1\pm b_2 \in [0, R)$,\footnote{Here and in the following sentence, ``$\pm$'' is to be understood as ``the same sign as the one in front of $\overline{ r(v_2)}$''.} that the set $r(\tau_1)$ forms a simplex in $\Linkhat^<_{\BAA_n^m}(w)$ and that $r$ extends over the simplex $\sigma = \tau_1 \ast \tau_2$. In all other cases, it holds that $b_0+b_1 \pm b_2 \not\in [0, R)$, that $r(\tau_1)$ is not a simplex (i.e.~$\sigma$ is carrying) and that $r$ extends over the subdivision $\sd(\sigma)$ by defining $r(t(\tau_1)) = r(\langle \bar v_0 + \bar v_1 \rangle)$. Here, $r(t(\tau_1))$ is equal to $\langle \overline{r(v_0)} + \overline{r(v_1)} - \bar w \rangle$ if $b_0+b_1 \in [R, 2R)$ or $\langle \overline{r(v_0)} + \overline{r(v_1)}\rangle$ if $b_0+b_1 \in [0, R)$. In particular, the definition of $r(t(\tau_1))$ depends on a choice of $v_0$ and $v_1$ as above.

\begin{proof}
  Firstly, assume that $v_i$ has last coordinate zero for some $0 \leq i \leq 3$. Possibly after reordering we may assume that $\bar v_2$ has last coordinate zero and that $\bar v_3$ has maximal last coordinate (perhaps not uniquely). It follows that $\bar v_3 = \bar v_0 + \bar v_1 \pm \bar v_2$ and \autoref{lem_properties_r} implies that $\overline{ r(v_3) } = \langle \overline {r(\bar v_0 + \bar v_1 \rangle)} \pm \bar v_2 \rangle$. There are two subcases.
  \begin{enumerate}[(1)]
  \item If $b_0 + b_1 = b_0 + b_1 \pm b_2 \in [0,R)$, then $\overline {r(v_3)} = \overline {r(v_0)} + \overline {r(v_1)} \pm \overline {r(v_2)}$. It follows that the set $r(\tau_1)$ is an internally $3$-additive simplex in $\Linkhat^{<}_{\BAA_n^m}(w)$ and hence by \autoref{lem_key_extend_from_additive_core} that $r(\sigma) = r(\tau_1) \ast r(\tau_2)$ is a $3$-additive simplex as well.
  \item If $b_0 + b_1 = b_0 + b_1 \pm b_2 \in [R,2R)$, then $\overline {r(v_3)} = \overline {r(v_0)} + \overline {r(v_1)} \pm \overline {r(v_2)} - \overline w$ and $r(\tau_1)$ is not a simplex in $\Linkhat^{<}_{\BAA_n^m}(w)$. At the end of this proof we will discuss how $r$ can be extended over $\sd(\sigma) = \sd(\tau_1) \ast \tau_2$ in this case.
  \end{enumerate}

  Secondly, assume that the last coordinate of all $\bar v_i$ for $0 \leq i \leq 3$ is nonzero. Let us assume that $\bar v_3$ has maximal last coordinate (possibly after reordering and perhaps not uniquely). By the definition of $3$-additive, it follows that $\bar{v}_3 = \pm \bar v_0 \pm \bar v_1 \pm \bar v_2$ for some choice of signs. Then, there are three cases up to reordering
  \begin{enumerate}[a)]
  \item \label{item_no_minus}$\bar{v}_3 = \bar v_0 + \bar v_1 + \bar v_2$ (no minus signs)
  \item \label{item_one_minus}$\bar v_3 = \bar v_0 + \bar v_1 - \bar v_2$ (one minus sign) or
  \item $\bar v_3 = \bar v_0 - \bar v_1 - \bar v_2$  (two minus signs).
  \end{enumerate}
  The last case cannot occur because then the last coordinate of $\bar v_0$ is bigger than the last coordinate of $\bar v_3$ (violating the assumption that the last coordinate of $\overline v_3$ is maximal). It follows that either $$\bar{v}_3 = \bar v_0 + \bar v_1 + \bar v_2 \text{ or } \bar{v}_3 = \bar v_0 + \bar v_1 - \bar v_2.$$
  Observe that $b_0+b_1+b_2 \in [0, 3R)$ and $b_0+b_1-b_2 \in (-R, 2R)$.
    \begin{enumerate}[(1)]\addtocounter{enumi}{2}
    \item In case \ref{item_no_minus}) and if $b_0+b_1+b_2 \in [2R, 3R)$, it follows that $\overline{ r(v_3)} = \overline{ r(v_0)} + \overline{ r(v_1)} + \overline{ r(v_2)} - 2 w$ and that $r(\tau_1)$ is not a simplex in $\Linkhat^{<}_{\BAA_n^m}(w)$. Observe that we must have $b_0 + b_1 \in [R,2R)$ in this case.
    \item In case \ref{item_no_minus}) and if $b_0+b_1+b_2 \in [R, 2R)$, it follows that $\overline{ r(v_3)} = \overline{ r(v_0)} + \overline{ r(v_1)} + \overline{ r(v_2)} - w$ and that $r(\tau_1)$ is not a simplex in $\Linkhat^{<}_{\BAA_n^m}(w)$.
    \item In case \ref{item_no_minus}) and if $b_0+b_1+b_2 \in [0, R)$, it follows that $\overline{ r(v_3)} = \overline{ r(v_0)} + \overline{ r(v_1)} + \overline{ r(v_2)}$, that $r(\tau_1)$ is an internally $3$-additive simplex and hence by \autoref{lem_key_extend_from_additive_core} that $r(\sigma) = r(\tau_1) \ast r(\tau_2)$ is a $3$-additive simplex in $\Linkhat^{<}_{\BAA_n^m}(w)$ as well.
    \item In case \ref{item_one_minus}) and if $b_0+b_1-b_2 \in [R, 2R)$, it follows that $\overline{ r(v_3)} = \overline{ r(v_0)} + \overline{ r(v_1)} - \overline{ r(v_2)} - w$ and that $r(\tau_1)$ is not a simplex in $\Linkhat^{<}_{\BAA_n^m}(w)$.
    \item In case \ref{item_one_minus}) and if $b_0+b_1-b_2 \in [0, R)$, it follows that $\overline{ r(v_3)} = \overline{ r(v_0)} + \overline{ r(v_1)} - \overline{ r(v_2)}$, that $r(\tau_1)$ is an internally $3$-additive simplex and hence by \autoref{lem_key_extend_from_additive_core} that $r(\sigma) = r(\tau_1) \ast r(\tau_2)$ is a $3$-additive simplex in $\Linkhat^{<}_{\BAA_n^m}(w)$ as well.
    \item In case \ref{item_one_minus}) and if $b_0+b_1-b_2 \in (-R, 0)$, it follows that $\overline{ r(v_3)} = \overline{ r(v_0)} + \overline{ r(v_1)} - \overline{ r(v_2)} + w$ and that $r(\tau_1)$ is not a simplex in $\Linkhat^{<}_{\BAA_n^m}(w)$. Observe that we must have $b_0 + b_1 \in [0,R)$ in this case.
    \end{enumerate}

    This establishes the first three claims in the paragraph after \autoref{TA-internally}. To finish, we are left with proving that the map extends over $\sd(\sigma)$ whenever $\sigma = \tau_1 \ast \tau_2$ is carrying, i.e.\ in the situations (2),(3),(4),(6) and (8). In $\sd(\Linkhat_{\TA^m_n}(w))$ 
    the simplex $\sigma$ has been subdivided as $\sd(\sigma) = \sd(\tau_1) \ast \tau_2$ into four simplices
    $$\alpha_i \ast \tau_2 = \{v_0, \ldots, \hat{v}_i, \ldots, v_3, t(\tau_1)\} \ast \tau_2 \quad \text{for $i=0,1,2,3$.}$$
    To see that $r$ extends over $\alpha_i \ast \tau_2$ by defining $r(t(\tau_1)) = r(\langle \bar v_0 + \bar v_1 \rangle)$, we first note that $r(\langle \bar v_0 + \bar v_1 \rangle) \in \Linkhat^{<}_{\BAA_n^m}(w)$ by definition, and that hence all elements in the set $r(\alpha_i \ast \tau_2)$ are vertices of $\Linkhat^{<}_{\BAA_n^m}(w)$. We only need to check that they from simplices. For this, we distinguish two cases depending on whether $b_0 + b_1 \in [0,R)$ or $b_0+b_1 \in [R, 2R)$.

    \textsl{Assume that $b_0+b_1 \in [R, 2R)$}. Observe that this is always true in the situation (2) and (3), might happen in situation (4) and (6), and is impossible in situation (8) described above. By \autoref{lem_properties_r} we have that $r(\langle \overline v_0 + \overline v_1 \rangle) = \langle \overline{r(v_0)} + \overline{r(v_1)} - \bar w \rangle$. With the value of $r(v_3)$ calculated above, it follows that in $\Linkhat^{<}_{\BAA_n^m}(w)$,
    \begin{itemize}
    \item $r(\alpha_3) = \{{\color{cyan} r(v_0),r(v_1)}, r(v_2) ,{\color{cyan}\langle \overline{r(v_0)} + \overline{r(v_1)} - \bar w \rangle}\}$ is $w$-related $3$-additive,
    \item $r(\alpha_2) = \{{\color{cyan} r(v_0),r(v_1)}, r(v_3) , {\color{cyan} \langle \overline{r(v_0)} + \overline{r(v_1)} - \bar w \rangle}\}$ is $w$-related $3$-additive,
    \item $r(\alpha_1) = \{r(v_0),{\color{cyan} r(v_2), r(v_3) , \langle \overline{r(v_0)} + \overline{r(v_1)} - \bar w \rangle}\}$ is $w$-related $3$-additive in situation (3) or $2$-additive in situation (2), (4) and (6), and
    \item $r(\alpha_0) = \{r(v_1),{\color{cyan}r(v_2), r(v_3) , \langle \overline{r(v_0)} + \overline{r(v_1)} - \bar w \rangle}\}$ is $w$-related $3$-additive in situation (3) or $2$-additive in situation (2), (4) and (6).
    \end{itemize}
    Invoking \autoref{lem_key_extend_from_additive_core} for $\sigma = \tau_1 \ast \tau_2$, we conclude that $r(\alpha_i \ast \tau_2) = r(\alpha_i) \ast r(\tau_2)$ spans a simplex of the same type. Hence, we can extend over $\sd(\sigma)$ in this case.

    \textsl{Assume that $b_0+b_1 \in [0,R)$}. Observe that this is always true in the situation (8), might happen in situation (4) and (6), and is impossible in situation (2) and (3) described above. By \autoref{lem_properties_r} we have that $r(\langle \overline v_0 + \overline v_1 \rangle) = \langle \overline{r(v_0)} + \overline{r(v_1)}\rangle$. With the value of $r(v_3)$ calculated above, it follows that in $\Linkhat^{<}_{\BAA_n^m}(w)$,
    \begin{itemize}
    \item $r(\alpha_3) = \{{\color{cyan}r(v_0),r(v_1)}, r(v_2) , {\color{cyan}\langle \overline{r(v_0)} + \overline{r(v_1)}\rangle}\}$ is $2$-additive,
    \item $r(\alpha_2) = \{{\color{cyan}r(v_0),r(v_1)}, r(v_3) , {\color{cyan}\langle \overline{r(v_0)} + \overline{r(v_1)}\rangle}\}$ is $2$-additive,
    \item $r(\alpha_1) = \{r(v_0),{\color{cyan} r(v_2), r(v_3) , \langle \overline{r(v_0)} + \overline{r(v_1)}\rangle}\}$ is $w$-related $3$-additive, and
    \item $r(\alpha_0) = \{r(v_1),{\color{cyan} r(v_2), r(v_3) , \langle \overline{r(v_0)} + \overline{r(v_1)}\rangle}\}$ is $w$-related $3$-additive.
    \end{itemize}
    Invoking \autoref{lem_key_extend_from_additive_core} for $\sigma = \tau_1 \ast \tau_2$, we conclude that $r(\alpha_i \ast \tau_2) = r(\alpha_i) \ast r(\tau_2)$ spans a simplex of the same type. Hence, we can extend over $\sd(\sigma)$ in this case as well. \qedhere
\end{proof}

\autoref{TA-non-carrying}, \autoref{TA-externally} and \autoref{TA-internally} imply \autoref{TAretraction} and \autoref{carrying-3-additive-simplices}, so this concludes our discussion of $3$-additive simplices.

\subsection{Extending over double-triple simplices}

The goal of this subsection is to extend the map $$r\colon  \sd(\Linkhat_{\TA_n^m}(w)) \m \Linkhat^{<}_{\BAA_n^m}(w)$$ defined in the previous subsection over all double-triple simplices. For this, we need to study minimal double-triple simplices in the sense of \autoref{minimal-simplices-and-additive-core}.

\begin{observation} \label{minimal-double-triple}
  A double-triple simplex of $\BAA_n^m$ is \emph{minimal} if all of its facets are $2$-additive or $3$-additive.
\end{observation}

The difficulty is to extend $r$ over carrying double-triple simplices, i.e.\ double-triple simplices that have a carrying facet.

\begin{definition}
  Let $\sigma = \tau_1 \ast \tau_2$ be a double-triple simplex in $\Linkhat_{\BAA^m_n}(w)$, where $\tau_1$ is a minimal double-triple simplex and $\tau_2$ is a standard simplex. The simplex $\sigma$ is called \emph{carrying} if $\tau_1$ has a carrying facet.
\end{definition}

We use the following characterisation of carrying double-triple simplices.

\begin{lemma} \label{carrying-double-triple-simplices}
  Let $\sigma = \tau_1 \ast \tau_2$ be a double-triple simplex of $\Linkhat_{\BAA_n^m}(w)$ such that $\tau_1$ is a minimal double-triple and $\tau_2$ is a standard simplex. For any vertex $v_i = \langle \bar v_i \rangle$, write the last coordinate of $\bar{v_i}$ as $a_i R + b_i$ with $a_i \geq 0$ and $0 \leq b_i < R$. Then $\sigma$ is carrying if and only if $\tau_1$ is of one of the following types for some $\epsilon \in \{-1, +1\}$:
  \begin{itemize}
  \item $\tau_1 = \{ v_0,v_1, \langle \bar v_0 + \bar v_1 + \epsilon \bar w \rangle, \langle \bar v_0 + \bar v_1 \rangle\} $ and $b_0 + b_1 \notin [0, R)$,
  \item $\tau_1 = \{ v_0,v_1,\langle \bar v_0 + \bar v_1 + \epsilon \bar e_i \rangle, \langle \bar v_0 + \epsilon \bar e_i \rangle  \} $ for some $i \leq m$, $b_0 + b_1 \notin [0, R)$,
  \item $\tau_1 = \{ v_0,v_1,\langle \bar v_0 + \bar v_1 + \epsilon \bar e_i \rangle, \langle \bar v_0 + \bar v_1 \rangle \} $ for some $i \leq m$ and $b_0 + b_1 \notin [0, R)$, or
  \item $\tau_1 = \{ v_0,v_1, v_2, \langle \bar v_0 + \bar v_1 + \epsilon \bar v_2 \rangle, \langle \bar v_0 + \bar v_1 \rangle \} $ and $b_0 + b_1 \notin [0, R)$ or $b_0 + b_1 + \epsilon b_2 \notin [0, R)$.
  \end{itemize}
\end{lemma}
  This follows from \autoref{lem:extending-over-dt-2-dimensional-case}, \autoref{lem:extending-over-dt-w-related-case}, \autoref{lem:extending-over-dt-external-case} and \autoref{lem:extending-over-dt-internal-case}, which are proved below.

Since all carrying $2$-additive and $3$-additive simplices have been subdivided in $\sd(\Linkhat_{\TA^m_n}(w))$, we will need to subdivide every double-triple simplex in a compatible fashion. The general type of subdivision of $\Linkhat_{\BAA^m_n}(w)$ that we will be considering is described in the next definition. The construction of such a subdivision will be part of the proof of the main result of this subsection.

\begin{definition} \label{DefDT}
  Assume that for every carrying minimal double-triple simplex $\tau_1$ in $\Linkhat_{\BAA^m_n}(w)$, we are given a simplicial disc $\sd(\tau_1)$ whose boundary sphere is exactly the subcomplex $\sd(\partial \tau_1)$ of $\sd(\Linkhat_{\TA^m_n}(w))$. Let $\sd(\Linkhat_{\BAA_n^m}(w))$ denote the coarsest subdivision of $\Linkhat_{\BAA_n^m}(w)$ that contains $\sd(\Linkhat_{\TA_n^m}(w))$ as a subcomplex and that subdivides every carrying minimal double-triple simplex $\tau_1$ according to $\sd(\tau_1)$.
\end{definition}

Concretely, $\sd(\Linkhat_{\BAA_n^m}(w))$ in \autoref{DefDT} is constructed as follows: In addition to the subdivisions described in \autoref{DefTA} on the subcomplex $\Linkhat_{\TA^m_n}(w)$ of $\Linkhat_{\BAA^m_n}(w)$, we subdivide carrying double-triple simplices of $\Linkhat_{\BAA_n^m}(w)$ in the following fashion. Let $\sigma = \tau_1 \ast \tau_2$ be a double-triple simplex of $\Linkhat_{\BAA_n^m}(w)$, where $\tau_1$ is a carrying minimal double-triple simplex and $\tau_2$ is standard. Then, when passing from $\Linkhat_{\BAA_n^m}(w)$ to $\sd(\Linkhat_{\BAA_n^m}(w))$, each such simplex is replaced by the simplicial join $$\sd(\sigma) = \sd(\tau_1) \ast \tau_2$$
  where $\sd(\tau_1)$ is the simplicial disc associated to $\tau_1$ that we fixed before. Note that $\sd(\Linkhat_{\TA^m_n}(w))$ and $\Linkhat_{\BAA^m_n}^<(w)$ are subcomplexes of $\sd(\Linkhat_{\BAA^m_n}(w))$.

The main result of this subsection is the following proposition, which implies \autoref{retraction}.

\begin{proposition} \label{DTretraction}
  There exists a subdivision $\sd(\Linkhat_{\BAA_n^m}(w))$ of $\Linkhat_{\BAA_n^m}(w)$ as in \autoref{DefDT} such that the simplicial map constructed in \autoref{TAretraction}
  $$r\colon  \sd(\Linkhat_{\TA_n^m}(w))  \m \Linkhat^{<}_{\BAA_n^m}(w)$$
  extends to a simplicial map
  $$r\colon  \sd(\Linkhat_{\BAA_n^m}(w))  \m \Linkhat^{<}_{\BAA_n^m}(w)$$
  that restricts to the identity map on the subcomplex $\Linkhat_{\BAA_n^m}^{<}(w)$ of $\sd(\Linkhat_{\BAA_n^m}(w))$.
\end{proposition}

The proof of this proposition, the precise definition of $\sd(\Linkhat_{\BAA_n^m}(w))$, and the definition of the extension of $r$ is split into several lemmas, which we present below.

Recall that the extension of the simplicial map $r$ over carrying $2$-additive and $3$-additive simplices involved a subdivision as well as a choice of vertices. This is the main source of difficulty in this section. The following discussion shows that for carrying $2$-additive simplices the two possible extensions of $r$ are ``homotopic''.

\subsubsection{Different extensions of $r$ over $2$-additive simplices are ``homotopic''}

Let $\beta \in \Linkhat_{\BAA^m_n}(w)$ be a minimal carrying $2$-additive simplex. Then, \autoref{carrying-2-additive} implies that $$\beta = \{v_0 , v_1, v_2 = \langle \bar v_0 + \bar v_1 \rangle \}.$$ 
The definition of the map $r$ in \autoref{BAretraction} on the subdivision $\sd(\beta)$ of $\beta$ involves a choice $v_l \in \{v_0, v_1\}$. This choice allowed Church--Putman \cite{CP} to specify $r$ on the new vertex $t(\beta) \in \sd(\beta)$, the barycentre of $\beta$, by the formula
$$r(t(\beta)) = \langle \overline{r(v_l)} - \bar w \rangle.$$
Let us write $r_0\colon  \sd(\beta) \to \Linkhat_{\BAA^m_n}(w)$ for the map defined using $l = 0$ and $r_1\colon  \sd(\beta) \to \Linkhat_{\BAA^m_n}(w)$ for the map defined using $l = 1$. The next lemma shows that these two maps are homotopic relative to the boundary $\partial \beta$ in $\Linkhat_{\BAA^m_n}^<(w)$.

\begin{lemma} \label{lem:choice-homotopy-2-additive}
  Let $\beta = \{v_0, v_1, v_2 = \langle \bar v_0 + \bar v_1 \rangle\} \in \Linkhat_{\BAA^m_n}(w)$ be a carrying $2$-additive simplex. Then, the two maps $r_0$ and $r_1$, where $r_l$ for $l = 0,1$ is as above, are homotopic relative to $\partial \beta$ via the simplicial map $$h\colon  \Delta^2 \ast \{t_0(\beta), t_1(\beta)\} \to \Linkhat_{\BAA^m_n}^<(w),$$ where $h(t_0(\beta)) = r_0(t(\beta))$ and $h(t_1(\beta)) = r_1(t(\beta))$ (see \autoref{FigureHomotopy2AdditiveChoices}). 
\end{lemma}
\begin{figure}[h!]
\begin{center}
  \includegraphics[scale=6]{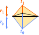} \qquad   \includegraphics[scale=6]{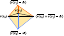}
\end{center}
\caption{ The simplices  $\Delta^2 \ast \{t_0(\beta), t_1(\beta)\}$ and their image under $h$.}
\label{FigureHomotopy2AdditiveChoices}
\end{figure} 

\begin{proof}
  Observe that the set of vertices in $$r_l(\sd(\beta)) = \{r(v_0), r(v_1), \langle \overline{r(v_0)} + \overline{r(v_1)} - \bar w \rangle, \langle \overline{r(v_l)} - \bar w \rangle\}$$ spans a double-triple simplex for each of the two choices, $l = 0$ and $l = 1$, and that the two simplices share their $3$-additive facet $\{r(v_0), r(v_1), \langle \overline{r(v_0)} + \overline{r(v_1)} - \bar w \rangle\}$. 
\end{proof}

This observation allow us to perform the following construction, which we will use later.

\begin{corollary} 
  \label{cor:changing-r-on-subdivision}
  Let $\tau_1$ be a carrying minimal double-triple simplex of $\Linkhat_{\BAA^m_n}(w)$ that contains a $2$-dimensional carrying \emph{facet} $\beta$ that is $2$-additive.
  Then,
  $$r|_{\partial \tau_1} \cup h\colon  \sd(\partial \tau_1) \cup_{\sd(\beta)} D(\beta) \to \Linkhat_{\BAA^m_n}^<(w)$$
  defines a simplicial homotopy between the two possible definitions of the map $r$ on $\sd(\partial \tau_1)$
, one obtained from $r_0$ and one obtained from $r_1$ as discussed above. Here, $D(\beta) = \Delta^2 \ast \{t_0(\beta), t_1(\beta)\}$ is the domain of the homotopy defined in \autoref{lem:choice-homotopy-2-additive}.
\end{corollary}
The homotopy in \autoref{cor:changing-r-on-subdivision} is illustrated in \autoref{FigureChangingSubdivisionType} for $\tau_1$ a 3-dimensional double-triple simplex with a carrying 2-additive face. 

\begin{figure}[h!]
\begin{center}
  \includegraphics[scale=5]{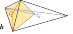}
\end{center}
\caption{The homotopy $r|_{\partial \tau_1} \cup h$}
\label{FigureChangingSubdivisionType}
\end{figure} 

We now start proving the main results of this subsection.

\subsubsection{Proof of \autoref{DTretraction} and \autoref{carrying-double-triple-simplices}}

Note that any minimal double-triple simplex $\tau_1$ in $\Linkhat_{\BAA^m_n}(w)$ satisfies $2 \leq \dim (\tau_1) \leq 4$, that any minimal double-triple simplex has a unique $3$-additive face and all other faces are $2$-additive (see \autoref{obs_list_facet_types}).

\begin{convention}
  Let $\tau_1$ be a minimal double-triple simplex $\tau_1$ in $\Linkhat_{\BAA^m_n}(w)$. In the remainder of this section, the unique $3$-additive face of $\tau_1$ will be denoted by $\gamma$. The simplex $\tau_1$ can then be written as a join $\gamma \ast z$, where $z \in \tau_1$ is a vertex. We remark that the vertex $z$ has the property that it is contained in the additive core of every $2$-additive facet of $\tau_1$.
\end{convention}

The next lemma shows that the map $r$ can be extended over minimal double-triple simplices of dimension two without any subdivisions.

\begin{lemma}
  \label{lem:extending-over-dt-2-dimensional-case}
  The map $r\colon  \Linkhat_{\TA^m_n}(w) \to \Linkhat_{\BAA^m_n}^<(w)$ in \autoref{TAretraction} extends simplicially over all double-triple simplices $\sigma = \tau_1 \ast \tau_2$ in $\Linkhat_{\BAA^m_n}(w)$, where $\tau_1$ is a minimal double-triple simplex of dimension two and $\tau_2$ is a standard simplex. I.e.\ for any such simplex it holds that $r(\sigma) = r(\tau_1) \ast r(\tau_2)$ forms a simplex in $\Linkhat_{\BAA^m_n}^<(w)$, so $\sigma$ is not carrying.
\end{lemma}

\begin{proof}
  Let $\tau_1 = \gamma \ast z$, where $\gamma = \{v_0, v_1\}$ is the unique $3$-additive facet of $\tau_1$. In this proof, we will use exactly the same strategy as in the proof of \autoref{TA-non-carrying} and consider the possible minimal $3$-additive simplices $\gamma$, one after the other. The underlying simplex of $\gamma$ in $\BAA_{m+n}$ 
  is a subset $\{v_0, v_1, v_2, v_3\} \subseteq \{e_1, \dots, e_m, w, v_0, v_1\}$, where $\{e_1, \dots, e_m, w, v_1\}$ is a standard simplex and $v_0 = \langle \bar v_1 + \epsilon_2 \bar v_2 + \epsilon_3 \bar v_3 \rangle$ for some choice of signs $\epsilon_2, \epsilon_3 \in \{-1,+1\}$. As in the first part of the proof of \autoref{TA-non-carrying}, we need to consider two cases.
  
  \emph{Case (a)}: If $v_2 = e_i$ and $v_3 = e_j$ for some $0 \leq i \neq j \leq m$, then $v_0 = \langle \bar v_1 + \epsilon_2 \bar e_i + \epsilon_3 \bar e_j \rangle$ and $ z = \langle \bar v_1  + \epsilon_2 \bar e_i \rangle$ or $z = \langle \bar v_1 + \epsilon_3 \bar e_j \rangle$. Hence, $r(v_0)= \langle \overline{ r(v_1) } + \epsilon_2 \bar e_i + \epsilon_3 \bar e_j \rangle$ and $r(z) = \langle \overline{r(v_1)} + \epsilon_2 \bar e_i \rangle$ or $r(z) = \langle \overline{r(v_1)} + \epsilon_3 \bar e_j \rangle$ by \autoref{lem_properties_r}. It follows that $r(\tau_1)$ is a double-triple simplex in $\Linkhat^{<}_{\BAA_n^m}(w)$. An application of \autoref{lem_key_extend_from_additive_core} implies that $r(\sigma) = r(\tau_1) \ast r(\tau_2)$ is a double-triple simplex as well.
  
  \emph{Case (b)}: If $v_2 = w, v_3 = e_i$ for some $1 \leq i \leq m$, then $v_0  = \langle \bar v_1 + \epsilon_2 \bar w + \epsilon_3  \bar e_i \rangle$ and  $z = \langle \bar v_1 + \epsilon_2 \bar w \rangle$ or $z = \langle \bar v_1 + \epsilon_3 \bar e_i \rangle$. By \autoref{lem_properties_r},  $r(v_0) = \langle \epsilon \overline{ r(\langle \bar v_1 + \epsilon_2 \bar w \rangle) } + \epsilon_3 \bar e_i \rangle$ where $\epsilon = -1$ if the last coordinate of $\bar v_1 + \epsilon_2 \bar w$ is negative, and $\epsilon = +1$ otherwise. \autoref{lem_properties_r} also implies that $r(\langle \bar v_1 + \epsilon_3 \bar e_i \rangle) = \langle \overline{r(v_1)} + \epsilon_3  \bar e_i \rangle$ and that $r(\langle \bar v_1 + \epsilon_2 \bar w \rangle) \in \{r(v_1), \langle \bar w - \overline{r(v_1)} \rangle \}$. Note that $r(\langle \bar v_1 + \epsilon_2 \bar w \rangle) = \langle \bar w - \overline{r(v_1)} \rangle$ requires that $\epsilon_2 = -1$ and that $\epsilon = -1$ (compare with \autoref{lem_properties_r}). Resolving the signs, it follows that if $z = \langle \bar v_1 + \epsilon_3 \bar e_i \rangle$, then 
    \begin{equation*} r(\tau_1) = \ls r(v_0),r(v_1), r(z) \rs = \qquad
      \begin{rcases}
        \langle \overline{ r(v_1) } + \epsilon_3 \bar e_i \rangle\\
        \langle \overline{r(v_1)} - \bar w  + \epsilon_3 \bar e_i \rangle 
      \end{rcases}
      \ast r(v_1) \ast \langle \overline{r(v_1)} + \epsilon_3 \bar e_i \rangle
    \end{equation*}
     is a $2$-additive or double-triple simplex. If $z = \langle \bar v_1 + \epsilon_2 \bar w \rangle$, then $r(\tau_1) = \ls r(v_0),r(v_1), r(z) \rs$ is equal to either
    $$\{\langle \overline{ r(v_1) } + \epsilon_3 \bar e_i \rangle, r(v_1), r(v_1)\},$$ which is externally $2$-additive, or
    $$\{\langle \overline{r(v_1)} - \bar w  + \epsilon_3 \bar e_i \rangle, r(v_1), \langle \bar w - \overline{r(v_1)} \rangle\},$$ which is a double-triple simplex.
    Hence $r(\tau_1)$ forms a simplex in each case. By \autoref{lem_key_extend_from_additive_core} we therefore conclude that $f(\sigma) = f(\tau_1) \ast f(\tau_2)$ is a simplex of the same type in $\Linkhat^{<}_{\BAA_n^m}(w)$. \qedhere
\end{proof}

We now work towards extending the retraction over minimal double-triple simplices that are $3$-dimensional. The unique $3$-additive facet of such simplices are either $w$-related or externally $3$-additive. We start by considering $3$-dimensional double-triple simplices whose unique $3$-additive facet is $w$-related. The next observation explains why such minimal double-triple simplices can only have one carrying facet.

\begin{observation}
  Let $\tau_1 = \gamma \ast z$ be a minimal double-triple simplex of dimension $3$ whose unique $3$-additive facet $\gamma = \{v_0  = \langle \bar v_1 + \epsilon_2 \bar v_2 + \epsilon_3 \bar w \rangle, v_1, v_2 \}$ is $w$-related, where $\epsilon_2, \epsilon_3 \in \{-1,+1\}$. If $z = \langle \epsilon_k \bar v_k + \epsilon_3 \bar w \rangle$ for $k \in \{1,2\}$ and $\epsilon_1 \coloneq +1$, then $\tau_1$ cannot have any carrying facet. If $z = \langle \bar v_1 + \epsilon_2 \bar v_2 \rangle$, then there exists a unique $2$-additive facet $\{v_1, v_2, z\}$ that might be carrying, as in \autoref{FigureDoubleTripleCarrying}. 
\end{observation}

\begin{figure}[h!]
\begin{center}
  \includegraphics[scale=3.5]{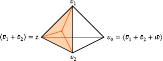}
\end{center}
\caption{The subdivision of $\partial \tau_1$ in the carrying case}
\label{FigureDoubleTripleCarrying}
\end{figure}

\begin{proof}
  Recall that $z$ is contained in the additive core of any $2$-additive facet of $\tau_1$. If $z = \langle \epsilon_k \bar v_k + \epsilon_3 \bar w \rangle$ for $k \in \{1,2\}$, then any facet of $\tau_1$ is $w$-related and it follows from \autoref{carrying-2-additive} and \autoref{carrying-3-additive-simplices} that no such simplex can be carrying. If $z = \langle \bar v_1 + \epsilon_2 \bar v_2 \rangle$, then all facets but the $2$-additive facet $\{v_1, v_2, z\}$ are $w$-related. Hence, it follows from \autoref{carrying-2-additive} and \autoref{carrying-3-additive-simplices} that $\{v_1, v_2, z\}$ is the unique possibly carrying facet.
\end{proof}

We now extend the retraction over the first type of minimal double-triple simplex of dimension $3$.

\begin{lemma}
  \label{lem:extending-over-dt-w-related-case}
  The map $r$ introduced in \autoref{TAretraction} extends over all double-triple simplices $\sigma = \tau_1 \ast \tau_2$ in $\Linkhat_{\BAA^m_n}(w)$, where $\tau_1 = \gamma \ast z$ is a minimal double-triple simplex of dimension $3$ with $w$-related $3$-additive facet $\gamma$ and $\tau_2$ is a standard simplex.
\end{lemma}

More precisely, in the proof of \autoref{lem:extending-over-dt-w-related-case} we check that the map
$$r\colon  \sd(\Linkhat_{\TA_n^m}(w) ) \m \Linkhat^{<}_{\BAA_n^m}(w)$$
in \autoref{TAretraction} extends over the simplex $\sigma = \tau_1 \ast \tau_2$ if the simplex is not carrying, and over a subdivision $\sd(\sigma) = \sd(\tau_1) \ast \tau_2$ if the simplex is carrying. Here, $\sd(\tau_1) = \sd(\eta) \ast v_0$ is the coarsest subdivision of $\tau_1$ that is compatible with the subdivision of its unique carrying $2$-additive facet $\eta = \{v_1, v_2, z = \langle \bar v_1 + \bar v_2 \rangle \}$ described in \autoref{definition-sd-BA}. The carrying case (illustrated in \autoref{FigureSubdivisionCarryingDouble-Triple}) occurs if and only if $\tau_1 = \{v_0 = \langle \bar v_1 + \bar v_2 \pm \bar w \rangle, v_1, v_2, z = \langle \bar v_1 + \bar v_2 \rangle\}$ contains a unique carrying $2$-additive facet $\{v_1, v_2, z = \langle \bar v_1 + \bar v_2 \rangle\}$.

\begin{figure}[h!]
\begin{center}
  \includegraphics[scale=3.5]{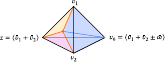}
\end{center}
\caption{ The subdivision of  $\tau_1 = \gamma \ast z$ in the carrying case }
\label{FigureSubdivisionCarryingDouble-Triple}
\end{figure} 

\begin{proof}
  Let $\tau_1 = \gamma \ast z$ and $\gamma = \{v_0, v_1, v_2\}$ with $v_0 = \langle \bar v_1 + \epsilon_2 \bar v_2 + \epsilon_3 \bar w \rangle$ for two signs $\epsilon_2, \epsilon_3 \in\{-1,+1\}$. Then, $z \in \{\langle \bar v_1 + \epsilon_2 \bar v_2 \rangle, \langle \bar v_1 + \epsilon_3 \bar w \rangle, \langle \epsilon_2 \bar v_2 + \epsilon_3 \bar w \rangle\}$. We will sometimes use the convention that $\epsilon_1 \coloneq +1$. In this proof, we use exactly the same strategy as in the proof of \autoref{TA-non-carrying}. The underlying simplex of $\gamma$ in $\BAA_{m+n}$ is the set $\{v_0, v_1, v_2, v_3 = w\}$. We need to consider three cases, which are similar to the cases in the second part of the proof of \autoref{TA-non-carrying}: We again set $v_0' = \langle \bar v_1 + \epsilon_2 \bar v_2 \rangle$ and recall that $$r(v_0) = r(\langle \bar v_1 + \epsilon_2 \bar v_2 + \epsilon_3 \bar w \rangle)  \in \{ r(v_0'), \langle \bar w - \overline{ r(v_0')} \rangle \}$$ by \autoref{lem_properties_r}. Furthermore, we note that the value of $r(v_0)$ depends on the last coordinate of $\bar v_1 + \epsilon_2 \bar v_2$, which might be negative, and the sign $\epsilon_3$ (compare with  \autoref{lem_properties_r}). There are three subcases.
  
  \emph{Case (a):} Assume that $\{v_1, v_2, v_0' = \langle \bar v_1 + \epsilon_2 \bar v_2 \rangle\}$ is not carrying and hence $r(v_0') = \langle \overline{r(v_1)} + \epsilon_2 \overline{r(v_2)} \rangle$. Using \autoref{lem_properties_r} to calculate $r(v_0)$ and $r(z)$, the possible values of $r(\tau_1) = \{r(v_0), r(v_1), r(v_2), r(z)\}$ are of the following form. If $z = \langle \bar v_1 + \epsilon_2 \bar v_2 \rangle$, then $r(\tau_1)$ is
     \begin{align*}
      \begin{rcases}
        \langle \overline{r(v_1)} + \epsilon_2 \overline{r(v_2)} \rangle\\
        \langle \overline{r(v_1)} + \epsilon_2 \overline{r(v_2)} + \epsilon_3 \bar w \rangle 
      \end{rcases}
      \ast r(v_1) \ast r(v_2) \ast \langle  \overline{r(v_1)} + \epsilon_2 \overline{r(v_2)} \rangle \text{ which is $2$-additive or double-triple.}
    \end{align*}
    Let $k \in \{1,2\}$ and $\epsilon_1 \coloneq +1$. If $z = \langle \epsilon_k \bar v_k + \epsilon_3 \bar w \rangle$ and $r(z) = r(v_k)$, then $r(\tau_1)$ is
    \begin{align*}
      \begin{rcases}
        \langle \overline{r(v_1)} + \epsilon_2 \overline{r(v_2)} \rangle\\
        \langle \overline{r(v_1)} + \epsilon_2 \overline{r(v_2)} + \epsilon_3 \bar w \rangle 
      \end{rcases}
      \ast r(v_1) \ast r(v_2) \ast r(v_k) \text{ which is $2$-additive or $3$-additive.}
    \end{align*}
    If $z = \langle \epsilon_k \bar v_k + \epsilon_3 \bar w \rangle$ and $r(z) = \langle \epsilon_k \overline{r(v_k)} + \epsilon_3 \bar w \rangle$, then $r(\tau_1) $ is
    \begin{align*}    
      \begin{rcases}
        \langle \overline{r(v_1)} + \epsilon_2 \overline{r(v_2)} \rangle\\
        \langle \overline{r(v_1)} + \epsilon_2 \overline{r(v_2)} + \epsilon_3 \bar w \rangle 
      \end{rcases}
      \ast r(v_1) \ast r(v_2) \ast \langle \epsilon_k \overline{r(v_k)} + \epsilon_3 \bar w \rangle \text{ which is double-triple.}
    \end{align*}
    It follows that $r(\tau_1)$ is a simplex in $\Linkhat^{<}_{\BAA_n^m}(w)$. Then, an application of \autoref{lem_key_extend_from_additive_core} implies that $r(\sigma) = r(\tau_1) \ast r(\tau_2)$ is a simplex in $\Linkhat^{<}_{\BAA_n^m}(w)$ as claimed.
    
  \emph{Case (b)}: Assume that $\{v_1, v_2, v_0' = \langle \bar v_1 + \epsilon_2 \bar v_2 \rangle\}$ is carrying and that $z \in \{\langle \bar v_1 + \epsilon_3 \bar w \rangle, \langle \epsilon_2 \bar v_2 + \epsilon_3 \bar w \rangle\}$. We start by recording two observations.
    \begin{itemize}
    \item If the absolute value of the last coordinate of $v_k$ is maximal among $\{v_1, v_2, v_0'\}$ for $k \in \{1,2\} = \{k,k'\}$, then one obtains the relation $\overline{r(v_k)} = \overline{r(v_{k'})} + \overline{r(v_0')} - \bar w$. Note that in this case, the last coordinate of $\bar v_k$ is contained in $[R, \infty)$ and hence that, if $\epsilon_k \neq \epsilon_3$, it is impossible that $r(\langle \epsilon_k \bar v_k + \epsilon_3 \bar w \rangle) = \langle \bar w - \overline{r(v_k)} \rangle$ (compare with \autoref{lem_properties_r}). Therefore, this case will not be considered below.
    \item If the absolute value of the last coordinate of $v_0'$ is maximal among $\{v_1, v_2, v_0'\}$, then one obtains the relation $\overline{r(v_0')} = \overline{r(v_1)} + \overline{r(v_2)} - \bar w$. As observed in case (b.2) in the proof of \autoref{TA-non-carrying}, it is impossible that $v_0'$ is maximal and $r(v_0) = \langle \bar w - \overline{r(v_0')} \rangle$. Hence, this case will not be considered below.
    \end{itemize}
    Using \autoref{lem_properties_r} to calculate $r(v_0)$ and $r(z)$, the possible values of $r(\tau_1) = \{r(v_0), r(v_1), r(v_2), r(z)\}$ are of the following form. Let $l \in \{1,2\}$. If $r(v_0) = r(v_0')$, then
    \begin{align*}
      &\begin{rcases}
        \langle \overline{ r(v_k) } - \overline{ r(v_{k'})} + \bar w  \rangle\\
        \langle \overline{ r(v_1) } + \overline{ r(v_{2}) } - \bar w \rangle 
      \end{rcases}
      \ast r(v_1) \ast r(v_2) \ast r(v_l) \text{ is $3$-additive or $2$-additive,}\\
      &\{\langle \overline{ r(v_k) } - \overline{ r(v_{k'})} + \bar w  \rangle, r(v_1), r(v_2), \langle \bar w - \overline{r(v_{k'})} \rangle \} \text{ is double-triple, or}\\
      &\{\langle \overline{ r(v_1) } + \overline{ r(v_{2})} - \bar w  \rangle, r(v_1), r(v_2), \langle \bar w - \overline{r(v_{l})} \} \rangle \} \text{ is double-triple.}
    \end{align*}
    If $r(v_0) = \langle \bar w - \overline{r(v_0')} \rangle$, then
    \begin{align*}
      \langle \overline{ r(v_k) } - \overline{ r(v_{k'})} \rangle \ast r(v_1) \ast r(v_2) \ast
      \begin{cases}
        r(v_l) \text{ is $2$-additive, or}\\
        \langle \bar w - r(v_l) \rangle \text{ is double-triple.}
      \end{cases}
    \end{align*}
    It follows that $r(\tau_1)$ is a simplex in $\Linkhat^{<}_{\BAA_n^m}(w)$. Hence \autoref{lem_key_extend_from_additive_core} implies that $r(\sigma) = r(\tau_1) \ast r(\tau_2)$ is a simplex in $\Linkhat^{<}_{\BAA_n^m}(w)$ as claimed.
    
  \emph{Case (c)}: Assume that $\eta = \{v_1, v_2, v_0' = \langle \bar v_1 + \epsilon_2 \bar v_2 \rangle\}$ is carrying and that $z = \langle \bar v_1 + \epsilon_2 \bar v_2 \rangle$. Then, the carrying facet $\eta = \{v_1,v_2,z\}$ of $\tau_1$ is subdivided into three simplices in $\sd(\Linkhat_{\TA_n^m}(w))$ and we extend this subdivision to $\tau_1$ by replacing $\tau_1$ with the simplicial join $\sd(\tau_1) = \sd(\eta) \ast v_0$. The resulting subdivision $\sd(\tau_1)$ of $\tau_1$ consists of the following three $3$-simplices
    \begin{equation*}
      \{ v_0, t(\eta)\} \ast
      \begin{cases}
        \{v_1, v_2\},\\
        \{v_1, \langle \bar v_1 + \epsilon_2 \bar v_2 \rangle \} \text{ and }\\
        \{v_2, \langle \bar v_1 + \epsilon_2 \bar v_2 \rangle \}.
      \end{cases}
    \end{equation*}
    Recall that the barycentre $t(\eta)$ is mapped to $r(t(\eta)) = \langle \overline{r(v_l)} - \bar w \rangle$ for some choice $l \in \{1, 2\}$. Note that we must have $r(v_0) = r(v_0')$ if $\eta$ is carrying, i.e.\ we can't have $r(v_0) = \langle \bar w - \overline{r(v_0)} \rangle$ (compare with \autoref{lem_properties_r}). The images of these simplices under $r$ are therefore given by the following.
    \begin{equation*}     
      \{\langle \overline{r(v_1)} + \overline{r(v_2)} - \bar w \rangle, \langle \overline{r(v_l)} - \bar w \rangle\} \ast
      \begin{cases}
        \{r(v_1), r(v_2)\} \text{ is double-triple,}\\
        \{r(v_1), \langle \overline{r(v_1)} + \overline{r(v_2)} - \bar w \rangle \} \text{ is $2$-additive and }\\
        \{r(v_2), \langle \overline{r(v_1)} + \overline{r(v_2)} - \bar w \rangle \} \text{ is $2$-additive.}
      \end{cases}
    \end{equation*}
    It follows that $r$ extends over $\sd(\tau_1)$. Hence, \autoref{lem_key_extend_from_additive_core} implies that $r$ extends over any simplex in $\sd(\sigma) = \sd(\tau_1) \ast \tau_2 = \sd(\eta) \ast v_0 \ast \tau_2$. \qedhere
\end{proof}

In the next step, we extend the retraction over all minimal double-triple simplices of dimension $3$ whose unique $3$-additive facet is externally $3$-additive. The next observation records that if such a simplex is carrying, then it has exactly two carrying facets.

\begin{observation} \label{carrying-facets-of-dt-external-gamma}
  Let $\tau_1 = \gamma \ast z$ be a minimal double-triple simplex of dimension $3$ whose unique $3$-additive facet $\gamma$ is externally $3$-additive. Assuming that $\bar v_0$ has maximal last coordinate, we get that $\gamma = \{v_0  = \langle \bar v_1 + \epsilon_2 \bar v_2 + \epsilon_3 \bar e_i \rangle, v_1, v_2\}$. If $\tau_1$ is carrying, then $\tau_1$ has exactly two carrying facets and it holds that $v_0 = \langle \bar v_1 + \bar v_2 + \epsilon_3 \bar e_i \rangle$, i.e. $\epsilon_2 = +1$. See an illustration in \autoref{FigureDoubleTriple2Carrying}. In this case, the $3$-additive facet $\gamma$ has to be carrying and the second carrying facet $\beta$ is the unique internally $2$-additive facet of $\tau_1$, which is one of the following $$\beta = \{v_0,  v_{k'}, z = \langle \bar v_k + \epsilon_3 \bar e_i \rangle \} \text{ for } \{k,k'\} = \{1,2\} \text{ or } \beta = \{v_1,  v_2, z = \langle \bar v_1 + \bar v_2 \rangle \}.$$
\end{observation}

\begin{figure}[h!]
\begin{center}
  \includegraphics[scale=3.5]{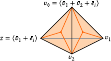}
\end{center}
\caption{The subdivision of $\partial \tau_1$ in the carrying case.}
\label{FigureDoubleTriple2Carrying}
\end{figure}

\begin{proof}
  If $\gamma$ is carrying, then it follows from \autoref{carrying-3-additive-simplices} that $\gamma = \{v_0  = \langle \bar v_1 + \bar v_2 + \epsilon_3 \bar e_i \rangle, v_1, v_2\}$ with $b_1 + b_2 \in [R, 2R)$. It follows that $z \in \{ \langle \bar v_l + \epsilon_3 \bar e_i \rangle, \langle \bar v_1 + \bar v_2 \rangle\}$ for $l \in \{1,2\}$. Recall from \autoref{carrying-2-additive} that the only carrying $2$-additive simplices are internally $2$-additive. $\tau_1$ contains a unique internally $2$-additive facet $\beta$ spanned by $\beta = \{v_0,  v_{k'}, z = \langle \bar v_k + \epsilon_3 \bar e_i \rangle \}$ for $\{k,k'\} = \{1,2\}$ or $\beta = \{v_1,  v_2, z = \langle \bar v_1 + \bar v_2 \rangle \}$. Observe that $\beta$ has to be carrying, because $b_1 + b_2 \in [R, 2R)$. If $\gamma$ is not carrying, then the unique internally $2$-additive facet $\beta$ cannot be carrying since adding or subtracting $\bar e_i$ does not change the last coordinate (compare with \autoref{lem_properties_r}).
\end{proof}

We now finish our discussion on how to extend the retraction over minimal double-triple simplices of dimension $3$. The following simplicial $3$-disc and \autoref{cor:changing-r-on-subdivision} will be used to describe the subdivision $\sd(\tau_1)$ of a carrying minimal double-triple simplex with externally $3$-additive facet.

\begin{definition} \label{subdivision-carrying-dt-external-gamma}
  Let $\sd(\partial \Delta^3)$ be subdivision of the standard simplicial $2$-sphere $\partial \Delta^3$ on the vertex set $\{w,x,y,z\}$ obtained by subdividing the facet $\gamma = \{w,x,y\}$ by placing the vertex $t(\gamma)$ at its barycentre and the facet $\beta = \{x,y,z\}$ by placing the vertex $t(\beta)$ at its barycentre. Let $\sd(\Delta^3)$ be the simplicial $3$-disc that is obtained by extending the subdivision of $\sd(\partial \Delta^3)$ to a subdivision of the $3$-simplex $\Delta^3$ using the following five $3$-simplices (shown in \autoref{FigureExtendingOver-AbstractLabels}), 
 $$\{t(\gamma), t(\beta), x,y\}, \quad 
  \{t(\gamma), t(\beta), w,x\}, \quad 
  \{t(\gamma), t(\beta), w,y\},
  \quad  \{t(\beta), w,x,z \}, \quad  \text{ and }
\quad  \{t(\beta), w,y,z \}.$$ 

\begin{figure}[h!]
\begin{center}
   \includegraphics[scale=3.5]{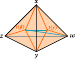}
\end{center}
\caption{The subdivision $\sd(\Delta^3)$}
\label{FigureExtendingOver-AbstractLabels}
\end{figure}
\end{definition}

\begin{lemma}
  \label{lem:extending-over-dt-external-case}
  The map $r$ in \autoref{TAretraction} extends over all double-triple simplices $\sigma = \tau_1 \ast \tau_2$ in $\Linkhat_{\BAA^m_n}(w)$, where $\tau_1 = \gamma \ast z$ is a minimal double-triple simplex of dimension $3$ with externally $3$-additive facet $\gamma$ and $\tau_2$ is a standard simplex.
\end{lemma}

More precisely, in the proof of \autoref{lem:extending-over-dt-external-case} we check that the map
$$r\colon  \sd(\Linkhat_{\TA_n^m}(w) ) \m \Linkhat^{<}_{\BAA_n^m}(w)$$
in \autoref{TAretraction} extends over the simplex $\sigma = \tau_1 \ast \tau_2$ if the simplex is not carrying, and over a subdivision $\sd (\sigma) = \sd(\tau_1) \ast \tau_2$ of $\sigma$ if the simplex is carrying. Here,
the subdivision $\sd(\tau_1)$ is of the form $\sd(\Delta^3)$ described in \autoref{subdivision-carrying-dt-external-gamma} or of the form $\sd(\Delta^3) \cup_{\sd(\beta)} D(\beta)$ using \autoref{subdivision-carrying-dt-external-gamma} and applying \autoref{cor:changing-r-on-subdivision} once to the internally $2$-additive carrying facet $\beta$ of $\tau_1$. The carrying case occurs if and only if $\tau_1 = \{ v_1,v_2,\langle \bar v_1 + \bar v_2 + \epsilon_3 \bar e_i \rangle, \langle \bar v_1 + \epsilon_3 \bar e_i \rangle  \} $ for some $i \leq m$ and $\epsilon_3 \in \{+1,-1\}$ or $\tau_1 = \{ v_1,v_2,\langle \bar v_1 + \bar v_2 + \epsilon_3 \bar e_i \rangle, \langle \bar v_1 + \bar v_2 \rangle \} $ for some $i \leq m$ and $\epsilon_3 \in \{+1,-1\}$. These cases are illustrated in \autoref{ExtendingOver}. 

\begin{figure}[h!]
\begin{center}
  \includegraphics[scale=3.5]{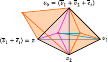} \qquad
    \includegraphics[scale=3.5]{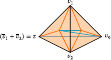}
\end{center}
\caption{The carrying cases of \autoref{lem:extending-over-dt-external-case}}
\label{ExtendingOver}
\end{figure}

\begin{proof}
  Let $\tau_1 = \gamma \ast z$ and $\gamma = \{v_0, v_1, v_2\}$ with $v_0 = \langle \bar v_1 + \epsilon_2 \bar v_2 + \epsilon_3 \bar e_i \rangle$ for $\epsilon_2, \epsilon_3 \in \{+1,-1\}$. Then, $z \in \{\langle \bar v_1 + \epsilon_2 \bar v_2 \rangle, \langle \bar v_1 + \epsilon_3 \bar e_i \rangle, \langle \epsilon_2 \bar v_2 + \epsilon_3 \bar e_i \rangle\}$.
  
  Firstly, assume that $\tau_1$ is not carrying. Then, it holds that $r(v_0) = \langle \overline{r(v_1)} + \epsilon_2 \overline{r(v_2)} + \epsilon_3 \bar e_i \rangle$ and $r(z) \in \{\langle \overline{r(v_1)} + \epsilon_2 \overline{r(v_2)} \rangle, \langle \overline{r(v_1)} + \epsilon_3 \bar e_i \rangle, \langle \epsilon_2 \overline{r(v_2)} + \epsilon_3 \bar e_i \rangle\}$ using \autoref{lem_properties_r}. It follows that
    \begin{equation*}
      \{\langle \overline{r(v_1)} + \epsilon_2 \overline{r(v_2)} + \epsilon_3 \bar e_i \rangle, r(v_1), r(v_2) \} \ast
      \begin{cases}
        \langle \overline{r(v_1)} + \epsilon_2 \overline{r(v_2)} \rangle \text{ is double-triple,}\\
        \langle \overline{r(v_1)} + \epsilon_3 \bar e_i \rangle \text{ is double-triple, and}\\
        \langle \epsilon_2 \overline{r(v_2)} + \epsilon_3 \bar e_i \rangle \text{ is double-triple.}\\
      \end{cases}
    \end{equation*}
    Hence, $r(\tau_1)$ spans a simplex in $\Linkhat_{\BAA^m_n}^<(w)$ and, by \autoref{lem_key_extend_from_additive_core}, it therefore follows that $r(\sigma) = r(\tau_1) \ast r(\tau_2)$ is a simplex in $\Linkhat^{<}_{\BAA_n^m}(w)$ as claimed.
  
  Secondly, assume that $\tau_1$ is carrying. Then, it holds by \autoref{carrying-facets-of-dt-external-gamma} that $\tau_1$ contains exactly two carrying facets and that we may assume $v_0 = \langle \bar v_1 + \bar v_2 + \epsilon_3 \bar e_i \rangle$, i.e.\ $\epsilon_2 = +1$. It follows that $r(v_0) = \langle \overline{r(v_1)} + \overline{r(v_2)} - \bar w + \epsilon_3 \bar e_i \rangle$.
  
  \emph{Case (a)}: Assume that $z = \langle \bar v_1 + \bar v_2 \rangle$. \autoref{carrying-facets-of-dt-external-gamma} implies that the two carrying facets of $\tau_1$ are the $3$-additive facet $\gamma$ and the unique internally $2$-additive facet $\beta = \{v_1, v_2, z\}$. Since $\beta$ is carrying,  we have that $\overline{r(z)} = \overline{r(v_1)} + \overline{r(v_2)} - \bar w$ (compare with \autoref{carrying-2-additive}). The two facets $\gamma = \{v_0, v_1, v_2\}$ and $\beta = \{v_1, v_2,z\}$ of $\tau_1$ have been subdivided in $\Linkhat_{\TA^m_n}(w)$. Applying \autoref{cor:changing-r-on-subdivision} once, it suffices to show that $r$ extends over the subdivision $\sd(\tau_1)$ of $\tau_1$ that extends $\sd(\partial \tau_1)$ as described in \autoref{subdivision-carrying-dt-external-gamma}, for the case $r(t(\gamma)) = \langle \overline{r(v_1)} + \overline{r(v_2)} - \bar w \rangle$ and $r(t(\beta)) = \langle \overline{r(v_1)} - \bar w \rangle$\footnote{The other choice for $r(t(\beta))$ is $\langle \overline{r(v_2)} - \bar w \rangle$.}. The following shows that the image of every simplex in $\sd(\tau_1)$ (compare with \autoref{subdivision-carrying-dt-external-gamma}) is a simplex of $\Linkhat^<_{\BAA^m_n}(w)$,
      \begin{equation*}
        \{r(t(\gamma)), r(t(\beta))\} \ast \begin{cases} 
          \{r(v_1), r(v_2)\} \text{ is double-triple,}\\
          \{r(v_0) = \langle \overline{r(t(\gamma))} + \epsilon_3 \bar e_i \rangle, r(v_1)\} \text{ is double-double,}\\
          \{r(v_0) = \langle \overline{r(t(\gamma))} + \epsilon_3 \bar e_i \rangle , r(v_2)\} \text{ is double-triple.}
        \end{cases}
      \end{equation*}
      and,
      \begin{equation*}
        \{r(t(\beta)), r(z) = \langle \overline{r(v_1)} + \overline{r(v_2)} - \bar w \rangle\} \ast
        \begin{cases}
          \{r(v_0), r(v_1)  \} \text{ is double-double,}\\
          \{r(v_0), r(v_2)\} \text{ is double-triple.}
        \end{cases}
      \end{equation*}
      It follows that the map extends the subdivision $\sd(\tau_1)$.  By \autoref{lem_key_extend_from_additive_core}, it therefore follows that $r$ extends over any simplex in $\sd(\sigma) = \sd(\tau_1) \ast \tau_2$.
  
  \emph{Case (b)}: Assume that $z = \langle  \bar v_k + \epsilon_3 \bar e_i \rangle$ for $k \in \{1,2\} = \{k, k'\}$. \autoref{carrying-facets-of-dt-external-gamma} implies that the two carrying facets of $\tau_1$ are the $3$-additive facet $\gamma$ and the unique internally $2$-additive facet $\beta = \{v_0, v_{k'}, z\}$. The two facets $\gamma = \{v_0, v_1, v_2\}$ and $\beta = \{v_0, v_{k'},z\}$ of $\tau_1$ have been subdivided in $\Linkhat_{\TA^m_n}(w)$. Applying \autoref{cor:changing-r-on-subdivision} once, it suffices to show that $r$ extends over the subdivision $\sd(\tau_1)$ of $\tau_1$ that extends $\sd(\partial \tau_1)$ as described in \autoref{subdivision-carrying-dt-external-gamma}, for the case $r(t(\gamma)) = \langle \overline{r(v_1)} + \overline{r(v_2)} - \bar w \rangle$ and $r(t(\beta)) = \langle \overline{r(v_{k'})} - \bar w \rangle$\footnote{The other choice for $r(t(\beta))$ is $\langle \overline{r(v_{k})} + \epsilon_3 \bar e_i - \bar w \rangle $.}. The following shows that the image of every simplex in $\sd(\tau_1)$ 
  is a simplex of $\Linkhat^<_{\BAA^m_n}(w)$,
      \begin{equation*}
        \{r(t(\gamma)), r(t(\beta))\} \ast \begin{cases}
          \{r(v_1), r(v_2)\} \text{ is double-triple,}\\
          \{r(v_0) = \langle \overline{r(t(\gamma))} + \epsilon_3 \bar e_i \rangle, r(v_{k'})\} \text{ is double-double,}\\
          \{r(v_0) = \langle \overline{r(t(\gamma))} + \epsilon_3 \bar e_i \rangle, r(v_k)\} \text{ is double-triple.}
        \end{cases}
      \end{equation*}
      and,
      \begin{equation*}
        \{r(t(\beta)), r(z) = \langle \overline{r(v_k)} + \epsilon_3 \bar e_i \rangle\} \ast \begin{cases}
          \{r(v_k), r(v_{k'}) \} \text{ is double-double,}\\
          \{r(v_0), r(v_k) \} \text{ is double-triple.}
        \end{cases}
      \end{equation*}
      It follows that the map extends the subdivision $\sd(\tau_1)$.  By \autoref{lem_key_extend_from_additive_core}, it therefore follows that $r$ extends over any simplex in $\sd(\sigma) = \sd(\tau_1) \ast \tau_2$. \qedhere
\end{proof}

In the final step, we prove that the retraction extends over minimal double-triple simplices of dimension $4$, i.e.\ these are internally double-triple simplices. The next observation gives a large class of examples of such double-triple simplices with have the property that every facet is carrying.

\begin{observation}
  Let $\tau_1 = \gamma \ast z$ be a double-triple simplex of dimension $4$. Then any face of $\tau_1$ can be a carrying simplex. This is for example the case for
  $$\{v_0, v_1, v_2, \langle \bar v_0 + \bar v_1 \rangle, \langle \bar v_0 + \bar v_1 + \bar v_2 \rangle\} \text{ with } b_0 + b_1 + b_2 \in [2R, 3R).$$
\end{observation}

Up to this point, the construction of the retraction involved explicit subdivisions. For this last case, the complexity is great enough that we will resort to computer calculations. In particular, we will use computers to check high connectivity of the following simplicial complexes, which will aid in our construction of the retraction.

\begin{definition}
\label{def_Q}
Let $n \geq 4$ and $\vec v_1,\vec v_2,\vec v_3,\vec w \in \Z^{m+n}$ be a partial basis such that $\{v_1,v_2,v_3,w\}$ is a simplex of $B_n^m$. Assume that the last coordinate $R$ of $\vec w$ is positive and that the last coordinates of $\vec v_1,\vec v_2,\vec v_3$ have absolute value smaller than $R$, that is, $ v_i \in \Linkhat^{<}_{\BAA_n^m}(w)$.
Let $Q_n^m(\vec v_1,\vec v_2,\vec v_3;\vec w)$ be the full subcomplex of $\Linkhat^{<}_{\BAA_n^m}(w)$ on the set of lines spanned by vectors of the form
\begin{multicols}{3} 
\begin{enumerate}
\item $\vec v_1+a_1 \vec w$,
\item $\vec v_2+a_2\vec w$,
\item $\vec v_3+a_3\vec w$,
\item $\vec v_1+\vec v_2+a_{12}\vec w$,
\item $\vec v_1+\vec v_2+\vec v_3+a_{123}\vec w$, or
\item $\vec v_1+\vec v_3+a_{13}\vec w$
\end{enumerate} 
\end{multicols}
for $a_i \in \Z$.
\end{definition}

\begin{theorem}\label{Qconn}
  The complexes $Q^m_n(\vec v_1,\vec v_2,\vec v_3;\vec w)$ are $3$-connected for all $m,n \in \N$ satisfying $m \geq 0$ and $n \geq 4$.
\end{theorem}

This theorem will be shown in \autoref{Sec3} with the help of computer calculations. We will assume it for now to deal with the last case for defining the retraction:

\begin{lemma}
  \label{lem:extending-over-dt-internal-case}
  The map $r$ in \autoref{TAretraction} extends over all double-triple simplices $\sigma = \tau_1 \ast \tau_2$ in $\Linkhat_{\BAA^m_n}(w)$ where $\tau_1$ is a minimal double-triple simplex of dimension $4$ and $\tau_2$ is a standard simplex.
\end{lemma}

More precisely, in the proof of \autoref{lem:extending-over-dt-internal-case} we check that the map
$$r\colon  \sd(\Linkhat_{\TA_n^m}(w) ) \to \Linkhat^{<}_{\BAA_n^m}(w)$$
in \autoref{TAretraction} extends over the simplex $\sigma = \tau_1 \ast \tau_2$ if the simplex is not carrying, and over a subdivision $\sd (\sigma) = \sd(\tau_1) \ast \tau_2$ of $\sigma$ if the simplex is carrying. On the subdivision $\sd(\tau_1)$ of $\tau_1$ the extension takes values in one of the complexes $Q^m_n(\vec a,\vec b,\vec c;\vec w)$ introduced in \autoref{def_Q}. The carrying case occurs if and only if $\tau_1 = \{ v_0,v_1, v_2, \langle \bar v_0 + \bar v_1 \pm \bar v_2 \rangle, \langle \bar v_0 + \bar v_1 \rangle \} $ and $b_0 + b_1 \notin [0, R)$ or $b_0 + b_1 \pm b_2 \notin [0, R)$.

\begin{proof}
  Let $\tau_1$ be a minimal internal double-triple simplex. Let $\sd(\partial \tau_1)$ be the subdivision of $\partial \tau_1$ in $\Linkhat_{\TA_n^m}(w)$. By \autoref{TAretraction} we obtain a map
  $$r\colon  \sd(\partial \tau_1) \to \Linkhat_{\BAA^m_n}^<(w).$$
  Let $\gamma \subset \tau_1$ be the unique $3$-additive facet. As in \autoref{TA-internally} et seq., we fix three lines $\{v_0, v_1, v_2\} \subseteq \gamma$ such that
  $$\gamma = \{v_0, v_1, v_2, v_3 = \langle \bar v_0 + \bar v_1 + (\epsilon_2 \bar v_2) \rangle\}$$
  for the choice of a sign $\epsilon_2 \in \{-1, 1\}$ and where the absolute value of the last coordinate of $\bar v_3$ is maximal in $\gamma$. Then,
  $\tau_1 = \gamma \ast v_4$ with 
  \begin{equation*}
    v_4 = \begin{cases}
      \langle \bar v_0 + \bar v_1 \rangle,\\
      \langle \bar v_1 + (\epsilon_2 \bar v_2) \rangle,\\
      \langle \bar v_0 + (\epsilon_2 \bar v_2) \rangle. 
    \end{cases}
  \end{equation*}

  Firstly, assume that $\tau_1$ is not carrying. Then, $\overline{r(v_3)} = \overline{r(v_0)} + \overline{r(v_1)} + \epsilon_2 \overline{r(v_2)}$ and
    \begin{equation*}
      \pm \overline{r(v_4)} = \begin{cases}
        \overline{r(v_0)} + \overline{r(v_1)},\\
        \overline{r(v_0)} + \epsilon_2 \overline{r(v_2)},\\
        \overline{r(v_1)} + \epsilon_2 \overline{r(v_2)}.
      \end{cases}
    \end{equation*}
    It follows that $r(\tau_1)$ is a double-triple simplex. By \autoref{lem_key_extend_from_additive_core}, it follows that $r(\sigma) = r(\tau_1) \ast r(\tau_2)$ is a double-triple simplex as well.
    
  Secondly, assume that $\tau_1$ is carrying. Consider the complexes $Q_{01} = Q^m_n(\overline{r(v_0)}, \overline{r(v_1)}, \epsilon_2 \cdot \overline{r(v_2)}; \bar w)$ and $Q_{10} = Q^m_n(\overline{r(v_1)}, \overline{r(v_0)}, \epsilon_2 \cdot \overline{r(v_2)}; \bar w)$ introduced in \autoref{def_Q}. We claim that $r(\sd(\partial \tau_1))$ is contained in $Q = Q_{01}$ or $Q = Q_{10}$. To see this, it suffices to check that the image of every vertex in $\sd(\partial \tau_1)$ is contained in this complex. We will explain how to choose between $Q_{01}$ and $Q_{10}$ in the first step of the proof of this claim.
  
  \emph{Step (a)}: Every vertex in $r(\partial \tau_1)$ is contained in $Q$. Indeed, observe that
      \begin{equation*}
        \pm \overline{r(v_4)} = \begin{cases}
          \overline{r(v_0)} + \overline{r(v_1)} + a_{01} \bar w,\\
          \overline{r(v_1)} + (\epsilon_2 \overline{r(v_2)}) + a_{12} \bar w,\\
          \overline{r(v_0)} + (\epsilon_2 \overline{r(v_2)}) + a_{02} \bar w,
        \end{cases}
      \end{equation*}
      for $a_{01}, a_{12}, a_{02} \in \{-1, 0, 1\}$. Hence, $r(v_4) \in Q_{10}$ in the first two cases and $r(v_4) \in Q_{01}$ in the third case. Fix this choice of $Q$. Observe furthermore that $r(v_0), r(v_1), r(v_2) \in Q$ and that $r(v_3) \in Q$, since $\overline{r(v_3)} = \overline{r(v_0)} + \overline{r(v_1)} + \epsilon_2 \overline{r(v_2)} + a_{012} \cdot \bar w$ for $a_{012} \in \{-2, -1, 0, 1\}$ by \autoref{TA-internally} et seq.
      
   \emph{Step (b)}: Assume that the unique $3$-additive facet $\gamma$ of $\tau_1$ is carrying and hence subdivided in $\sd(\partial \tau_1)$ using the new vertex $t(\gamma)$. \autoref{TA-internally} et seq.\ shows that $r(t(\gamma)) = r(\langle \bar v_0 + \bar v_1 \rangle)$ is equal to $\langle \overline{r(v_0)} + \overline{r(v_1)} - \bar w \rangle$ or $\langle \overline{r(v_0)} + \overline{r(v_1)}\rangle$. In either case, $r(t(\gamma))$ is contained in $Q = Q_{01}$ and $Q = Q_{10}$.
   
   \emph{Step (c)}: Assume that one of the $2$-additive facets $\alpha$ of $\tau_1$ is carrying and hence subdivided in $\sd(\partial \tau_1)$. Let $\eta \subset \alpha$ denote the minimal $2$-additive simplex that has been barycentrically subdivided using the new vertex $t(\eta)$. Then,
      $$r(t(\eta)) = \langle \overline{r(v_l)} - \bar w \rangle \text{ for some vertex } v_l \in \eta \subset \tau_1.$$
      By Step $(a)$ it holds that $r(v_l) \in Q$ for any vertex $v_l \in \tau_1$ and since the last coordinate of $\overline{r(v_l)} - \bar w$ is contained in $(-R,0)$, it therefore holds that
      $$r(t(\eta)) = \langle \overline{r(v_l)} - \bar w \rangle = \langle - \overline{r(v_l)} + \bar w \rangle \in Q.$$
    This completes the proof of the claim that $r(\sd(\partial \tau_1))$ is contained in $Q = Q_{01}$ or $Q = Q_{10}$.
    
    It follows that
    $$r\colon  \sd(\partial \tau_1) \to Q \hookrightarrow \Linkhat^{<}_{\BAA_n^m}(w).$$
    Recall that $\sd(\partial \tau_1)$ is a simplicial $3$-sphere. By \autoref{Qconn} the complex $Q$ is $3$-connected. It follows, that there exists a simplicial pair $(\sd(\tau_1), \sd(\partial \tau_1)) \cong (D^4, S^3)$ and a simplicial extension $r_{\tau_1}\colon  \sd(\tau_1) \to Q$ of $r\colon  \sd(\partial \tau_1) \to Q$. 
    Subdivide every simplex $\sigma = \tau_1 \ast \tau_2$ by using the coarsest simplicial structure $\sd(\sigma)$ on $\sigma$ that is compatible with the simplicial structure specified by $\sd(\tau_1)$ on $\tau_1$. I.e.\ this is defined by replacing the internal double-triple simplex $\sigma = \tau_1 \ast \tau_2$ by the collection of simplices
    $$\{\nu \ast \tau_2 \mid \nu \text{ a simplex of } \sd(\tau_1)\}.$$
    An application of \autoref{lem_key_extend_from_additive_core} for $\sigma = \tau_1 \ast \tau_2$ implies that we can extend the map
    $$ r_{\tau_1}\colon  \sd(\tau_1) \to \Linkhat_{\BAA^m_n}(w)$$
    to a map
    $$ r_{\tau_1} \ast r\colon  \sd(\sigma) = \sd(\tau_1) \ast \tau_2 \to \Linkhat_{\BAA^m_n}(w).$$
    This completes the proof.
\end{proof}

\autoref{lem:extending-over-dt-2-dimensional-case}, \autoref{lem:extending-over-dt-w-related-case}, \autoref{lem:extending-over-dt-external-case} and \autoref{lem:extending-over-dt-internal-case} imply \autoref{DTretraction} and \autoref{carrying-double-triple-simplices}, so this concludes our discussion of double-triple simplices. Furthermore, the proof of \autoref{DTretraction} completes the construction of the retraction map and establishes the main result of this section, \autoref{retraction}.

\section{High connectivity of the complexes $Q^m_n(\vec v_1,\vec v_2,\vec v_3;\vec w)$}

\label{Sec3}
The aim of this section is to prove \autoref{Qconn}, which states that the complexes $Q^m_n(\vec v_1,\vec v_2,\vec v_3;\vec w)$ introduced in \autoref{def_Q} are $3$-connected for all $m,n \in \N$. This was used to define the retraction on double-triple simplices. Throughout this section, we assume that $n\geq 4$.

To prove this theorem, we will first observe that all of these complexes are finite and then show that there is a finite list that contains all of their isomorphism types.
Afterwards, we use a computer to verify that the reduced homology of this finite list of finite simplicial complexes vanishes in homological degrees $i \leq 3$ and that each complex is simply connected. The result then follows from Hurewicz's theorem. 

\subsection{Listing the isomorphism types}

We start by introducing notation that will be useful for studying the isomorphism types of $Q^m_n(\vec v_1,\vec v_2,\vec v_3;\vec w)$. It slightly differs from similar notation used in previous sections, but allows for an easy formalisation on a computer.

\begin{definition} \label{def:r-interval}
Let $R\in \mathbb{Z}_{\geq 1}$ and $z \in \Z$. We define the \emph{$R$-interval} of $z$ as
\begin{equation*}
\inter_R(z)\coloneqq 
\begin{cases}
  2  k, & \text{ if } z = k  R \text{ for some } k \in \Z;\\
  2 k+1, & \text{ if } k  R < z < (k+1)  R \text{ for some } k \geq 0;\\
  -(2 k+1), & \text{ if } -(k+1) R < z < - k R \text{ for some } k \geq 0.          
\end{cases}
\end{equation*}
If $\vec v \in \Z^{m+n}$ is a vector with last coordinate equal to $z \in \Z$, we write $\inter_R(\vec v) \coloneqq \inter_R(z)$.
\end{definition}

In other words, for $k\in \Z$, the $R$-interval of $z$ is $2k+1$ if $z$ lies in the open interval between $kR$ and $(k+1)R$ and it is equal to $2k$ if $z$ is equal to $kR$. 

\begin{example} The $3$-intervals and $4$-intervals of some integers are labelled in \autoref{Figure4Int}.
\begin{figure}[h!]
\labellist
\Large \hair 0pt
\pinlabel { \color{black} ${z}$} [tr] at -10 8
\pinlabel { \color{black} ${\inter_3(z)}$} [tr] at 0 48
\endlabellist
\qquad
\includegraphics[scale=.7]{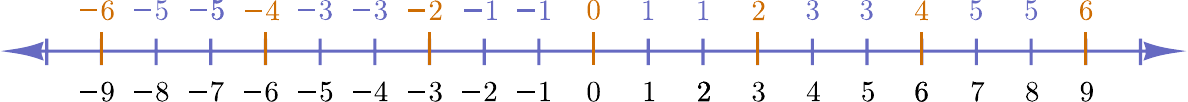}   \\[2em]
\labellist
\Large \hair 0pt
\pinlabel { \color{black} ${z}$} [tr] at -10 8
\pinlabel { \color{black} ${\inter_4(z)}$} [tr] at 0 48
\endlabellist
\qquad
\includegraphics[scale=.7]{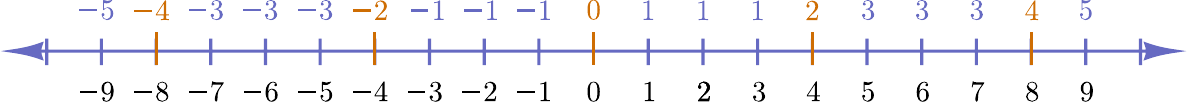}   \\[.2em]
\caption{Some values of $\inter_3(z)$ and $\inter_4(z)$.} 
\label{Figure4Int}
\end{figure}
\end{example}

The next lemma contains three elementary observations about the $R$-interval function (with $R$ fixed) that say that it is close to being linear: It commutes with scalar multiplication by $-1$, it is always close to being additive and it is actually additive if one of the inputs is a multiple of $R$. See \autoref{lem_properties_r}, which states similar results in a slightly different language.

\begin{lemma}
\label{R_int_additivity}
Let $R\in \mathbb{Z}_{\geq 1}$ and  $z, z_1,\, z_2\in \Z$. Then
\begin{enumerate}
\item \label{item_negative_R_int} $\inter_R(-z) = -\inter_R(z)$,
\item $\inter_R(z_1+z_2) \in  \ls \inter_R(z_1)+\inter_R(z_2)-1,\inter_R(z_1)+\inter_R(z_2), \inter_R(z_1)+\inter_R(z_2)+1\rs$, and
\item \label{item_sum_R_int_even} $\inter_R(z_1+z_2) = \inter_R(z_1)+\inter_R(z_2)$ if at least one of $\inter_R(z_1),\, \inter_R(z_2)$ is even.
\end{enumerate}
\end{lemma}
\begin{proof}
All three claims follow immediately from the definitions.
\end{proof}

The reason that we use these $R$-intervals is that they allow us to give a formal description of the vertex set of $Q^m_n(\vec v_1,\vec v_2,\vec v_3;\vec w)$. Using \autoref{R_int_additivity}, it is easy to deduce the following properties:
\begin{lemma} \label{parity-of-r-interval}
  Let $\vec v \in \Z^{m+n}$, let $\vec w \in \Z^{m+n}$ be a vector with last coordinate equal to $R \geq 1$ and $a\in \Z$.

  \begin{enumerate}
  	\item \label{item_last_coordinate_by_Rinter}The last coordinate of $\vec v$ has absolute value strictly smaller than $R$ if and only if $\inter_R(\vec v) \in \{-1, 0, 1\}$.
    \item If $\inter_R(\vec v) = 2k$ is even, then the last coordinate of $\vec v + a \vec w$ has absolute value strictly smaller than $R$ if and only if $ a = -k$. 
  \item If $\inter_R(\vec v) = 2k+1$ is odd,  then the last coordinate of $\vec v + a \vec w$ has absolute value strictly smaller than $R$ if and only if $ a \in \ls  -(k+1),-k \rs$.
  \end{enumerate}
\end{lemma}

A consequence of \autoref{parity-of-r-interval} is that all the $Q^m_n(\vec v_1,\vec v_2,\vec v_3;\vec w)$ are finite simplicial complexes on at most $12$ vertices.
Our next aim is to create an explicit (finite) list of simplicial complexes such that $Q^m_n(\vec v_1,\vec v_2,\vec v_3;\vec w)$ is isomorphic to one of these for every list of elements $(\vec v_1,\vec v_2,\vec v_3;\vec w)$ . This list will consist of complexes of the following form.

\begin{definition}
\label{def_Q_integers}
Let $\vec e_1, \vec e_2,\vec e_3,\vec e_4$ denote the standard basis of $\Z^4$.
We write $Q(r_1,r_2,r_3,r_{12},r_{123},r_{23})$ for the full subcomplex of $\Link_{\BAA_4}(e_4)$ on all lines spanned by vertices of the form 
\begin{multicols}{3} 
\begin{enumerate}
\item $\vec e_1+a_1 \vec e_4$,
\item $\vec e_2+a_2 \vec e_4$,
\item $\vec e_3+a_3 \vec e_4$,
\item $\vec e_1+\vec e_2+a_{12}\vec e_4$,
\item $\vec e_1+\vec e_2+\vec e_3+a_{123}\vec e_4$, or
\item $\vec e_1+\vec e_3+a_{13}\vec e_4$,
\end{enumerate}
\end{multicols} 
where $a_i = -k$ if $r_i = 2k$ is even and $ a_i \in \ls  -(k+1),-k \rs$ if  $r_i = 2k+1$ is odd.
\end{definition}

The following key proposition tells us that the isomorphism type of any complex $Q^m_n(\vec v_1,\vec v_2,\vec v_3 ;\vec w)$ is determined by six integers, namely the $R$-intervals of $\vec v_1,\vec v_2, \vec v_3, \vec v_1+\vec v_2, \vec v_1+\vec v_2+\vec v_3$ and $\vec v_1+\vec v_3$.

\begin{proposition}
\label{iso_type_by_Rint}
Let $\vec v_1,\vec v_2,\vec v_3$ and $\vec w$ be as in \autoref{def_Q} and let 
\begin{equation*}
(r_1,r_2,r_3,r_{12},r_{123},r_{13}) \coloneqq (\inter_R(\vec v_1),\inter_R(\vec v_2), \inter_R(\vec v_3), \inter_R(\vec v_1+ \vec  v_2), \inter_R(\vec v_1+ \vec v_2+ \vec v_3), \inter_R(\vec v_1 + \vec v_3)).
\end{equation*}
Then there is an isomorphism
\begin{equation*}
	\phi \colon Q^m_n(\vec v_1,\vec v_2,\vec v_3;\vec w) \to Q(r_1,r_2,r_3,r_{12},r_{123},r_{13}).
\end{equation*}
In particular, the isomorphism type of $Q^m_n(\vec v_1,\vec v_2,\vec v_3;\vec w)$ only depends on the six-tuple of integers $(r_1,r_2,r_3,r_{12},r_{123},r_{13})$.
\end{proposition}
\begin{proof}
To shorten notation in this proof, we set $\vec v_{12} \coloneqq \vec v_1+ \vec  v_2$, $\vec v_{123} \coloneqq \vec v_1+ \vec v_2+ \vec v_3$, $\vec v_{13} \coloneqq \vec v_1+ \vec  v_3$ and $\vec e_{12} \coloneqq \vec e_1+ \vec  e_2$, $\vec e_{123} \coloneqq \vec e_1+ \vec e_2+ \vec e_3$, $\vec e_{13} \coloneqq \vec e_1+ \vec  e_3$. With this, we have $\inter_R(\vec v_i) = r_i$ and by \autoref{parity-of-r-interval}, the span of $\vec v_i + a \vec w$ is a vertex in $Q^m_n(\vec v_1,\vec v_2,\vec v_3;\vec w)$ if and only if the span of $\vec e_i + a \vec e_4$ is a vertex in $Q(r_1,r_2,r_3,r_{12},r_{123},r_{13})$. This gives rise to an obvious bijection $\phi$ between the vertex sets of the two complexes.

We want to show that $\phi$ induces a simplicial isomorphism. Let $ l_0, \ldots, l_k \in Q^m_n(\vec v_1,\vec v_2,\vec v_3;\vec w)$. We need to show that $ l_0, \ldots, l_k $ form a simplex if and only if their images $ \phi(l_0), \ldots, \phi(l_k)$ do. Spelling out the definitions, one sees that $\ls l_0, \ldots, l_k \rs$ is a simplex in $Q^m_n(\vec v_1,\vec v_2,\vec v_3;\vec w)$ if and only if $\ls l_0, \ldots, l_k, w, e_1, \ldots, e_m \rs$ is a simplex in $\BAA_{m+n}$ and none of the $\vec l_i$ is in the span $\left\langle \vec e_1, \ldots, \vec e_m, \vec w \right\rangle$ (with a slight abuse of notation, we use the same symbols $\vec e_i$ to denote the standard basis of $\Z^{m+n}$ and $\Z^4$).
A set of vectors gives rise to a simplex in $\BAA_i$ if up to two of them are certain linear combinations of the others and the others from a partial basis; the form of the linear combinations depends on the type of simplex, see the definitions in \autoref{sec_definitions}. Assume that there is a linear dependency between $\ls \vec l_0, \ldots, \vec l_k, \vec w, \vec e_1, \ldots,\vec e_m \rs$, i.e.~there are $c_i,d_j$ such that
\begin{equation}
\label{eq_linear_dependency}
\underbrace{\sum_{i=0}^k c_i \vec l_i + c_{k+1} \vec w}_{\in \left\langle \vec v_1, \vec v_2, \vec v_3, \vec w \right\rangle} + \underbrace{\sum_{j=1}^m d_j \vec e_j}_{\in \left\langle \vec e_1, \ldots, \vec e_m \right\rangle} = 0.
\end{equation}
By assumption, $\ls v_1, v_2, v_3, w\rs$ is a simplex in $B_n^m$, which means that $\ls \vec v_1, \vec v_2, \vec v_3, \vec w, \vec e_1, \ldots, \vec e_m \rs$ is a partial basis. Hence, \autoref{eq_linear_dependency} implies $\sum_{i=0}^k c_i \vec l_i + c_{k+1} \vec w = 0$. It follows that $\ls l_0, \ldots, l_k, w, e_1, \ldots, e_m \rs$ is a simplex in $\BAA_{m+n}$ if and only if $\ls l_0, \ldots, l_k, w\rs$ is a simplex of the same type in 
the full subcomplex of $\BAA_{m+n}$ on all lines that are contained in $\left\langle \vec v_1, \vec v_2, \vec v_3, \vec w \right\rangle \cong \Z^4$. The latter is clearly equivalent to saying that $\ls \phi(l_0), \ldots, \phi(l_k), \phi(w)=e_4\rs$ is a simplex in $\BAA_4$, i.e.~that $\ls \phi(l_0), \ldots, \phi(l_k)\rs$ is a simplex in $Q(r_1,r_2,r_3,r_{12},r_{123},r_{13})$.
\end{proof}

By the above proposition, we can produce a list with all isomorphism types of the complexes $Q^m_n(\vec v_1,\vec v_2,\vec v_3;\vec w)$ by listing the possible combinations of $R$-intervals of $\vec v_1,\vec v_2, \vec v_3, \vec v_1+\vec v_2, \vec v_1+\vec v_2+\vec v_3$ and $\vec v_1+\vec v_3$. Before we do this in \autoref{listing_all_Qs}, we record in the following lemma isomorphisms between these complexes. These are easy to show and allow us to reduce the size of the list of isomorphism types, which is helpful for the computer calculations we want to perform.

\begin{lemma}
\label{isomorphic_Qs}
Let $r_1,r_2,r_3,r_{12},r_{123},r_{13} \in \Z$. We have the following identities:
\begin{enumerate}
\item \label{item_k1_k2} $Q(r_1,r_2,r_3,r_{12},r_{123},r_{13}) \cong Q(r_1,r_3,r_2,r_{13},r_{123},r_{12})$;
\item \label{item_k_minus_k}$Q(r_1,r_2,r_3,r_{12},r_{123},r_{13}) \cong Q(-r_1,-r_2,-r_3,-r_{12},-r_{123},-r_{13})$;
\end{enumerate}
\end{lemma}

We would like to remark that these are not the only isomorphisms that exist between complexes $Q(r_1,r_2,r_3,r_{12},r_{123},r_{13})$ and $Q(r'_1,r'_2,r'_3,r'_{12},r'_{123},r'_{13})$. However, they are sufficient to reduce the list of isomorphism types to a size that is small enough to allow computer calculations.

\begin{corollary}
\label{listing_all_Qs}
Let $\vec v_1,\vec v_2,\vec v_3,\vec w$ as in \autoref{def_Q}. Then
\begin{equation*}
Q(\inter_R(\vec v_1),\inter_R(\vec v_2), \inter_R(\vec v_3), \inter_R(\vec v_1+ \vec  v_2), \inter_R(\vec v_1+ \vec v_2+ \vec v_3), \inter_R(\vec v_1 + \vec v_3))
\end{equation*}
agrees with $Q(r_1,r_2,r_3,r_{12},r_{123},r_{13})$ for one of the tuples $(r_1,r_2,r_3,r_{12},r_{123},r_{13})$ listed in \autoref{table_iso_types}.
\end{corollary}

\begin{table}
\begin{tabular}{cccccc}
$r_1$ &$r_2$ &$r_3$ &$r_{12}$ &$r_{123}$ & $r_{13}$\\
-1 & -1 & -1 & -3 & -5 & -3 \\
-1 & -1 & -1 & -3 & -4 & -3 \\
-1 & -1 & -1 & -3 & -3 & -3 \\
-1 & -1 & -1 & -3 & -3 & -2 \\
-1 & -1 & -1 & -3 & -3 & -1 \\
-1 & -1 & -1 & -2 & -3 & -2 \\
-1 & -1 & -1 & -2 & -3 & -1 \\
-1 & -1 & -1 & -1 & -3 & -1 \\
-1 & -1 & -1 & -1 & -2 & -1 \\
-1 & -1 & -1 & -1 & -1 & -1 \\
-1 & -1 & 0 & -3 & -3 & -1 \\
-1 & -1 & 0 & -2 & -2 & -1 \\
-1 & -1 & 0 & -1 & -1 & -1 \\
-1 & -1 & 1 & -3 & -3 & -1 \\
-1 & -1 & 1 & -3 & -2 & -1 \\
-1 & -1 & 1 & -3 & -1 & -1 \\
-1 & -1 & 1 & -3 & -1 & 0 \\
-1 & -1 & 1 & -3 & -1 & 1 \\
-1 & -1 & 1 & -2 & -1 & -1 \\
-1 & -1 & 1 & -2 & -1 & 0 \\
-1 & -1 & 1 & -2 & -1 & 1 \\
-1 & -1 & 1 & -1 & -1 & -1 \\
-1 & -1 & 1 & -1 & -1 & 0 \\
-1 & -1 & 1 & -1 & -1 & 1 
\end{tabular}
\hspace{1cm}
\begin{tabular}{cccccc}
$r_1$ &$r_2$ &$r_3$ &$r_{12}$ &$r_{123}$ & $r_{13}$\\
-1 & -1 & 1 & -1 & 0 & 1 \\
-1 & -1 & 1 & -1 & 1 & 1 \\
-1 & 0 & 0 & -1 & -1 & -1 \\
-1 & 0 & 1 & -1 & -1 & -1 \\
-1 & 0 & 1 & -1 & 0 & 0 \\
-1 & 0 & 1 & -1 & 1 & 1 \\
-1 & 1 & 1 & -1 & -1 & -1 \\
-1 & 1 & 1 & -1 & 0 & -1 \\
-1 & 1 & 1 & -1 & 1 & -1 \\
-1 & 1 & 1 & 0 & 1 & -1 \\
-1 & 1 & 1 & 0 & 1 & 0 \\
-1 & 1 & 1 & 0 & 1 & 1 \\
-1 & 1 & 1 & 1 & 1 & -1 \\
-1 & 1 & 1 & 1 & 1 & 1 \\
-1 & 1 & 1 & 1 & 2 & 1 \\
-1 & 1 & 1 & 1 & 3 & 1 \\
0 & -1 & -1 & -1 & -3 & -1 \\
0 & -1 & -1 & -1 & -2 & -1 \\
0 & -1 & -1 & -1 & -1 & -1 \\
0 & -1 & 0 & -1 & -1 & 0 \\
0 & -1 & 1 & -1 & -1 & 1 \\
0 & -1 & 1 & -1 & 0 & 1 \\
0 & -1 & 1 & -1 & 1 & 1 \\
0 & 0 & 0 & 0 & 0 & 0 
\end{tabular}
\caption{A list of 48 tuples $(r_1,r_2,r_3,r_{12},r_{123},r_{13})$ containing (at least) one representative for each isomorphism type $Q(r_1,r_2,r_3,r_{12},r_{123},r_{13})$.}
\label{table_iso_types}
\end{table}

\begin{proof}
Let $\vec v_{12} \coloneqq \vec v_1+ \vec  v_2$, $\vec v_{123} \coloneqq \vec v_1+ \vec v_2+ \vec v_3$, $\vec v_{13} \coloneqq \vec v_1+ \vec  v_3$ and let $z_i$ denote the last coordinate of $v_i$.
As the last coordinates of $\vec v_1, \vec v_2, \vec v_3$ are smaller than $R$, we have
\begin{equation*}
\inter_R(z_i) \in \ls -1,0,1 \rs \text{ for } 1\leq i \leq 3.
\end{equation*}
Furthermore, by \autoref{R_int_additivity}, we know that $\inter_R(z_{12}) = \inter_R(z_1+z_2)$ is either equal to the sum $\inter_R(z_1)+\inter_R(z_2)$ or differs from it by at most $1$, depending on the parity of $\inter_R(z_1)$ and $\inter_R(z_2)$. Similarly, $\inter_R(z_{123}) = \inter_R(z_1+z_2+z_3)$ can differ from the sum $\inter_R(z_{12})+\inter_R(z_3)$ by at most one.
Also, $z_{13} = z_1+z_3$ can be written both as $(z_1)+(z_3)$ and as $(z_1+z_2+z_3)-(z_2)$. Hence, its $R$-interval differs both from $\inter_R(z_1)+\inter_R(z_3)$ and from $\inter_R(z_{123})-\inter_R(z_2)$ by at most one. 

These rules allow one to generate a list with all possible tuples that can occur as
\begin{equation*}
(\inter_R(z_1),\inter_R(z_2), \inter_R(z_3), \inter_R(z_{12}), \inter_R(z_{123}), \inter_R(z_{13}) ).
\end{equation*}
The list can be further shortened by using the identities of \autoref{isomorphic_Qs}. We did this using computer calculations (available under \url{https://github.com/benjaminbrueck/codim2_cohomology_SLnZ/blob/main/Connectivity%20Q%20complexes.ipynb}) and the result was \autoref{table_iso_types}.
\end{proof}

\subsection{Computer implementation of the complexes}
To show that all the complexes obtained in \autoref{listing_all_Qs} are indeed 3-connected, we use computer calculations. These are made available under the following link \url{https://github.com/benjaminbrueck/codim2_cohomology_SLnZ}.

The core of the calculations is a function, written in python, that takes as an input a set of vectors in $\Z^n$ and returns the subcomplex of $\BAA_n$ that is spanned by these vectors. This simplicial complex is implemented using the \texttt{Simplex Tree} module of  {\sc gudhi} \cite{gudhi:FilteredComplexes}. The {\sc gudhi} library was developed for topological data analysis. It allows to conveniently work with filtered simplicial complex and we used the filtration functionality to keep track of the type of the simplices (standard, 2-additive, 3-additive, double-triple or double-double). However, one cannot compute homology with integral coefficients in {\sc gudhi}. For performing these homology computations, we use the \texttt{SimplicialComplex} class of {\sc SageMath} \cite{sagemath}.

\subsubsection{Simplices by facet type}
One fact that we used for building subcomplexes of $\BAA_n$ on a computer is that for many simplices, it is sufficient to know what types of simplices their facets form. This is used in the computations to check whether a set of vertices forms a simplex.

\begin{definition}
Let $S = \ls v_0, \ldots, v_d \rs$ be a set of vertices of $\BAA_n$ such that every $d$-element subset of $S$ forms a simplex in $\BAA$. Then the \emph{facet type} of $S$ is the multiset of simplex types that arise among these $d$-element subsets.
\end{definition}

With slight abuse of notation, call a subset of size $k$ a facet of $S = \ls v_0, \ldots, v_k \rs$ (even if it does not necessarily form a simplex in $\BAA_n^m$).

\begin{example}
Let $e_1, e_2, e_3, e_4$ be the standard basis of $\Z^4$ and let $S = \ls e_1, e_2, e_3, \langle \vec e_1+ \vec e_2 \rangle, \langle \vec e_1+ \vec e_3 \rangle, e_4\rs$. Then the facet type of $S$ is 
\begin{equation*}
	\ls \text{3-additive}, \text{2-additive}, \text{2-additive}, \text{2-additive}, \text{2-additive}, \text{double-triple} \rs.
\end{equation*}
\end{example}

\begin{observation}
\label{obs_list_facet_types}
If $\tau$ is one of the types of simplices defined in \autoref{sec_definitions} and $d\in \Z$, then every set that forms a $d$-dimensional simplex of type $\tau$ has the same facet type. These types are as follows:
Let $S$ be a set of vertices of $\BAA$.
\begin{enumerate}
\item If $S$ forms a standard simplex then its facet type is
	\begin{equation*}
		\ls \text{standard}, \ldots , \text{standard}\rs.
	\end{equation*}
\item If $S$ forms a 2-additive simplex, then its facet type is
	\begin{equation*}
		\ls \text{standard}, \text{standard}, \text{standard}, \text{2-additive}, \ldots , \text{2-additive} \rs.
	\end{equation*}
\item If $S$ forms a 3-additive simplex, then its facet type is
	\begin{equation*}
		\ls \text{standard}, \text{standard}, \text{standard}, \text{standard}, \text{3-additive}, \ldots , \text{3-additive} \rs.
	\end{equation*}
\item If $S$ forms a double-triple simplex, then its facet type is
	\begin{equation*}
		\ls \text{2-additive}, \text{2-additive},  \text{3-additive}, \text{2-additive}, \text{2-additive}, \text{double-triple}, \ldots, \text{double-triple} \rs.
	\end{equation*}
\item If $S$ forms a double-double simplex, then its facet type is
	\begin{equation*}
		\ls \text{2-additive}, \text{2-additive}, \text{2-additive}, \text{2-additive}, \text{2-additive}, \text{2-additive}, \text{double-double}, \ldots, \text{double-double} \rs.
	\end{equation*}	
\end{enumerate}
\end{observation}

We will see that in most cases, the converse of this is true as well, i.e.~if we have a set of vertices whose facet type agrees with one of the types of the list above, then it already forms a simplex of the corresponding type.

\begin{definition}
Let $\tau$ be one of the types of simplices defined in \autoref{sec_definitions} and let $d\in \Z$. We say that \emph{$\tau$ is determined by its facet type in dimension $d$} if the following is true: Given a set of vertices $S = \ls v_0, \ldots, v_d \rs$ of $\BAA_n$ such that every $d$-element subset of $S$ forms a simplex in $\BAA_n$. Then $S$ forms a simplex of type $\tau$ if and only if it has the same facet type as a $d$-dimensional simplex of type $\tau$.
\end{definition}

It is not hard to check that only 2- and 3-additive simplices are determined by their facet type if they are not minimal and that double-triple and double-double simplices are even determined by their facet types in all possible dimensions. We use these properties for the computer implementation of the complexes. We record them in the following lemma, but omit the (elementary) proofs.

\begin{lemma}
\label{double-triple_determined_by_facets}
\label{double-double_determined_by_facets}
The following hold:
\begin{enumerate}
\item An $m$-additive simplex is determined by its facet type in all dimensions $d>m$.
\item A double-triple simplex is determined by its facet type in all dimensions $d\geq 4$.
\item A double-double simplex is determined by its facet type in all dimensions $d\geq 5$.
\end{enumerate}
\end{lemma}

\subsection{Results of the homology calculations and simple connectivity}

In addition to computing the homology of the complexes $Q^m_n(\vec v_1,\vec v_2,\vec v_3;\vec w)$, we need to show that they are simply connected. We will do this by showing that they are very ``dense'' and using the following criterion:

\begin{lemma}
\label{simple_conn_by_density}
Let $Q$ be a simplicial complex with $k$ vertices. Assume that every pair of vertices forms an edge in $Q$ and that there are only $m$ triples of vertices that do not form a two-simplex. If $m<k-2$, then $Q$ is simply connected.
\end{lemma}
\begin{proof}
Let $x$ be a vertex of $Q$ and let $\pi_1(Q,x)$ denote the fundamental group of $Q$ with base point $x$. The inclusion of the $1$-skeleton $Q^{(1)}\hookrightarrow Q$ induces a surjection $f \colon  \pi_1(Q^{(1)},x) \twoheadrightarrow \pi_1(Q,x)$. Hence, it is sufficient to show that $f$ has trivial image.

The $1$-skeleton $Q^{(1)}$ is the full graph with vertex set $Q^{(0)}$. This implies that $\pi_1(Q^{(1)},x)$ is a free group with generating set given by
$\ls \Delta(x,u,v) | u,v \in Q^{(0)}\setminus \{x\} \rs$, where $\Delta(x,u,v)$ is the triangle consisting of the three (oriented) edges from $x$ to $u$, $u$ to $v$ and $v$ to $x$. We need to show that $\Delta(x,u,v)$ is trivial in $\pi_1(Q,x)$. This is definitely true if $\ls x,u,v \rs$ forms a 2-simplex in $Q$, so we can assume that this is not the case. It then suffices to show that there is a vertex $w$ such that $\ls x,u,w \rs$, $\ls x,v,w \rs$ and $\ls u,v,w \rs$ are all 2-simplices in $Q$; such a $w$ would form a cone point for the triangle $\Delta(x,u,v)$, showing that it is trivial in $\pi_1(Q,x)$. 

Now by assumption, there are at most $m-1$ triples other than $\ls x,u,v \rs$ that do not form a 2-simplex. Hence, if there are more than $(m-1+3) = m+2$ vertices, there is at least one $w$ with the desired properties.
\end{proof}

The preceding \autoref{simple_conn_by_density} also follows from \cite[Lemma 2.1]{Babson2011}.

\begin{lemma}
\label{Q_simply_conn}
For all tuples $(r_1,r_2,r_3,r_{12},r_{123},r_{13})$ in \autoref{table_iso_types}, the complex $Q(r_1,r_2,r_3,r_{12},r_{123},r_{13})$ is simply connected.
\end{lemma}
\begin{proof}
This follows from \autoref{simple_conn_by_density} using the computer calculations in the jupyter notebook \url{https://github.com/benjaminbrueck/codim2_cohomology_SLnZ/blob/main/Connectivity%20Q%20complexes.ipynb}.
\end{proof}

We are now ready to show that every complex $Q^m_n(\vec v_1,\vec v_2,\vec v_3;\vec w)$ is $3$-connected.
\begin{proof}[Proof of \autoref{Qconn}]
By \autoref{iso_type_by_Rint} and \autoref{listing_all_Qs}, it suffices to find for every tuple $(r_1,r_2,r_3,r_{12},r_{123},r_{13})$ in \autoref{table_iso_types} vectors $\vec v_1,\vec v_2,\vec v_3,\vec w \in \Z^4$ with last entries $z_1,z_2,z_3,R$ such that $(r_1,r_2,r_3,r_{12},r_{123},r_{13})$ is given by 
\begin{equation*}
(\inter_R(z_1),\inter_R(z_2), \inter_R(z_3), \inter_R(z_1+z_2), \inter_R(z_1+z_2+z_3), \inter_R(z_1+z_3) )
\end{equation*}
and to show that $Q_4^0(\vec v_1,\vec v_2,\vec v_3;\vec w)$ is 3-connected.

As this complex is always simply connected (\autoref{Q_simply_conn}), Hurewicz's theorem implies that it is sufficient to show that its integral homology vanishes in degrees $2$ and $3$. This reduces the proof to computing the homology of a finite list of finite simplicial complexes. We performed these calculations with a computer, the results can be found in the following notebook \url{https://github.com/benjaminbrueck/codim2_cohomology_SLnZ/blob/main/Connectivity%20Q%20complexes.ipynb}.
\end{proof}

\subsection{Resource consumption, runtime and verifiability of the computer calculations}
\label{sec_details_on_computer_calculations}
All of the used algorithms are exact and guaranteed to terminate. The entire computations take less than one minute on a mid-class laptop and have negligible memory consumption. 

There are four steps in this section where we use computer calculations.
Firstly, to find the list of isomorphism types of the complexes $Q(r_1,r_2,r_3,r_{12},r_{123},r_{13})$ given in \autoref{listing_all_Qs}. Finding this list, i.e.~creating \autoref{table_iso_types}, amounts in a simple application of the relations given in \autoref{R_int_additivity} and \autoref{isomorphic_Qs}. This is done by a sequence of case distinctions. While this is a tedious task and we believe that the computer is less likely to make mistakes, this can also be verified by hand.
Secondly, to find a representative for each such isomorphism type, i.e.~to find for each tuple of integers $(r_1,r_2,r_3,r_{12},r_{123},r_{13})$ in \autoref{table_iso_types} a basis $\vec v_1,\vec v_2,\vec v_3,\vec w$ of $\Z^4$ such that 
\begin{equation*}
(\inter_R(\vec v_1),\inter_R(\vec v_2), \inter_R(\vec v_3), \inter_R(\vec v_1+ \vec  v_2), \inter_R(\vec v_1+ \vec v_2+ \vec v_3), \inter_R(\vec v_1 + \vec v_3)) = (r_1,r_2,r_3,r_{12},r_{123},r_{13}).
\end{equation*}
For this, the computer needs to calculate $R$-intervals and to check whether a set of vectors forms a basis of $\Z^4$.
It is easy to verify by hand (also just in examples) that the vectors given by the computer actually form a basis and do have the correct $R$-intervals.
Thirdly, to calculate the set of simplices for each of the 48 complexes $Q_4^0(\vec v_1,\vec v_2,\vec v_3;\vec w)$. This is done by iterating through increasingly big subsets of the vertex set and for each such subset checking whether it forms a simplex in $\BAA_4^0$. Using \autoref{double-triple_determined_by_facets}, it is sufficient to do these checks for standard, 2-additive and 3-additive simplices. This amounts in verifying whether a set of lines in $\Z^4$ forms a partial basis or satisfies a certain linear relation. The code for this, together with comments giving further explanations, is contained in the files \texttt{complex\_constructor.py} and \texttt{simplex\_constructor.py} in the repository \url{https://github.com/benjaminbrueck/codim2_cohomology_SLnZ}.
Lastly, the computer calculates the homology of the 48 given complexes $Q_4^0(\vec v_1,\vec v_2,\vec v_3;\vec w)$ and counts their simplices in order to show that they are simply connected by \autoref{simple_conn_by_density}. These homology calculations are performed with established software (the \texttt{SimplicialComplex} class of {\sc SageMath} \cite{sagemath}) and can also be verified with different existing or self-written code. The efficiency of the used software here is not very important as the complexes are all comparably small (they each have between 62 and 1097 simplices).

\section{Towards the connectivity of $\BAA_n^m$}
\label{Sec5}
In this section, we prepare for proving that the complexes $\BAA_n^m$ are Cohen--Macaulay (\autoref{BAAnmCM}). We study links and certain subcomplexes of the links. We prove the case $n=1$, which will be our induction base case, and we show some auxiliary results that will be used in the induction step. Throughout this section, we assume that $n\geq 1$ and $n+m \geq 3$.

\subsection{Description of $\Link$, $\Linkhat$ and the Cohen--Macaulay property}

In this subsection, we show that the complexes $\BAA_n^m$ are Cohen--Macaulay, provided that they are highly-connected.

\begin{definition}
\label{def_J}
Let $\sigma$ be 3-additive simplex of $\BAA_n^m$. We can write $\sigma=\ls v_0,v_1,\ldots,v_k\rs$, where $\{ \vec v_1,\ldots, \vec v_k,  \vec e_1, \dots, \vec e_m\}$ is a partial basis and $\vec v_0 = \vec w_1 + \vec w_2 + \vec w_3$ for certain $w_1, w_2, w_3 \in \{ v_1,\ldots,v_k,e_1,\ldots,e_m \}$. 
Let $J(\sigma)$ be the set of vertices of $\BAA_n^m$ that are lines spanned by a vector of the form $\{ \vec w_1+\vec w_2, \vec w_1+\vec w_3, \vec w_2+\vec w_3 \}$.
\end{definition}

Note that $J(\sigma)$ might contain less than three elements (e.g.~if $\vec v_0 = \vec v_1+ \vec e_1 + \vec e_2$, because $\ls \langle \vec e_1 + \vec e_2 \rangle \rs \not\in \BAA_m^n$), but it is always nonempty. Going thorugh the definitions of different simplex types, one obtains the following:

\begin{lemma}
\label{links_of_BAA_augmentations}
Let $\sigma$ be a simplex of $\BAA_m^n$.
\begin{enumerate}
\item If $\sigma$  is a standard simplex of dimension $k$, there is an isomorphism $\Linkhat_{\BAA_n^m}(\sigma) \cong \BAA_{n-k-1}^{m+k+1}$.
\item If $\sigma$ is a 2-additive simplex, we can write $\sigma=\ls v_0,v_1,\ldots,v_k\rs$ with $\ls \vec v_1,\ldots, \vec v_k\rs$ a partial basis. We then have $\Linkhat_{\BAA_n^m}(\sigma) = \Linkhat_{\BA_n^m}(\ls v_1,\ldots,v_k\rs)$.
\item If $\sigma$ is a 3-additive simplex, we can write $\sigma=\ls v_0,v_1,\ldots,v_k\rs$ with $\ls \vec v_1,\ldots, \vec v_k\rs$ a partial basis. We then have 
	\begin{equation*}
		\Linkhat_{\BAA_n^m}(\sigma) = \Link_{\B_n^m}(\ls v_1,\ldots,v_k\rs) \text{ and } \Link_{\BAA_n^m}(\sigma) = \Link_{B_n^m}(\ls v_1,\ldots,v_k\rs)\ast J(\sigma),
	\end{equation*}
where $J(\sigma)$ is seen as a 0-dimensional complex.
\item If $\sigma$ is a double-triple or double-double simplex, we can write $\sigma = \ls v_0,v_1,\ldots,v_k\rs$ with $\ls \vec v_2,\ldots, \vec v_k\rs$ a partial basis. We then have 
	\begin{equation*}
		\Link_{\BAA_n^m}(\sigma) = \Linkhat_{\BAA_n^m}(\sigma) = \Link_{\B_n^m}(\ls v_2,\ldots,v_k\rs).
	\end{equation*}
\end{enumerate}
\end{lemma}

The description of the links in the following lemma is easy to see and will both be used in \autoref{sec_connectivity_BAA} and in the  proof of  \autoref{prop_connecitivity_links} below.

\begin{lemma}
\label{link_linkhat_linkhat}
Let $\sigma$ be a simplex of $\BAA_n^m$ and $\tau \in \Link_{\BAA_n^m}(\sigma)$ such that no vertex of $\tau$ is in $\Linkhat_{\BAA_n^m}(\sigma)$. Then $\Link_{\Link_{\BAA_n^m}(\sigma)}(\tau)\cap \Linkhat_{\BAA_n^m}(\sigma) = \Linkhat_{\BAA_n^m}(\sigma\cup \tau)$.
\end{lemma}

Once we show connectivity of $\BAA_n^m$,  \autoref{prop_connecitivity_links} below implies that the complex is Cohen--Macaulay. To prove this proposition, we will use the following lemma. 

\begin{lemma}[{Galatius--Randal-Williams \cite[Proposition 2.5]{GRWI}}] \label{GRW-SubcomplexLemma}
Let $X$ be a simplicial complex, and $Y \subseteq X$ be a full subcomplex.
Let $N$ be an integer with the property that for each $p$-simplex $\tau$ in $X$ having no
vertex in $Y$, the complex $Y \cap \Lk_X (\tau)$ is $(N - p - 1)$-connected. Then the inclusion
$|Y |\to |X|$ is $N$-connected.
\end{lemma}

\begin{proposition}
\label{prop_connecitivity_links}
If $\BAA_n^m$ is $n$-connected for all $n \geq 1$ and $m+n \geq 3$,
then for every $k$-simplex $\sigma$, the link $\Link_{\BAA_n^m}(\sigma)$ is $(n-k-1)$-connected.
\end{proposition}

\begin{proof}
\noindent {\bf Case 1: $\sigma$ is a 3-additive, double-double or double-triple simplex. } 
By the work of Church--Putman, there is an isomorphism $\Link_{B_n^m}(\ls v_0,\ldots,v_\ell \rs)\cong \B_{n-\ell-1}^{m+\ell+1}$ \cite[Lemma 4.3]{CP} and this complex is $(n-\ell-3)$-connected for all $n,m\geq 0$  and $0 \leq \ell  \leq n-1$ \cite[Theorem 4.2]{CP}.
Combining this with \autoref{links_of_BAA_augmentations}, we obtain the claim if $\sigma$ is a 3-additive, double-double or double-triple simplex. \\

\noindent {\bf Case 2: $\sigma$ is a 2-additive simplex.} 
Next assume that $\sigma$ is 2-additive. If $k=n$, then we must check that  $\Link_{\BAA_n^m}(\sigma)$ is nonempty. First suppose $\sigma$ has the form $\ls \langle \vec v_1 + \vec v_2 \rangle, v_1,\ldots,v_n\rs$, where $\{\vec v_1,\ldots,\vec v_n, \vec e_1, \ldots \vec e_m\}$ is a basis for $\Z^{m+n}$.  Since $m+n \geq 3$, either $m \geq 1$ or  $m=0$ and $n\geq 3$. In these two cases, either $\langle \vec v_1 + \vec e_1 \rangle$ or $\langle \vec v_1 + \vec v_3 \rangle$, respectively, is an element of $\Link_{\BAA_n^m}(\sigma)$. Alternatively suppose $\sigma$ has the form $\ls \langle \vec v_1 + \vec e_1 \rangle, v_1,\ldots,v_n\rs$, where $\{\vec v_1,\ldots,\vec v_n, \vec e_1, \ldots \vec e_m\}$ is a basis for $\Z^{m+n}$. Now $m+n \geq 3$ implies either $m\geq 2$ or $n \geq 2$. But then at least one of $\langle \vec v_1 + \vec e_2 \rangle$ or $\langle \vec v_1 + \vec v_2 \rangle$ must be a vertex in $\Link_{\BAA_n^m}(\sigma)$. 
 
Now suppose that $k\not= n$. By \autoref{links_of_BAA_augmentations}, we can write $\sigma=\ls v_0,v_1,\ldots,v_k\rs$ with $\ls \vec v_1,\ldots, \vec v_k\rs$ a partial basis and  $$\Linkhat_{\BAA_n^m}(\sigma) = \Linkhat_{\BA_n^m}(\ls v_1,\ldots,v_k\rs).$$
  
By \cite[Lemma 4.12(b)]{CP}, $\Linkhat_{\BA_n^m}(\ls v_1,\ldots,v_k\rs)$ is isomorphic to $\BA_{n-k}^{m+k}$ and this complex is $(n-k-1)$-connected by \cite[Theorem C']{CP}.
We want to extend this to $\Link_{\BAA_n^m}(\sigma)\supseteq \Linkhat_{\BAA_n^m}(\sigma)$. 
We will apply \autoref{GRW-SubcomplexLemma} with $N=n-k-1$. 
We need to show that for every $\tau \in \Link_{\BAA_n^m}(\sigma)$ that has no vertex in $\Linkhat_{\BAA_n^m}(\sigma)$, the intersection $$\Link_{\Link_{\BAA_n^m}(\sigma)}(\tau)\cap \Linkhat_{\BAA_n^m}(\sigma)$$ is $(n-k-\dim(\tau)-2))$-connected. By \autoref{link_linkhat_linkhat}, this complex is equal to $\Linkhat_{\BAA_n^m}(\sigma\cup \tau)$.
Every vertex $v$ of $\tau$ satisfies $\vec v \in \langle \vec v_0, \ldots, \vec v_k, \vec e_1, \ldots, \vec e_m \rangle$, so in particular, it is contained in the additive core of $\sigma\cup \tau$. Hence, $\sigma\cup \tau$ is a double-double or double-triple simplex of dimension $k+\dim(\tau)+1$.\footnote{In fact, it follows that $\dim(\tau) = 0$ here.} As observed in \autoref{links_of_BAA_augmentations},  $\Linkhat_{\BAA_n^m}(\sigma\cup \tau) = \Link_{\BAA_n^m}(\sigma\cup \tau)$ for such simplices. In Case 1 we showed that this link is $(n-k-\dim(\tau ) -2)$-connected as claimed. \\

\noindent {\bf Case 3: $\sigma$ is a standard simplex. }  
If $\sigma$ is a standard simplex, we apply the same argument in two steps. 
By \autoref{links_of_BAA_augmentations}, $\Linkhat_{\BAA_n^m}(\sigma)$ is isomorphic to $\BAA_{n-k-1}^{m+k+1}$ and this is $(n-k-1)$-connected by our assumption.
Furthermore, every vertex in $\Link_{\BAA_n^m}(\sigma)\setminus \Linkhat_{\BAA_n^m}(\sigma)$ is either of the form $\langle \vec w_1+ \vec w_2 \rangle$ or $\langle \vec w_1+\vec w_2+\vec w_3 \rangle$ for some $w_1,w_2,w_3 \in \ls v_0, \ldots, v_k, e_1, \ldots, e_m\rs$.
We will apply \autoref{GRW-SubcomplexLemma} twice to the chain subcomplexes $$\Linkhat_{\BAA_n^m}(\sigma)\subseteq Z \subseteq \Link_{\BAA_n^m}(\sigma),$$ where $Z$ is spanned by $\Linkhat_{\BAA_n^m}(\sigma)$ and all vertices of the form $\langle \vec w_1+ \vec w_2 \rangle$ as above. We first consider the inclusion $\Linkhat_{\BAA_n^m}(\sigma)\hookrightarrow Z$ and consider \autoref{GRW-SubcomplexLemma} with $N=n-k-1$. Let $\tau$ be a simplex of $Z$ that has no vertex in $\Linkhat_{\BAA_n^m}(\sigma)$. Then using \autoref{link_linkhat_linkhat}, 
\begin{equation*}
\Link_{Z}(\tau)\cap \Linkhat_{\BAA_n^m}(\sigma) = \Linkhat_{\BAA_n^m}(\sigma\cup \tau).
\end{equation*}
Depending on the form of $\tau$, the simplex  $\sigma\cup \tau$ is 2-additive, a double-triple simplex or a double-double simplex. 
For each possibility, we have already seen in Case 2 that $\Linkhat_{\BAA_n^m}(\sigma\cup \tau)$ is $(n-k-\dim(\tau)-2)$-connected.
 It follows that $Z$ is $(n-k-1)$-connected.

Now we apply  \autoref{GRW-SubcomplexLemma} to the inclusion $Z \hookrightarrow \Link_{\BAA_n^m}(\sigma)$ and again let $N=n-k-1$.  Let $\tau$ be a simplex of $\Link_{\BAA_n^m}(\sigma)$ that has no vertex in $Z$. Then $\tau$ has the form $\ls \langle w_1 + w_2 +w_3 \rangle \rs$ and $\sigma \cup \tau$ is 3-additive. It follows from \autoref{link_linkhat_linkhat} that
\begin{equation*}
	\Link_{\Link_{\BAA_n^m}(\sigma)}(\tau)\cap Z = \Link_{\BAA_n^m}(\sigma\cup \tau).
\end{equation*}
We already demonstrated in Case 1 that this 3-additive simplex's link is $(n-k-\dim(\tau)-2)$-connected. This implies that $\Link_{\BAA_n^m}(\sigma)$ is $(n-k-1)$-connected, and concludes the final case in the proof. \end{proof}

\begin{remark}
The preceding proof shows that for a standard, 2-additive, double-triple and double-double simplex $\sigma$, not only $\Link_{\BAA_n^m}(\sigma)$, but also $\Linkhat_{\BAA_n^m}(\sigma)$ is $(n-\dim(\sigma)-1)$-connected. The latter is however not the case for 3-additive simplices.
\end{remark}

\subsection{Description of $\Link^<$, $\Linkhat^<$}

We next describe the structure of certain links in $\BAA_n^m$. As before, we omit the proofs of a few statements that simply follow by spelling out the definitions.

\begin{lemma}
\label{equaility_link_linkhat}
Let $v \in \BAA_n^m$ be a vertex with nonzero last coordinate. Then
\begin{enumerate}
\item $\Link^{<}_{\BAA_n^m}(v) = \Linkhat^{<}_{\BAA_n^m}(v) $, and
\item $\Link_{\BAA_n^m}^{<}(\ls v,\langle \vec v \pm \vec e_i\rangle\rs) = \Linkhat_{\BAA_n^m}^{<}(\ls v,\langle \vec v \pm \vec e_i\rangle\rs)$ for all $1\leq i \leq m$.
\end{enumerate}
\end{lemma}

Later on, we will need to know that $\Link_{\BAA_n^m}^{<}(\ls v,\langle \vec v \pm \vec e_i\rangle\rs) = \Linkhat_{\BAA^m_n}^{<}(\ls v,\langle \vec v \pm \vec e_i\rangle\rs)$ is highly-connected. To prepare for this, we compare this complex to $\Linkhat_{\BA^m_n}^{<}(v)$:

\begin{lemma}
\label{lem_compare_linkhatBA_and_linhatBAA}
Let $v \in \BAA_n^m$ be a vertex with nonzero last coordinate and $1\leq i \leq m$.
Let $\sigma$ be a set of vectors in $\Z^{m+n}$. Then the following hold. 
\begin{enumerate}
\item  \label{it_standard} The simplex $\sigma$ is a standard simplex of $\Linkhat_{\BA^m_n}^{<}(v)$ if and only if it is a $2$-additive simplex of $\Linkhat_{\BAA^m_n}^{<}(\{v,\langle \vec v \pm \vec e_i\rangle\})$. 
\item \label{it_2add_doubledouble} The simplex  $\sigma$ is a $2$-additive simplex of $\Linkhat_{\BA^m_n}^{<}(v)$ such that the additive core of $\ls e_1, \dots, e_m\rs \cup \ls v\rs \cup \sigma$ does not contain $v$ or $e_i$ if and only $\sigma$ is a double-double simplex of $\Linkhat_{\BAA^m_n}^{<}(\{v,\langle \vec v \pm \vec e_i\rangle\})$.
\item \label{it_2add_doubletriple} If $\sigma$ is a $2$-additive simplex of $\Linkhat_{\BA^m_n}^{<}(v)$ such that the additive core of $\ls e_1, \dots, e_m\rs \cup \ls v\rs \cup \sigma$ contains $v$ or $e_i$, then $\sigma$ is a double-triple simplex of $\Linkhat_{\BAA^m_n}^{<}(\{v,\langle \vec v \pm \vec e_i\rangle\})$.
\item \label{it_doubletriple} If $\sigma$ is a double-triple simplex of $\Linkhat_{\BAA^m_n}^{<}(\{v,\langle \vec v \pm \vec e_i\rangle\})$, then it is a simplex of $\Linkhat_{\BA^m_n}^{<}(v)$ except if it is of the form $\sigma = \ls w, \langle \vec v \pm \vec e_i+ \vec w\rangle, v_2, \ldots, v_k \rs$.
\item \label{it_3add} No simplex of $\Linkhat_{\BAA^m_n}^{<}(\{v,\langle \vec v \pm \vec e_i\rangle\})$ is $3$-additive.
\end{enumerate}
In particular, $\Linkhat_{\BA^m_n}^{<}(v)$ is a subcomplex of $\Linkhat_{\BAA^m_n}^{<}(\{v,\langle \vec v \pm \vec e_i\rangle\}) \subseteq \Link_{\BAA_n^m}(\{v,\langle \vec v \pm \vec e_i\rangle\})$ and every simplex of $\Linkhat_{\BAA^m_n}^{<}(\{v,\langle \vec v \pm \vec e_i\rangle\})$ that is not contained in $\Linkhat_{\BA^m_n}^{<}(v)$ is of type double-triple.
\end{lemma}
\begin{proof}
Throughout this proof, we will use the observation that for $1\leq j,k\leq m$, the lines $\langle \vec v \pm \vec e_j\rangle$, $\langle \vec v \pm \vec e_i \pm \vec e_j \rangle$, $\langle  \vec e_j \pm \vec e_k \rangle$ or $\langle \vec e_i \pm \vec e_j \pm \vec e_k \rangle$ are {\bf not} vertices of $\Linkhat_{\BA^m_n}^{<}(v)$ or $\Linkhat_{\BAA^m_n}^{<}(v)$. This follows because their last coordinate is equal to that of $v$ or they lie in $\langle \vec e_1, \ldots, \vec e_m \rangle$.

\autoref{it_standard} is immediate.
\autoref{it_2add_doubledouble}  follows because, if the additive core of $\ls e_1, \dots, e_m\rs \cup \ls v\rs \cup \sigma$ does not contain $v$ or $e_i$, then $\ls e_1, \dots, e_m\rs \cup \ls v, \langle \vec v \pm \vec e_i\rangle\ \rs \cup \sigma$ contains two disjoint 2-additive faces.

If $\sigma$ is a $2$-additive simplex of $\Linkhat_{\BA^m_n}^{<}(v)$ such that the additive core of $\ls e_1, \dots, e_m\rs \cup \ls v\rs \cup \sigma$ contains $v$ or $e_i$, then this additive core must be of the form $\ls v, w, \langle \vec v + \vec w \rangle \rs$ or $\ls e_i, w, \langle \vec e_i + \vec w \rangle \rs$ for some $w \in \sigma$. This implies 
\autoref{it_2add_doubletriple}.

For \autoref{it_doubletriple} note that if $\sigma$ is a double-triple simplex of $\Linkhat_{\BAA^m_n}^{<}(\{v,\langle \vec v \pm \vec e_i\rangle\})$, then the additive core of $\ls e_1, \dots, e_m\rs \cup \ls v, \langle \vec v \pm \vec e_i\rangle\ \rs \cup \sigma$ is of the form 
\begin{equation*}
\{e_i, v, \langle \vec v \pm \vec e_i\rangle\} \cup \tau,
\end{equation*}
where $\tau = \ls w, \langle \vec v + \vec w \rangle \rs$, $\ls w, \langle \pm \vec e_i + \vec w \rangle \rs$ or $\ls w, \langle \vec v \pm \vec e_i+ \vec w\rangle\rs$ for some $w \in \sigma$. 
In the first two cases, $\sigma = \tau \cup \ls v_2, \ldots, v_k \rs$ is a simplex of $\Linkhat_{\BA^m_n}^{<}(v)$. If however $\ls w, \langle \vec v \pm \vec e_i+ \vec w\rangle\rs$, then $\ls e_1, \dots, e_m\rs \cup \ls v \rs \cup \sigma$ contains the 3-additive simplex $\ls e_i, v,  w, \langle \vec v \pm \vec e_i+ \vec w\rangle\rs$, so  $\sigma$ is not a simplex in $\Linkhat_{\BA^m_n}^{<}(v)$.

Finally, \autoref{it_3add} follows because for any simplex $\sigma$ in $\Linkhat_{\BAA^m_n}^{<}(\{v,\langle \vec v \pm \vec e_i\rangle\})$, the simplex $\ls e_1, \dots, e_m\rs \cup \ls v, \langle \vec v \pm \vec e_i\rangle\ \rs \cup \sigma$ contains the $2$-additive face $\{e_i, v, \langle \vec v \pm \vec e_i\rangle\}$. As every face of a $3$-additive simplex is either standard or $3$-additive (see \autoref{obs_list_facet_types}), this implies that $\sigma$ cannot be $3$-additive.
\end{proof}

The following will be used to describe the link of 3-additive simplices during the proof that $\BAA_n^m$ is spherical.

\begin{lemma}
\label{lem_characterisation_critical_3add}
Let $\sigma$ be a 3-additive simplex in $\BAA_n^m$ and $R> 0$ the highest absolute value of the last coordinates of all of its vertices. 
As in \autoref{def_J}, write $\sigma=\ls v_0,v_1,\ldots,v_k\rs$, where $\vec v_0 = \vec w_1 + \vec w_2 + \vec w_3$, let $J(\sigma)$ be as in \autoref{def_J} and let $J^< \subseteq J(\sigma)$ be the subset of all vertices with last coordinate \emph{smaller} in absolute value than $R$.

Then $J^<$ is empty if and only if the last coordinate of $v_0$ is $\pm R$ and there are $1 \leq l \leq k$, $1\leq i \not= j \leq m$ such that $\vec v_0 = \vec v_l \pm \vec e_i \pm \vec e_j $.
\end{lemma}

\subsection{Induction beginning}

The following is an adaptation of \cite[Proof of Theorem C', Base Case]{CP}.

\begin{lemma}\label{inductionBeg}
Let $m \geq 2$. The complex $\BAA_1^m$ is $1$-connected.
\end{lemma}

\begin{proof}
We show this by successively describing the structures of $\B_1^m,\BA_1^m$ and $\BAA_1^m$. All of these complexes have the same vertex set. 
Every vertex is a line $v$ that is spanned by a vector $\vec v$ of the form $(a_1,\ldots,a_m,1)$, which we will write as $\vec v = (\vec a ,1)$ for $\vec a \in \Z^m$.  This gives an identification of the vertex set with $\Z^m$.
The complex $\B_1^m$ has dimension zero, so it has no simplices other than these vertices.

The complex $\BA_1^m$ has dimension one.  In \cite[Proof of Theorem C', Base Case]{CP}, Church--Putman show that it is isomorphic to the Cayley graph of $\Z^m$ with respect to the generating set given by $\vec e_1, \ldots, \vec e_m$:
Every edge in $\BA_1^m$ can be written in the form $\sigma = \ls v, \langle \vec v + \vec e_i \rangle \rs$ for some $v\in B_1^m$ and $1\leq i \leq m$; such an edge comes from the $2$-additive simplex $\ls v, \langle \vec v+ \vec e_i \rangle \rs \cup \ls e_1, \ldots, e_m \rs$ in $\BA_{1+m}$. 
For $\vec v = (\vec a, 1)$, this edge gets identified with the edge $\ls \vec a, \vec a + \vec e_i \rs$ of the Cayley graph. (We slightly abuse notation here by writing $\vec e_i$ both for elements in $\Z^{1+m}$ and in $\Z^{m}$.)
 
The complex $\BAA_1^m$ has dimension two. It is obtained from $\BA_1^m$ by attaching simplices $\sigma$ such that $\sigma \cup \ls e_1, \ldots, e_m \rs$ is either 3-additive or of type double-triple in $\BAA_{1+m}$. (No double-double simplices can occur in this low dimensional case.)  Concretely, the double-triple simplex in $\BAA_1^m$ are all of the form
\begin{gather*}
\sigma =   \ls v,\langle \vec v \pm \vec e_i \rangle, \langle \vec v \pm \vec e_i \pm \vec e_j \rangle\rs, \quad \sigma = \ls v, \langle \vec v \pm \vec e_j \rangle , \langle \vec v \pm \vec e_i \pm \vec e_j \rangle\rs, \\
\sigma = \ls v,\langle \vec v \pm \vec e_i \rangle,\langle \vec v\pm\vec e_j \rangle \rs, 
   \text{ or }
\sigma = \ls \langle \vec v \pm \vec e_i \rangle,\langle \vec v\pm\vec e_j \rangle , \langle \vec v \pm \vec e_i \pm \vec e_j \rangle\rs.
\end{gather*}
for some $v\in \B_1^m$ and $1 \leq i < j \leq m$. The 3-additive simplices arise as faces of these.
\autoref{BAA1-1Skeleton} shows the $1$-skeleton of $\BAA_1^m$ and its relationship to the Cayley graph of $\Z^m$ with respect to the standard generators. In fact, the $1$-skeleton of $\BAA_1^m$ is isomorphic to the Cayley graph of $\Z^m$ with respect to generators of the form $\vec e_i$ and $\vec e_i \pm \vec e_j$. 

\begin{figure}[h!]
\begin{center}
\includegraphics[scale=0.5]{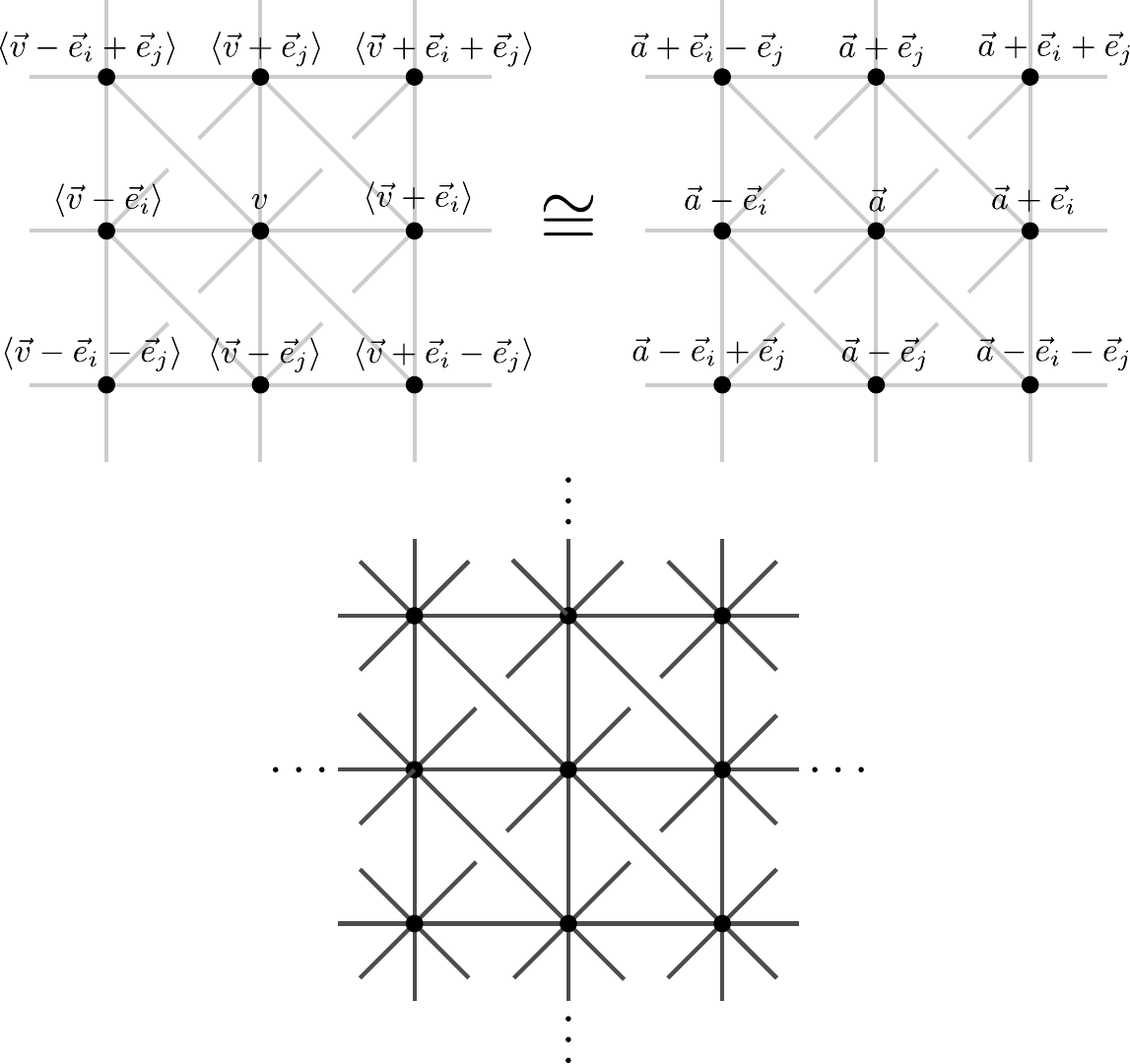}
\end{center}
\caption{We identify the 1-skeleton of $\BAA_1^m$ (left) with a graph obtained by gluing edges onto the Cayley graph of $\Z^m$ (right). Here $v$ is spanned by $\vec v = (\vec a, 1)$.}
\label{BAA1-1Skeleton}
\end{figure}

It follows that $\BAA_1^m$ is isomorphic to a complex that is obtained as follows. Start with the Cayley graph of $\Z^m$ with respect to the generating set $\vec e_1, \ldots, \vec e_m$. Every minimal cycle in this graph has length four and vertices $\vec a$, $\vec a+\vec e_i$, $\vec a+\vec e_j$, $\vec a+\vec e_i+\vec e_j$ for some $\vec a \in \Z^m$ and $1 \leq i < j \leq m$.

Now attach to each such cycle two quadrilaterals along their boundaries. Both quadrilaterals are composed of two triangles, the first one of 
\begin{equation*}
\ls\vec a, \vec a+\vec e_i, \vec a+\vec e_i+\vec e_j \rs \text{ and } \ls\vec a, \vec a+\vec e_j, \vec a+\vec e_i+\vec e_j \rs,
\end{equation*}
the second one of 
\begin{equation*}
\ls\vec a, \vec a+\vec e_i, \vec a+\vec e_j\rs \text{ and }\ls\vec a+\vec e_i, \vec a+\vec e_j, \vec a+\vec e_i+\vec e_j \rs.
\end{equation*}
See \autoref{BAA-Sphere}.
The fundamental group of this Cayley graph (with base point the identity) is generated by loops of the form $p_{\vec a}\cdot l_{\vec a,i,j} \cdot p_{\vec a}^{-1}$, where $p_{\vec a}$ is a path from the identity to $\vec a$,  $p_{\vec a}$ is its inverse and $l_{\vec a,i,j}$ is the 4-edges loop around the square  $\ls\vec a, \vec a+\vec e_i, \vec a+\vec e_j,  \vec a+\vec e_i+\vec e_j \rs$. Our complex is constructed by gluing a 2-disk (in fact, two 2-disks) to each such square, the resulting complex is 1-connected. We conclude $\BAA_1^m$ is 1-connected as claimed. 
\begin{figure}[h!]
\begin{center}
\includegraphics[scale=0.6]{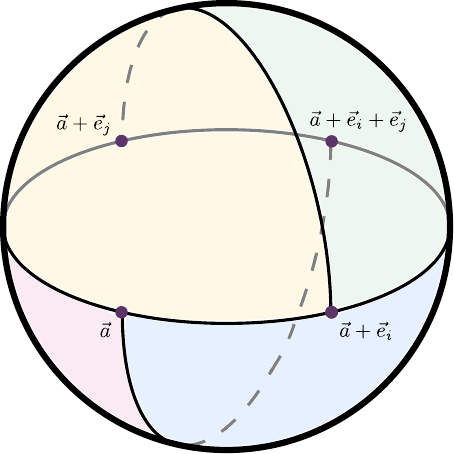}
\end{center}
\caption{The subcomplex spanned by $\ls\vec a, \vec a+\vec e_i, \vec a+\vec e_j,  \vec a+\vec e_i+\vec e_j \rs$.  }
\label{BAA-Sphere}
\end{figure}
\end{proof}


\section{Proof of \autoref{BAAnmCM}}
\label{sec_connectivity_BAA}

In this section, we will finish the proof of \autoref{BAAnmCM}, which states that $\BAA_{n}^{m}$ is Cohen--Macaulay of dimension $(n+1)$ whenever $n\geq 1$ and $m+n \geq 3$. By \autoref{prop_connecitivity_links}, to prove this it suffices to show  $\BAA_{n}^{m}$ is $n$-connected whenever $n\geq 1$ and $m+n \geq 3$.
Our proof roughly follows the strategy of Church--Putman \cite[Proof of Theorem C', Steps 1-4]{CP}. The analogue of \cite[Proof of Theorem C', Step 2]{CP} does not work in our context but fortunately it is not essential here or in \cite[Proof of Theorem C']{CP}. \hyperref[step1]{Step 1} in our proof is roughly speaking a combination of Step 1 and Step 3 of the proof of Church--Putman while our \hyperref[step2]{Step 2} corresponds to their Step 4.
\newline

Let $n \geq 1$ and $m+n \geq 3$. By \autoref{inductionBeg}, $\BAA_{1}^{m}$ is $1$-connected for all $m\geq 2$. We use this as a base case for an induction on $n$. Now assume that $n\geq 2$ and that by induction, $\BAA_{n-1}^{m+1}$ is $(n-1)$-connected.  For $d \leq n$, let $f \colon S^d \m \BAA_n^m$ be a map that is simplicial with respect to some simplicial structure on $S^d$. Here and from now on, we will assume that all simplicial structures on manifolds (possibly with boundary) are chosen to be combinatorial. This ensures that links of simplices are homeomorphic to spheres of the appropriate dimension. Let $R$ be the maximum of the absolute value of the last coordinate of $f(x)$ over all vertices $x \in S^d$. If $R=0$, then $f$ can be extended to a disk via coning its image with the vertex $e_{m+n}$. Thus, we are done if we can show that we can homotope $f$ to lower $R$. A visual outline of the proof is shown in \autoref{Flowchart}. 

This homotopy is done in two steps: In \hyperref[step1]{Step 1}, we isolate vertices in $S^d$ that get mapped to vertices with last entry $\pm R$, i.e.~we homotope $f$  such that if $x,y$ form an edge in $S^d$ and $f(x)$, $f(y)$ have last entry $\pm R$, then $f(x) = f(y)$. In \hyperref[step2]{Step 2}, we then successively replace all of these ``bad'' vertices by vertices whose last coordinate has absolute value less than $R$. Only this second step uses our inductive hypothesis. 
In order to perform these two steps, we will perform a sequence of homotopies that step-by-step replace $f$ by ``better'' maps. Before we start with these, we make some definitions that help us to keep track of the progress we make and describe a \hyperref[procedure1]{Procedure 1} that we will repeatedly use during \hyperref[step1]{Step 1}.

\begin{definition}
A simplex $\sigma$ of $S^d$ is called \emph{edgy} if $f(\sigma)=\{ v_0, v_1\}$ is an edge with the last coordinates of $v_0$ and $v_1$ equal to $\pm R$.
\end{definition}
\noindent If $f \colon S^d \m \BAA_n^m$ has no edgy simplices, then the bad vertices are isolated in the above sense. So removing all edgy simplices is the aim of \hyperref[step1]{Step 1}.

Our method for removing edgy simplices only works if we can control the stars of such simplices. 
For this, we need to make sure that there are no simplices of the following type:

\begin{definition}
A simplex $\sigma$ of $S^d$ is called \emph{$(a,b,c)$-over-augmented}, $a,b,c \in \N$, if 
\begin{itemize}
\item $f(\sigma)$ is a 3-additive, double-triple, or double-double simplex,
\item every vertex of $f(\sigma)$ either has last coordinate $\pm R$ or is contained in the additive core,
\item $\sigma$ contains exactly $a\geq 1$ vertices $x$ such that $f(x)$ has last coordinate $\pm R$,
\item $\sigma$ contains exactly $b\geq 0$ vertices $x$ such that $f(x)$ is contained in the additive core of a 3-additive face of $f(\sigma)$,
\item $\dim(\sigma) = c$, and
\item if $f(\sigma)$ is $3$-additive, then for all $v_0 \in f(\sigma) $ with last coordinate $\pm R$, there does {\bf not} exist $v_1\in f(\sigma) $  and $1 \leq i \not= j \leq m$ such that $\vec v_0 = \vec v_1 \pm \vec e_i \pm \vec e_j $. 
\end{itemize}
We call a simplex \emph{overly augmented} if it is $(a,b,c)$-over-augmented for some $a\geq 1$ and $b,c\geq 0$.
Suppose $\sigma$ is an $(a,b,c)$-over-augmented simplex and $\tau$ is a $(a',b',c')$-over-augmented simplex. We call $\tau$ \emph{better} than $\sigma$ if $(a,b,c)< (a',b',c')$ lexicographically. 
\end{definition}
Note that the last condition of the definition coincides with the one given in the last bullet point of \autoref{lem_characterisation_critical_3add}. It excludes the case of edgy simplices $\sigma$ whose image $f(\sigma) = \ls \langle \vec v_1 \pm \vec e_i \pm \vec e_j \rangle,  v_1 \rs$ is 3-additive. These are considered in detail later on (\hyperref[step1.2]{Step 1.2}).
For later reference, we record the following observation. It describes the stars of edgy simplices in the case where $f$ has no overly augmented simplices.
\begin{observation}
\label{observatio_structure_without_overly_augmented}
If $f$ has no overly augmented simplices, then the following is true: Let $\sigma$ be a simplex of $S^d$ such that $f(\sigma)$ contains two vertices with last coordinate $\pm R$. Then $f(\sigma)$ is neither a double-triple nor a double-double simplex. If it is 3-additive, it can be written in the form $f(\sigma) = \ls \langle \vec v_1 \pm \vec e_i \pm \vec e_j \rangle , v_1, v_2, \ldots, v_k \rs$, where the last coordinate of $v_1$ is $\pm R$; in particular, $\sigma$ then contains an edgy simplex with 3-additive image $\ls \langle \vec v_1 \pm \vec e_i \pm \vec e_j \rangle , v_1\rs$.
\end{observation}

\begin{figure}
\begin{center}
\includegraphics[height=\textheight]{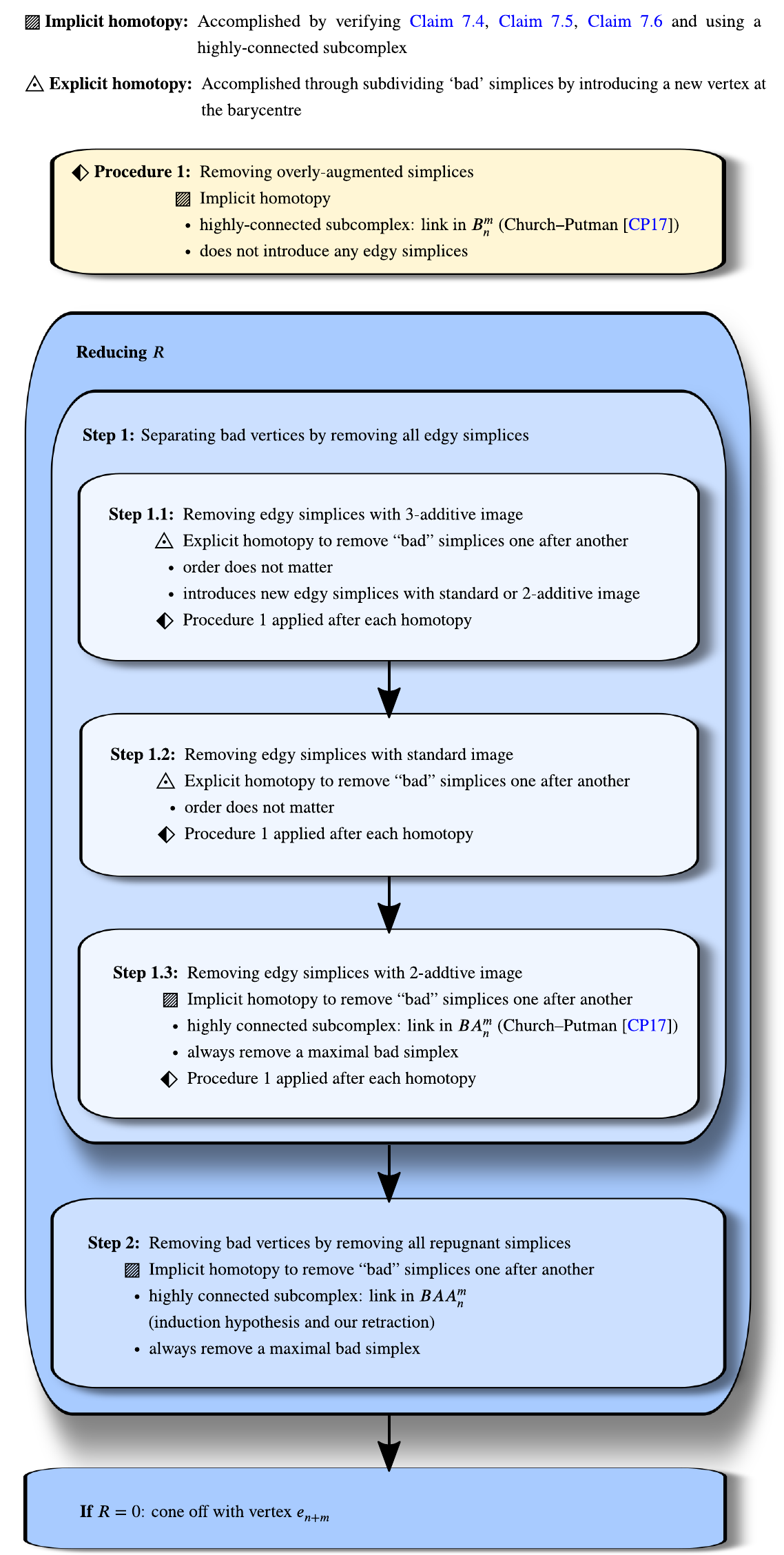}
\end{center}
\caption{A schematic of the proof of \autoref{BAAnmCM}}
\label{Flowchart}
\end{figure}

%
%

\subsection*{Procedure 1: Removing overly augmented simplices}
\label{procedure1}
We will now describe a procedure that allows us to remove overly augmented simplices from $f$. 
Let $\sigma$ be an $(a,b,c)$-over-augmented simplex with $(a,b,c)$ as large as possible lexicographically. Our goal is to homotope $f$ to have one less $(a,b,c)$-over-augmented simplex while only adding better simplices and no new edgy simplices. In order to do so, we will modify $f|_{\Star_{S^d}(\sigma)}$ such that image of the result lies in $f(\partial \sigma) * K(\sigma)$, where $K(\sigma)$ is a certain subcomplex of $\BAA_n^m$ whose vertices have more desirable properties than those of $f(\sigma)$. The same type of argument will be used several times in this article (\hyperref[step1.3]{Step 1.3}, \hyperref[step2]{Step 2}, \autoref{BAtoBAA}). We spell it out in detail here and will use this as a blueprint for later occurrences. This is a standard procedure that has been used by many authors to prove various simplicial complexes are highly-connected. This proof strategy is often called a ``bad simplex'' argument.

We start by defining $K(\sigma)$.
If $f(\sigma)$ is a double-triple or double-double simplex, we can write $f(\sigma) = \ls v_0,v_1,\ldots,v_k\rs$, where $\ls \vec v_2,\ldots, \vec v_k\rs$ is a partial basis. We define $$K(\sigma)\coloneqq \Link^{<}_{B_n^m}(\ls v_2,\ldots,v_k \rs) \qquad \qquad \text{[$f(\sigma)$ double-triple or double-double]}.$$
If $f(\sigma)$ is a 3-additive simplex, we can write $f(\sigma)=\ls v_0,v_1,\ldots,v_k\rs$ as in \autoref{def_J}, i.e.~such that $\ls \vec v_1,\ldots, \vec v_k,  \vec e_1, \dots, \vec e_m\rs$ is a partial basis and $\vec v_0 = \vec w_1 + \vec w_2 + \vec w_3$ for $w_1, w_2, w_3 \in \ls v_1,\ldots,v_k,e_1,\ldots,e_m \rs$. Let $J(\sigma)$ be the set of vertices of $\BAA_n^m$ that are lines spanned by a vector of the form $\ls \vec w_1+\vec w_2, \vec w_1+\vec w_3, \vec w_2+\vec w_3 \rs$, as in \autoref{def_J} and $J^< \subseteq J(\sigma)$ the subset of all vertices with last coordinate smaller in absolute value than $R$. By the last assumption in the definition of overly augmented simplices and \autoref{lem_characterisation_critical_3add}, the set $J^<$ is nonempty. We view $J^<$ as a 0-dimensional simplicial complex and define 
$$ K(\sigma)\coloneqq \Link^{<}_{\B_n^m}(\ls v_1,\ldots,v_k \rs)*J^< \qquad \qquad \text{[$f(\sigma)$  $3$-additive]}.$$

\begin{claim}
\label{claim_K_in_link}
$K(\sigma)$ is a subcomplex of $\Link_{\BAA_n^m}(f(\sigma))$ and $f\left(\Link_{S^d}(\sigma)\right)\subseteq K(\sigma)$.
\end{claim}

As $f$ is simplicial, we have $f(\Link_{S^d}(\sigma))\subseteq \Star_{\BAA_n^m}(f(\sigma))$. Since $\sigma$ is maximally over-augmented, every $x \in \Link_{S^d}(\sigma)$ gets mapped to a vertex $f(x)\in \Link_{\BAA_n^m}(f(\sigma))$ with last coordinate smaller in absolute value than $R$.
Hence, we actually have $f(\Link_{S^d}(\sigma))\subseteq \Link^{<}_{\BAA_n^m}(f(\sigma))$ and it suffices to show that $K(\sigma) = \Link^{<}_{\BAA_n^m}(f(\sigma))$. 
This follows immediately from \autoref{links_of_BAA_augmentations}.

\begin{claim}
\label{claim_link_highly_conn}
$K(\sigma)$ is $(\dim \Link_{S^d}(\sigma))$-connected.
\end{claim}

By the work of Church--Putman, $\Link^{<}_{B_n^m}(\ls v_0,\ldots,v_{\ell} \rs)$ is $(n-\ell-3)$-connected \cite[Theorem 4.2 and Lemma 4.5; see the first paragraph on p.~1016]{CP}.
This implies that $K(\sigma)$ is $(n-k-1)$-connected in all cases under consideration; note when $\sigma$ is 3-additive, we know $J^< \neq \emptyset$ by  \autoref{lem_characterisation_critical_3add}. The claim follows because $\dim \Link_{S^d}(\sigma) = d-\dim(\sigma)-1\leq n-k-1$.
\newline

These two claims allow us to modify $f$ up to homotopy on $\Star(\sigma)$:
By \autoref{claim_K_in_link}, $f$ restricts to a map $$\Link_{S^d}(\sigma)\to K(\sigma)$$ whose domain $\Link_{S^d}(\sigma)$ is isomorphic to a triangulated sphere. By \autoref{claim_link_highly_conn}, this map can be extended to a map 
\begin{equation*}
g\colon  \Cone ( \Link_{S^d}(\sigma) ) \m K(\sigma)
\end{equation*}
that is simplicial with respect to some simplicial structure on  $ \Cone ( \Link_{S^d}(\sigma) )$.
Again by \autoref{claim_K_in_link}, $K(\sigma)$ is a subcomplex of $\Link_{\BAA_n^m}(f(\sigma))$. This implies that $g$ extends to 
\begin{equation*}
f|_\sigma * g\colon  \sigma * \Cone ( \Link_{S^d}(\sigma) ) \to f(\sigma) * K(\sigma)\subset \BAA_n^m.
\end{equation*}
Topologically, $ \sigma * \Cone ( \Link_{S^d}(\sigma) )$ is a ball whose boundary can be decomposed as
\begin{equation*}
\partial (\sigma * \Cone ( \Link_{S^d}(\sigma) )) = (\partial \sigma * \Cone ( \Link_{S^d}(\sigma) )) \cup \Star_{S^d}(\sigma).
\end{equation*}
Note that $f|_{\Link_{S^d}(\sigma)} = g|_{\Link_{S^d}(\sigma)}$. It follows that the restriction of $f$ to $\Star_{S^d}(\sigma)$ is homotopic to a simplicial map $h\colon  \partial \sigma * \Cone ( \Link_{S^d}(\sigma) ) \to f(\sigma) * K(\sigma)$ that agrees with $f$ on $\partial \sigma * \Link_{S^d}(\sigma)$. 

\begin{claim}
\label{claim_not_worse_but_better}
The map $h\colon \partial \sigma * \Cone ( \Link_{S^d}(\sigma) ) \to f(\sigma) * K(\sigma)$ has only simplices that are better than $\sigma$.
Furthermore, every edgy simplex of $h$ is contained in $\partial \sigma$.
\end{claim}

Every simplex in $\partial \sigma * \Cone ( \Link_{S^d}(\sigma) )$ is of the form $\sigma' = \tilde{\sigma} \cup \tau$, where $\emptyset \subseteq \tilde{\sigma} \subset \sigma$ is a proper face of $\sigma$. Such a simplex gets mapped to $h(\sigma') = f(\tilde{\sigma}) \cup  g(\tau)$, where $g(\tau) \subseteq K(\sigma)$.
Observe that every vertex of $K(\sigma)$ has last entry of absolute value smaller than $R$. This implies that every edgy simplex of $h$ must be contained in $\partial \sigma$.
Now let $\sigma' = \tilde{\sigma} \cup \tau$ be a simplex in the domain of $h$ that is $(a',b',c')$-over-augmented. We need to show that $(a',b',c')<(a,b,c)$ lexicographically. That $a'\leq a$ follows immediately because every vertex of $K(\sigma)$ has last entry of absolute value smaller than $R$.

Assume that $f(\sigma)$ is a double-triple or double-double simplex. In this case, the definition of $K(\sigma)$ implies that no vertex of $g(\tau)$ can be contained in the additive core of $h(\sigma')$.
This and the assumption that $\sigma'$ is overly augmented imply firstly that $b'\leq b$
and secondly that $\tau$ is the empty simplex, i.e.~$\sigma' = \tilde{\sigma} \subset \sigma$. 
But then, as $c' = \dim(\sigma') < \dim(\sigma) = c$, we have $(a',b',c')< (a,b,c)$.

Next assume that $f(\sigma)$ is 3-additive. Here, we defined $K(\sigma) = \Link^{<}_{\B_n^m}(\ls v_1,\ldots,v_k \rs)*J^<$.
The vertices of $g(\tau)$ that are contained in $\Link^{<}_{\B_n^m}(\ls v_1,\ldots,v_k \rs)$ can neither be in the additive core of $h(\sigma')$ nor do they have last coordinate of absolute value $\pm R$. Hence, as $\sigma' = \tilde{\sigma}\cup \tau$ is overly augmented, we have $g(\tau)\cap \Link^{<}_{\B_n^m}(\ls v_1,\ldots,v_k \rs) = \emptyset$. In other words, either $\tau$ is the empty simplex and $\sigma' = \tilde{\sigma} \subset \sigma$ or $h(\sigma') = f(\tilde{\sigma}) \cup \ls j \rs$ for some $j \in J^<$ and $\tilde{\sigma}\subset \sigma$.
In the first case, we have $(a',b',c')< (a,b,c)$ for the same reasons as in the situation of double-triple or double-double simplices.
For the second case, note that although $j$ might be contained in the additive core of $h(\sigma')$, it cannot be contained in the additive core of a 3-additive face: We know that $f(\sigma)\cup \ls j \rs$ is a double-triple simplex containing $h(\sigma')$ and that $f(\sigma)$ is a 3-additive face of it. But a double-triple simplex has exactly one 3-additive face (see \autoref{obs_list_facet_types}). Hence, $b'\leq b$.
As $\sigma$ is overly augmented, every vertex of it is either mapped to a vertex with last coordinate $\pm R$ or to the additive core of the 3-additive simplex $f(\sigma)$. This implies that every vertex contributes either to $a$ or $b$ (or to both). On the other hand, the vertices of $\tau$ are mapped to $j$, which neither has rank $R$ nor is it contained in the additive core of a 3-additive face of $h(\sigma')$. It follows that these vertices neither contribute to $a'$ nor to $b'$. Consequently, we have $a'+b' < a+b$, which implies $(a',b',c')< (a,b,c)$.
\newline

We can now replace $f$ by the homotopic map $f'\colon S^d\to \BAA_n^m$ that is obtained by replacing $\Star(\sigma)$ with $\partial \sigma * \Cone ( \Link_{S^d}(\sigma) ) $ and setting $f'$ to be equal to $h$ on this subset of $S^d$. By \autoref{claim_not_worse_but_better}, the map $f'$ has one less $(a,b,c)$-over-augmented simplex than $f$, no worse simplices and no additional edgy simplices.
Iterating this shows that we may replace $f$ by a map that has no overly augmented simplices and no other edgy simplices than those of $f$.

We now proceed with the process of reducing $R$, the maximum of the absolute values of the last coordinate of vectors in the image of $f$.

\subsection*{Step 1: Separating bad vertices.}
\label{step1}
In this first step, we will remove all edgy simplices. If $\sigma$ is edgy, then its image $f(\sigma)=\ls v_0, v_1\rs$ is either a standard simplex (if $\ls \vec v_0, \vec v_1, \vec e_1, \ldots, \vec e_m \rs$ is a partial basis) or a 2-additive simplex (if $\vec v_0 = \vec v_1 \pm \vec e_i $) or a 3-additive simplex (if $\vec v_0 = \vec v_1 \pm \vec e_i \pm \vec e_j $).
We will now successively remove edgy simplices by first removing those with 3-additive image, then those with standard image and finally those with 2-additive image.
While doing so, we will repeatedly apply \hyperref[procedure1]{Procedure 1}.
\subsubsection*{Step 1.1 Removing edgy simplices with $3$-additive image}
\label{step1.1}
Let $\sigma$ be an edgy simplex such that $f(\sigma) = \ls v_0, v_1 \rs$ is 3-additive. We can find representatives $\vec v_0$ and $\vec v_1$ such that their last coordinates are equal to $R$ and $\vec v_0 = \vec v_1  \pm \vec e_i \pm \vec e_j$ for some $1\leq i\not = j \leq m$.
Define $v = \langle \vec v_1 \pm \vec e_i \rangle$, where the sign $\pm$ of $e_i$ agrees with its sign in the sum $\vec v_0 = \vec v_1  \pm \vec e_i \pm \vec e_j$.
Our aim is to use $v$ to replace $f$ by a map $f'$ that avoids the simplex $\sigma$ and has no further edgy simplices with 3-additive image than those of $f$.

Consider a simplex $\tau \supsetneq \sigma$ of $S^d$ that contains $\sigma$. Then its image $f(\tau) = \ls v_0, v_1, \ldots, v_k \rs$ contains $v_0$ and $v_1$, which have last coordinates $\pm R$.
As $f$ has no overly augmented simplices, this implies that $f(\tau)$ cannot be a double-triple or double-double simplex (see \autoref{observatio_structure_without_overly_augmented}). On the other hand, $f(\tau)$ contains the 3-additive edge $f(\sigma)$, so it must be 3-additive itself, with additive core $\ls v_0, v_1, e_i, e_j \rs$ (see \autoref{obs_list_facet_types}).
Hence, $f(\tau) \cup \ls v \rs  = \ls \langle \vec v_1  \pm \vec e_i \rangle, \langle \vec v_1  \pm \vec e_i \pm \vec e_j \rangle, v_1, \ldots, v_k \rs$ is a double-triple simplex in $\BAA_n^m$ with additive core $\ls v, v_0, v_1, e_i, e_j \rs$.
This implies that $f$ maps $\Star_{S^d}(\sigma)$ to $\Star_{\BAA_n^m}(\ls v_0,  v_1, v\rs)$.

Let $(S^d)'$ be the coarsest subdivision of $S^d$ that subdivides $\sigma$ by adding a new vertex $t$ at its barycentre. Let $f'\colon (S^d)' \m \BAA_n^m$ be the map that agrees with $f$ on vertices of $S^d$ and sends $t$ to $v$. The previous paragraph proves that $f'$ is simplicial, and $f$ and $f'$ are homotopic. The structure of $f'$ can be described as follows: To obtain $f'$ from $f$,  subdivide every simplex $\tau \supseteq \sigma$ that contains $\sigma$  into $(\dim(\sigma)+1)$-many simplices of the same dimension as $\tau$. Each such new simplex is obtained by replacing one vertex of $\sigma\subseteq \tau$ with the newly added $t$. This vertex gets mapped to  $f'(t) = v$ and $f'$ agrees with $f$ on the remaining vertices of $\tau'$.  Every simplex of $S^d$ that does not contain $\sigma$ is also a simplex in $(S^d)'$ and the maps $f$ and $f'$ agree on these simplices.

Clearly, $f'$ does not contain the edgy simplex $\sigma$ anymore. We claim that furthermore, no new edgy simplices with 3-additive image were created when passing from $f$ to $f'$. To see this, assume that $\sigma'$ is an edgy simplex of $f'$ that is not an edgy simplex of $f$. Then $\sigma'$ must contain the newly added vertex $t$ and hence is a face of some $\tau'$ that was obtained by subdividing a simplex $\tau \supseteq \sigma$. This implies that the image $f(\sigma')$ must be of one of the forms
\begin{gather*}
\ls v,v_0 \rs = \ls \langle \vec v_1 \pm \vec e_i \rangle, \langle \vec v_1 \pm \vec e_i \pm \vec e_j \rangle  \rs, \, \ls v, v_1 \rs = \ls \langle \vec v_1 \pm \vec e_i  \rangle, v_1 \rs , \text{ or}  \\
\ls v, v_l \rs \text{for some } v_l \text{ such that } \ls \vec v, \vec v_l, \vec e_1,\ldots, \vec e_m \rs \text{ is a partial basis.}
\end{gather*}
But $f(\sigma')$ is not 3-additive in any of these cases.\footnote{The subdivision created new edgy simplices with 2-additive image though, e.g.~of the form $\ls \langle \vec v_1 \pm \vec e_i  \rangle, v_1 \rs$. These will be removed in the next \hyperref[step1.2]{Step 1.2}.}

The subdivision mentioned above might have introduced new overly augmented simplices. Before we can remove another edgy simplex, we need to get rid of these simplices. To do so, we apply \hyperref[procedure1]{Procedure 1} again. This removes all overly augmented simplices without introducing new edgy simplices. Afterwards, we can remove another edgy simplex whose image is 3-additive. Iterating this procedure leads to a map in which the image of every edgy simplex is either standard or 2-additive.

\subsubsection*{Step 1.2: Removing edgy simplices with standard image}
\label{step1.2}
After the previous step, we can assume that $f$ has no edgy simplices with 3-additive image and (after possibly applying \hyperref[procedure1]{Procedure 1} again) also has no overly augmented simplices. In this step, we will also remove all edgy simplices with standard image.

Let $\sigma$ be an edgy simplex such that $f(\sigma) = \ls v_0, v_1 \rs$ is  standard. We will use a procedure that is very similar to the one described in \hyperref[step1.1]{Step 1.1} in order to replace $f$ by a map $f'$ that avoids $\sigma$. Choose representatives $\vec v_0$ and $\vec v_1$ such that their last coordinates are equal to $R$ and define $ \vec v \coloneqq \vec v_0 - \vec v_1$. Clearly, $v$ is a vertex in $\BAA^m_n$ and has last coordinate equal to $0$.

Let $\tau \supseteq \sigma$ be a simplex of $S^d$ that contains $\sigma$. Then its image $f(\tau) = \ls v_0, v_1, \ldots, v_k \rs$ contains $v_0$ and $v_1$, which have last coordinates $\pm R$.
As $f$ has no overly augmented simplices and no edgy simplices with 3-additive image, this implies that $f(\tau)$ cannot be a 3-additive, double-triple or double-double simplex (see \autoref{observatio_structure_without_overly_augmented}). Hence, it must be either standard or 2-additive. In either case, $f(\tau) \cup \ls v \rs  = \ls \langle \vec v_0 - \vec v_1 \rangle, v_0, v_1, \ldots, v_k \rs$ is a simplex in $\BAA_n^m$. Here, we use the observation that $v = \langle \vec v_0 - \vec v_1 \rangle$ might be contained in $f(\tau)$, but $\langle \vec v_0 + \vec v_1 \rangle$ cannot:  the last coordinate of $\vec v_0 + \vec v_1$ is $2R$, which would contradict the definition of $R$. 
This implies that $f$ maps $\Star_{S^d}(\sigma)$ to $\Star_{\BAA_n^m}(\ls v_0,  v_1, v\rs)$.

Let $(S^d)'$ be the coarsest subdivision of $S^d$ that subdivides $\sigma$ by adding a new vertex $t$ at its barycentre. Let $f'\colon (S^d)' \m \BAA_n^m$ be the map that agrees with $f$ on vertices of $S^d$ and sends $t$ to $v$.
By the observations of the previous paragraph, this map is simplicial and $f$ and $f'$ are homotopic. 
Just as in \hyperref[step1.1]{Step 1.1}, every edgy simplex of $f'$ is either also an edgy simplex of $f$ or it contains the vertex $t$. However, the latter is impossible here as $t$ gets mapped to the vertex $v$. This has last coordinate $0$, whereas every vertex in the image of an edgy simplex must have last coordinate $\pm R$.

It follows that $f'$ has one less edgy simplex than $f$ (namely $\sigma$, which got subdivided) and that every edgy simplex of $f'$ also forms an edgy simplex of $f$. In particular, as $f$ does not have any edgy simplex with 3-additive image, neither does $f'$.
It might be that $f'$ has overly augmented simplices\footnote{If $\sigma$ is contained in $\tau$ and $f(\tau)$ is 2-additive, it might be that the image of $f'$ contains a 3-additive simplex with a vertex that has last coordinate $\pm R$. For example if $f(\tau) = \ls v_0,v_1, \langle \vec v_1+ \vec v_2 \rangle, v_2 \rs$, then there is a simplex $\tau'$ with $f'(\tau') = \ls v_0,\langle \vec v_0- \vec v_1\rangle, \langle \vec v_1+ \vec v_2 \rangle, v_2 \rs$.}.
However, we can use \hyperref[procedure1]{Procedure 1} again to remove those without introducing new edgy simplices. Afterwards, we can remove another edgy simplex with standard image. After finitely many iterations, we obtain a map that has only edgy simplices with 2-additive image.

\subsubsection*{Step 1.3: Removing edgy simplices with 2-additive image}
\label{step1.3}
We can now assume that $f$ has no edgy simplices whose image is standard or 3-additive. After performing \hyperref[procedure1]{Procedure 1}, we can also assume that is has no overly augmented simplices.
What remains to be done for completing \hyperref[step1]{Step 1} is to remove edgy simplices with 2-additive image. Let $\sigma$ be a maximal such simplex, i.e.~$f(\sigma) = \ls v_0, v_1 \rs$ is a 2-additive simplex, the last coordinates of $v_0$ and $v_1$ are equal to $\pm R$ and if $\tau \supset \sigma$, then $f(\tau) \not = f(\sigma)$.
As $f(\sigma) = \ls v_0, v_1 \rs$ is 2-additive, we have $\vec v_0 = \vec v_1 \pm \vec e_i$ for some $1\leq i \leq m$.
Here, we cannot proceed as in the case of standard simplices (\hyperref[step1.2]{Step 1.2}), because if $\vec v_0, \vec v_1$ have last coordinate $R$, then $\vec v_0 - \vec v_1 = \pm \vec e_i$ is a not vertex in $\BAA^m_n$.
What we will do instead is to apply an argument similar to the one of \hyperref[procedure1]{Procedure 1}: We will define a complex $K(\sigma)$ and homotope $f$ such that it maps $\Star_{S^d}(\sigma)$ to $f(\partial \sigma) * K(\sigma)$.

Define $K(\sigma)\coloneqq \Linkhat_{\BA^m_n}^{<}(v_0)$. In order to perform an argument similar to \hyperref[procedure1]{Procedure 1}, we need to verify the analogues of \autoref{claim_K_in_link}, \autoref{claim_link_highly_conn} and \autoref{claim_not_worse_but_better}.

That $K(\sigma)$ is a subcomplex of $\Link_{\BAA_n^m}(f(\sigma))$ is a part of \autoref{lem_compare_linkhatBA_and_linhatBAA}. Furthermore, $f$ maps $\Link_{S^d}(\sigma)$ to $K(\sigma)$: As $f$ is simplicial, we have $ f(\Link_{S^d}(\sigma)) \subseteq \Star_{\BAA_n^m}(f(\sigma))$ and because we assumed $\sigma$ to be maximal with respect to inclusion, $f(\Link_{S^d}(\sigma)) \subseteq \Link_{\BAA_n^m}(f(\sigma))$.
Next, we show that the image of every vertex of $\Link_{S^d}(\sigma)$ has last coordinate of absolute value less than $R$. 
Assume for contradiction that there is a vertex $x\in \Link_{S^d}(\sigma)$ such that the last coordinate of $f(x)$ is $\pm R$.
As we assumed $\sigma$ to be maximal, the image of the simplex $\sigma \cup \ls x \rs$ has three vertices. Its image $f(\sigma \cup \ls x \rs)$ contains vertices with last coordinate $\pm R$ and has the 2-additive simplex $f(\sigma)$ as a (proper) face. 
As $f$ has no overly augmented simplices, this implies that $f(\sigma \cup \ls x \rs)$ is 2-additive as well (see \autoref{observatio_structure_without_overly_augmented} and \autoref{obs_list_facet_types}).
But then it has a face that is a standard edge. As all three vertices of $f(\sigma \cup \ls x \rs)$ have last coordinate $\pm R$, this shows that $f$ needs to have an edgy simplex whose image is standard. This is a contradiction to our assumption. Hence, we have $f(\Link_{S^d}(\sigma))\subseteq \Link_{\BAA_n^m}^{<}(f(\sigma))$.

By \autoref{equaility_link_linkhat}, we know that $$\Link_{\BAA^m_n}^{<}(f(\sigma)) = \Linkhat_{\BAA^m_n}^{<}(f(\sigma)) $$  and by \autoref{lem_compare_linkhatBA_and_linhatBAA}, every simplex of $\Linkhat_{\BAA^m_n}^{<}(f(\sigma))$ is either contained in $\Linkhat_{\BA^m_n}^{<}(v_0) = K(\sigma)$ or is of type double-triple. However, as $v_0$ has last coordinate $\pm R$ (as does $v_1$) and there are no overly augmented simplices, there are no double-triple simplices in $f(\Link_{S^d}(\sigma)) $ (see \autoref{observatio_structure_without_overly_augmented}). This finishes the proof of our claim that $f(\Link_{S^d}(\sigma))\subseteq K(\sigma)$.

The analogue of \autoref{claim_link_highly_conn} is to show that $K(\sigma)$ is $(\dim \Link_{S^d}(\sigma))$-connected. 
Here, we can use again a result of Church--Putman. By \cite[Section 4.5, third paragraph after Step 4 on p.~1029]{CP}, $K(\sigma) = \Linkhat_{\BA^m_n}^{<}(v_0)$ is $(n-2)$-connected. The claim follows because $\dim \Link_{S^d}(\sigma) = d - \dim (\sigma) - 1 \leq n - 1 - 1$.

As in \hyperref[procedure1]{Procedure 1}, it follows that the restriction of $f$ to $\Star_{S^d}(\sigma)$ is homotopic to a simplicial map $$h\colon  \partial \sigma * \Cone ( \Link_{S^d}(\sigma) ) \to f(\sigma) * K(\sigma)$$ that agrees with $f$ on $\partial \sigma * \Link_{S^d}(\sigma)$ and has the property that $h(\Cone(\Link_{S^d}(\sigma)))\subseteq K(\sigma).$
We next verify the analogue of \autoref{claim_not_worse_but_better}, namely that every edgy simplex of $h$ is contained in $\partial \sigma$.
This is immediate here because every vertex of $K(\sigma)$ has last coordinate of absolute value smaller than $R$. Hence, a simplex can only be edgy if $h$ maps it to $f(\sigma)$. This is only the case for simplices in $\partial \sigma$.

We can now replace $\Star(\sigma)$ with $\partial \sigma * \Cone ( \Link_{S^d}(\sigma) )$ and replace $f$ by a homotopic map $f'$ that agrees with $f$ outside $\Star(\sigma)$ and is equal to $h$ on $\partial \sigma * \Cone ( \Link_{S^d}(\sigma) )$.
As every edgy simplex of $h$ is contained in $\partial \sigma$, every edgy simplex of $f'$ is also an edgy simplex of $f$. Hence, no new edgy simplices are created when passing from $f$ to $f'$. In particular, $f'$ still has only edgy simplices whose image is 2-additive. However, $f'$ has one less of these simplices than $f$ (namely $\sigma$). 

After applying \hyperref[procedure1]{Procedure 1} again to remove overly augmented simplices, we can go on and remove another edgy simplex of the resulting map. Iterating this leads to a map that has no edgy simplices (with 2-additive, standard, or 3-additive image).

\subsection*{Step 2: Removing bad vertices.}
\label{step2}
We can now assume that $f$ has no edgy simplices. Call a simplex $\sigma$ of $S^d$ \emph{bad} if $f(\sigma)=\ls v \rs$ with the last coordinates of $v$  equal to $\pm R$.
Recall that our aim is to the replace $f$ by a map whose image has only vertices with last entries of absolute value less than $R$. Hence, we are done if we can remove all bad simplices. Let $\sigma$ be a bad simplex that is maximal with respect to inclusion among all bad simplices. 
We define $K(\sigma)\coloneqq \Link^{<}_{\BAA_n^m}(v)$ and proceed as in \hyperref[procedure1]{Procedure 1} above, verifying in the following three paragraphs the analogues of \autoref{claim_K_in_link}, \autoref{claim_link_highly_conn} and \autoref{claim_not_worse_but_better}.

First note that $f$ maps $\Link_{S^d}(\sigma)$ to $K(\sigma)$: As $f$ is simplicial and $\sigma$ is maximal among bad simplices, we have $ f(\Link_{S^d}(\sigma)) \subseteq \Link_{\BAA_n^m}(f(\sigma))$.
Assume that there was $x\in \Link_{S^d}(\sigma)$ that gets mapped to a line with last entry $\pm R$. Then, as there are no edgy simplices, we have $f(x) = v $ and $\sigma\cup \ls x \rs$ gets mapped to $\ls v \rs$. This contradicts $\sigma$ being maximal. Consequently, we have $ f(\Link_{S^d}(\sigma)) \subseteq \Link^<_{\BAA_n^m}(f(\sigma))=K(\sigma)$.

Next, we want to verify that $K(\sigma)$ is $(\dim \Link_{S^d}(\sigma))$-connected. For this, we finally use the inductive hypothesis and the retraction defined in \autoref{Sec4}: First note that by the first item of \autoref{equaility_link_linkhat}, $K(\sigma)$ actually coincides with $\Linkhat^{<}_{\BAA_n^m}(v)$. Hence, it suffices to show that this complex is $(\dim \Link_{S^d}(\sigma))=(d-\dim(\sigma)-1)$-connected. 
As noted in \autoref{links_of_BAA_augmentations}, there is an isomorphism
\begin{equation*}
\Linkhat_{\BAA_n^m}(v) \cong \BAA_{n-1}^{m+1},
\end{equation*}
so by induction, $\Linkhat_{\BAA_n^m}(v)$ is $(n-1)$-connected. By \autoref{retraction}, $\Linkhat^{<}_{\BAA_n^m}(v)$ is as highly-connected as $\Linkhat_{\BAA_n^m}(v)$ and hence is also $(n-1)$-connected. The claimed connectivity of $K(\sigma)$ now follows because $(n-1)\geq (d-\dim(\sigma)-1)$.

As in \hyperref[procedure1]{Procedure 1} and \hyperref[step1.3]{Step 1.3}, it follows that the restriction of $f$ to $\Star_{S^d}(\sigma)$ is homotopic to a simplicial map $h\colon  \partial \sigma * \Cone ( \Link_{S^d}(\sigma) ) \to f(\sigma) * K(\sigma)$ that agrees with $f$ on $\partial \sigma * \Link_{S^d}(\sigma)$ and such that $h(\Cone(\Link_{S^d}(\sigma)))\subseteq K(\sigma)$.
For the analogue of \autoref{claim_not_worse_but_better}, observe that every bad simplex of $h$ is contained in $\partial \sigma$ and that $h$ does not have any edgy simplices: This follows similarly to \hyperref[step1.3]{Step 1.3} because every vertex of $K(\sigma)$ has last coordinate of absolute value smaller than $R$.

We now replace $\Star(\sigma)$ with $\partial \sigma * \Cone ( \Link_{S^d}(\sigma) )$ and $f$ by the map $f'$ that agrees with $f$ outside $\Star(\sigma)$ and is equal to $h$ on $\partial \sigma * \Cone ( \Link_{S^d}(\sigma) )$. This removes the bad simplex $\sigma$ without introducing any new bad or edgy simplices. Iterating this, we obtain a map that has no bad simplices and hence maps every vertex of $S^d$ to a line with last entry of absolute value less than $R$.\qed

\section{Maps of posets}
\label{Sec6}
In this section, we recall Quillen's map of posets spectral sequence \cite{Quillen-Poset} and some of its corollaries. In this and the following sections, we use posets as they are closely related to simplicial complexes. In fact, to each poset $\mathbb A$, we associate a simplicial complex of chains in $\mathbb A$, i.e.\ its vertices are the elements of $\mathbb A$ and a set $\{a_0, \dots, a_p\}$ forms a $p$-simplex if it is a chain $a_0< \dots < a_p$ in $\mathbb A$. Vice versa, given a simplicial complex $X$, we denote by $\mathbb P(X)$ the poset of simplices of $X$. The associated simplicial complex to $\mathbb P(X)$ is the barycentric subdivision of $X$.

We begin by fixing some terminology concerning posets.

\begin{definition}
Let $\mathbb A$ be a poset and $a \in \mathbb A$. Define \[\operatorname{ht}(a): =\min (\{k \,  | \, \exists \, a_1<a_2 < \dots < a_k<a\}).\] We call $\operatorname{ht}(a)$ the \emph{height} of $a$.
\end{definition}

\begin{definition}
Let $\mathbb A $ be a poset and let $a \in \mathbb A $. Let $\mathbb A_{>a}$ be the subposet of $\mathbb A $ of elements $x$ with $x>a$.
\end{definition}

\begin{definition}
Let $\phi \colon \mathbb A \m \mathbb B$ be a map of posets and $b \in \mathbb B$. Let $\phi^{\leq b}$ be the subposet of $\mathbb A$ of elements $a$ with $\phi(a) \leq b$.
\end{definition}

When we speak about the homology of a poset $\mathbb A$, we mean the homology of the geometric realisation of its associated simplicial complex, which we will just refer to as the geometric realisation of $\mathbb A$. Similarly, when we say that a poset is $d$-connected, $d$-dimensional, etc., we mean its geometric realisation has this property.

We now define a more general notion of homology of posets.

\begin{definition} Let $\Ab$ denote the category of abelian groups. Let $\mathbb A$ be a poset (viewed as a category with objects the elements of $\mathbb A$ and exactly one morphism $a_1\to a_2$ if $a_1\le a_2$ and none otherwise) and let $T\colon \mathbb A \m \Ab$ be a functor. For $p \geq 0$, let $$C_p(\mathbb A;T): =\bigoplus_{a_0 < \dots<a_p} T(a_0).$$
Define maps
$$  d_i\colon C_p(\mathbb A;T) \m C_{p-1}(\mathbb A;T), \qquad (0 \leq i \leq p)$$
as follows.   For $i >0$ let $d_i$ be given by the identity map $T(a_0) \m T(a_0)$ from the summand indexed by $a_0 < \dots<a_p$ to the summand indexed by $a_0 <\dots < a_{i-1}<a_{i+1}< \dots<a_p$. Let $d_0\colon C_p(\mathbb A;T) \m C_{p-1}(\mathbb A;T)$ be given by $T(a_0 \m a_1)\colon T(a_0) \m T(a_1)$ from the summand indexed by $a_0 < \dots<a_p$ to the summand indexed by $a_1 <\dots <a_p$. Let $$d = \sum_i (-1)^{i} d_i \colon C_p(\mathbb A;T) \m C_{p-1}(\mathbb A;T).$$ Since $d \circ d=0$, these groups and maps form a chain complex which we denote by $C_*(\mathbb A;T)$. Let $$H_i(\mathbb A;T)=H_i(C_*(\mathbb A;T)). $$
\end{definition}

One of the most basic examples of a functor is the constant functor $\Z$ which sends every object to $\Z$ and every morphism to the identity map. Note that $H_i(\mathbb A;\Z)$ is isomorphic to the homology of the geometric realisation of $\mathbb{A}$. Another class of functors that we will consider is the following.

\begin{definition}
Let $\phi \colon \mathbb A \m \mathbb B$ be a map of posets. Let $H_i(\phi)\colon   \mathbb B \m \Ab$ be the functor sending $b \in \mathbb B$ to $H_i(\phi^{\leq b})$ and $b_1\le b_2$ to $H_i(\phi^{\leq b_1}) \to H_i(\phi^{\leq b_2})$ induced by the inclusion $\phi^{\leq b_1} \subset \phi^{\leq b_2}$.
\end{definition}

 The following spectral sequence is due to Quillen \cite{Quillen-Poset}.

\begin{theorem}[Quillen] \label{SSquillen}
Let $\phi \colon \mathbb A \m \mathbb B$ be a map of posets. There is a homologically graded spectral sequence: $$E^2_{pq}=H_p(\mathbb B;H_q(\phi)) \implies H_{p+q}( \mathbb A). $$

\end{theorem}

See Charney \cite[Lemma 1.3]{Charney} or \cite[Lemma 3.2]{MPWY} for a proof of the following.

\begin{lemma}\label{CharneyLemma}
Let $\mathbb A$ be a poset, let $T\colon \mathbb A \m \Ab$ be a functor, and $m \in \mathbb N$. Suppose $T(a) \cong 0$ if $\operatorname{ht}(a) \neq m$. Then there is a natural isomorphism: $$H_i(\mathbb A;T) \cong\bigoplus_{\operatorname{ht}(a)=m} \widetilde H_{i-1}(\mathbb A_{>a} ;T(a))
.$$

\end{lemma}

Here $ \widetilde H_{i-1}(\mathbb A_{>a} ;T(a))$ means the reduced homology of the geometric realisation with (untwisted) coefficients $T(a)$. This lemma gives the following corollary (see e.g. \cite[Lemma 3.7]{MPP}).

\begin{proposition} \label{AndyLemma}
Let $\phi \colon  \mathbb A \m \mathbb B$ be a map of posets and let $E^r_{p,q}$ denote the map of posets spectral sequence. Assume for some fixed $d,e,r \geq 0$, the following holds for all $V \in \mathbb B$: \begin{itemize}
\item $ \widetilde H_i(\phi^{\leq V}) \cong 0$ for all $i \not \in [\operatorname{ht}(V)+d-r,\operatorname{ht}(V)+d]$.
\item $ \widetilde H_i(\mathbb B_{>V}) \cong 0$ for all $i \neq e- \operatorname{ht}(V) -1$.
\end{itemize}
\noindent Then for all $a \geq 0$ and $b \geq 1$ satisfying $a+b \not \in [d+e-r,d+e]$, we have that $E^2_{a,b} \cong 0$.

\end{proposition}

A poset is called Cohen--Macaulay of dimension $d$ if its associated simplicial complex is Cohen--Macaulay. A map $f \colon A \m B$ is called $k$-acyclic if it induces an isomorphism on $H_i$ for $i<k$ and a surjection for $k=i$.

\begin{proposition}[{\cite[Proposition 2.3]{CP}}]
\label{mopss_CP}
Fix $m \geq 0$ and let $\phi \colon \mathbb A \to \mathbb B$ be a map of posets. Assume that $\mathbb B$ is Cohen--Macaulay of dimension $d$ and that for all $b\in \mathbb B$ and $q\not = \operatorname{ht}(b)+m$, we have $\widetilde{H}_q(\phi^{\leq b}) = 0$. Then $\phi$ is $(d+m)$-acyclic.
\end{proposition}

\section{Proof of \autoref{TheoremA} and \autoref{TheoremB}}
\label{Sec7}
The goal of this section is to prove \autoref{TheoremA}, which describes the relations among the relations in Steinberg modules and use this to prove \autoref{TheoremB}, which states that the codimension-$2$ rational homology of $\SL_n(\Z)$ vanishes for $n \geq 3$. Throughout this section, we will assume that $n\geq 3$.

For a field $\F$, we write $\bT_n(\F)$ for the poset of proper nonzero subspaces of $\F^n$. As in the introduction, the geometric realisation of this poset is the Tits building associated to $\SL_n(\F)$, denoted by $\cT_n(\F)$. It is elementary to see that  $\bT_n(\Q)$ is isomorphic to the following poset.

\begin{definition} We write $\bT_n(\Z)$ (or simply $\bT_n$) for the poset of proper nonzero  direct summands of $\Z^n$ under inclusion. We write its geometric realisation as $\cT_n(\Z)$.
\end{definition}

We prove  \autoref{TheoremA} and \autoref{TheoremB} using $n$-connectivity of $\BAA_n$. The proof here works very similarly to \cite[Proof of Theorem A and B]{CP}; we largely follow \cite[Section 3]{CP}. 

\begin{definition} For $\Lambda = \Z$ or $\Lambda = \F_p$, let $\BAA_n^\pm(\Lambda)'$ be the subcomplex of $\BAA_n^\pm(\Lambda)$ consisting of all simplices $\ls v_0,\ldots , v_k \rs$ such that $\langle \vec v_0,\ldots , \vec v_k\rangle_\Lambda$ is a proper subgroup of $\Lambda^n$. Let $\BAA_n'=\BAA_n^\pm(\Z)'$.
\end{definition}

In other words, the simplices of $\BAA_n^\pm(\Lambda)$ that are not contained in $\BAA_n^\pm(\Lambda)'$ are precisely
\begin{itemize}
\item the standard simplices of dimension $n-1$,
\item the 2-additive and 3-additive simplices of dimension $n$, and
\item the double-triple and double-double simplices of dimension $n+1$.
\end{itemize}
In particular, 
\begin{align*}
&C_k(\BAA_n,\BAA_n') &&  \text{vanishes for $k \leq (n-2)$}, &\\ 
&C_{n-1}(\BAA_n,\BAA_n') && \text{is spanned by standard simplices}, &\\ 
&C_{n}(\BAA_n,\BAA_n') && \text{is spanned by 2-additive and 3-additive simplices}, &\\ 
&C_{n+1}(\BAA_n,\BAA_n') && \text{is spanned by double-triple and double-double simplices.} &
\end{align*}
(Note that in the case $n=3$, the complex $\BAA_n$ contains no double-double simplices. So in this case, $C_{n+1}(\BAA_n,\BAA_n')$ is spanned by double-triple simplices.)

The connectivity of $\BAA_n$ gives us isomorphisms between the homology of these relative chains and the homology of $\BAA_n'$:
\begin{lemma}
\label{isos_relative_homology}
Let $n \geq 3$. There are isomorphisms
\begin{equation*}
H_{n-2}(\BAA_n') \cong H_{n-1}(\BAA_n,\BAA_n') \text{ and } H_{n-1}(\BAA_n') \cong H_{n}(\BAA_n,\BAA_n').
\end{equation*}
\end{lemma}
\begin{proof}
This follows immediately from the long exact sequence of the pair $(\BAA_n,\BAA_n')$ because $\BAA_n$ is $n$-connected (\autoref{BAAnmConn}).
\end{proof}

We use \autoref{mopss_CP} to get an explicit description of the homology of $\BAA_n'$ in high degrees.

\begin{lemma}
\label{homology_BAA'} Let $n \geq 3$. 
There are isomorphisms
\begin{equation*}
H_{n-2}(\BAA_n') \cong \St_n(\Q) \text{ and } H_{n-1}(\BAA_n') \cong  0.
\end{equation*}
\end{lemma}
\begin{proof}
Let $\mathbb P(\BAA_n')$ denote the poset of simplices of $\BAA_n'$ under inclusion, and consider the map of posets
\begin{align*}
\phi \colon \mathbb P(\BAA_n') & \longrightarrow \bT_n \\ 
\{v_0, \ldots, v_k\} & \longmapsto \langle \vec v_0, \ldots, \vec v_k\rangle.
\end{align*}
We want to apply \autoref{mopss_CP} with $m=2$. As $\bT_n$ is Cohen--Macaulay of dimension $(n-2)$, we have to verify that for every proper direct summand $\ls 0 \rs \not = V \subset \Z^n$, the fibre $\phi_{\leq V}$ has vanishing reduced homology in all degrees except $(\operatorname{ht}(V)+2)=(\rank(V)+1)$.
But we have 
\begin{equation*}
\phi^{\leq V} = \{ \sigma \in \mathbb P(\BAA_n') \; | \; \phi(\sigma) \leq V \} \cong \mathbb{P}(\BAA(V)). 
\end{equation*}
The complex $\BAA(V)$ has dimension at most\footnote{In fact, its dimension is equal to $\rank(V)+1$ if $\rank(V)\geq 3$, see the comments after \autoref{BAAnmConn}.} $\rank(V)+1$ and is $\rank(V)$-connected by \autoref{BAAnmConn}.

It follows that the map $\phi$ induces isomorphisms
\begin{equation*}
H_{n-2}(|\mathbb P(\BAA_n')|) \to H_{n-2}(\cT_n) \cong \St_n(\Q) \text{ and } H_{n-1}(|\mathbb P(\BAA_n')|) \to H_{n-1}(\cT_n) \cong 0.
\end{equation*}
The claim follows because $|\mathbb P(\BAA_n')|$ is the geometric realisation of the barycentric subdivision of $|\BAA_n'|$.
\end{proof}

\begin{proposition}
\label{partial_resolution_Steinberg}
Let $n \geq 3$. The sequence
\begin{multline*}
C_{n+1}(\BAA_n, \BAA_n') \stackrel{\partial_{n+1}}{\longrightarrow} C_{n}(\BAA_n, \BAA_n') \stackrel{\partial_{n}}{\longrightarrow}  C_{n-1}(\BAA_n, \BAA_n') \stackrel{q}{\longrightarrow} \\
\stackrel{q}{\longrightarrow} H_{n-1}(\BAA_n,\BAA_n') \cong \St_n(\Q) \longrightarrow 0
\end{multline*}
is exact.
\end{proposition}
\begin{proof}
Firstly, the map $q$ is surjective by the definition of homology, so we have exactness at $H_{n-1}(\BAA_n,\BAA_n')$. Secondly, we noted above that $C_{n-2}(\BAA_n, \BAA_n')$ is trivial. Hence, 
\begin{equation*}
H_{n-1}(\BAA_n,\BAA_n')\cong C_{n-1}(\BAA_n, \BAA_n')/ \im(\partial_{n}),
\end{equation*}
which shows that the sequence is exact at $C_{n-1}(\BAA_n, \BAA_n')$.
Lastly, exactness at $C_{n}(\BAA_n, \BAA_n')$ is equivalent to the vanishing of the homology group $H_{n}(\BAA_n, \BAA_n')$. By \autoref{isos_relative_homology}, this group is isomorphic to $H_{n-1}(\BAA_n')$, which vanishes by \autoref{homology_BAA'}.

The isomorphism $H_{n-1}(\BAA_n,\BAA_n') \cong \St_n(\Q)$ is also an immediate consequence of \autoref{isos_relative_homology} and \autoref{homology_BAA'}.
\end{proof}

This proposition implies \autoref{TheoremA}, our partial resolution of $\St_n(\Q)$.

\begin{proof}[Proof of \autoref{TheoremA}]
If $n=2$, the group $\cM_2$ is trivial and the result was shown by Church--Putman \cite[Theorem B]{CP}. If $n\geq 3$, for $i=0,1$ or $2$, the groups $\cM_i$ in the statement of \autoref{TheoremA} are isomorphic to the relative chain groups $C_{n-1+i}(\BAA_n,\BAA_n')$. The claim now follows from \autoref{partial_resolution_Steinberg}.
\end{proof}

To deduce \autoref{TheoremB} from this, we need the following well-known lemma. The proof is adapted from  Church--Putman \cite[Lemma 3.2]{CP} using Putman--Studenmund \cite[Lemma 2.2]{PutStu}. Also see Putman--Studenmund \cite[Lemma 2.3]{PutStu}.
 
 \begin{lemma}\label{projectiveGeneral}
 Let $G$ be a group and let $Y \hookrightarrow X$ be an inclusion of $G$-simplicial complexes. Assume that the setwise stabiliser subgroup of every $k$-simplex of $X$ that is not contained in the image of $Y$ is finite. Let $R$ be a ring such that the orders of all stabiliser groups of such simplices are invertible in $R$.  Then $C_k(X,Y)$ is a projective $R[G]$-module.
 
 \end{lemma}

\begin{proof}
Let $\sigma$ be a $k$-simplex of $X$ not contained in $Y$ and pick an orientation on $\sigma$. Let $G_\sigma \subset G$ be the stabiliser of $\sigma$. By abuse of notation, also view $\sigma$ as an element of $C_k(X,Y;R)$. Let $M_\sigma \subset C_k(X,Y;R)$ be the $R[G]$-submodule generated by $\sigma$. Let $R_\sigma$ be the $R[G_\sigma]$-module whose underlying $R$-module is just $R$ but an element of $G$ acts by $\pm 1$ depending on whether it reverses the orientation on $\sigma$ or not. As in  Church--Putman \cite[Lemma 3.2]{CP}, we have that \[M_\sigma \cong \Ind_{G_\sigma}^G R_\sigma.\] Putman--Studenmund \cite[Lemma 2.2]{PutStu} states that $R_\sigma$ is a projective $R[G_\sigma]$-module and hence a summand of a free $R[G_\sigma]$-module. Since \[\Ind_{G_\sigma}^G  R[G_\sigma] \cong R[G], \] it follows that $M_\sigma$ is a summand of a free $R[G]$-module and hence $M_\sigma$ is projective. Since $C_k(X,Y;R)$ is a direct sum of modules of the form $M_\sigma$, the module $C_k(X,Y;R)$ is also projective.
\end{proof}

\begin{lemma}
\label{projectiveness}
Let $R$ be a ring and let $\Gamma$ be a subgroup of $\SL_n(\Z)$. Assume that for any $g \in \Gamma$ of finite order $j<\infty$, the element $j$ is a unit in $R$. Then $C_{k}(\BAA_n, \BAA_n'; R)$ is projective as an $R[\Gamma]$-module. 
\end{lemma}
\begin{proof}
Note that the groups $C_{k}(\BAA_n, \BAA_n'; R)$ vanish unless $k\in \ls n-1,n, n+1 \rs$, so we shall restrict attention to those cases.

In order to apply \autoref{projectiveGeneral}, we will first show that for every $k\in \ls n-1,n, n+1 \rs$ and for every $k$-simplex $\sigma$ of $\BAA_n$ that is not contained in $\BAA_n'$, the setwise stabiliser $\SL_n(\Z)_\sigma$ of $\sigma$ under the action of $\SL_n(\Z)$ is finite. 
Let $\sigma = \ls v_0, \ldots, v_{k} \rs$ be such a simplex. Then by definition, we have $\langle \vec v_0, \ldots, \vec v_k\rangle =\Z^n$ and we can assume that $\ls \vec v_0, \ldots, \vec v_{n-1} \rs$ is a basis.
An element $\phi\in \SL_n(\Z)$ that stabilises $\sigma$ induces a signed permutation of the set 
$\ls \vec v_0, \ldots, \vec v_{k}\rs$. Furthermore, any such $\phi$ is uniquely determined by the images of $\vec v_0, \ldots, \vec v_{n-1}$ because these form a basis of $\Z^n$. It follows that $\SL_n(\Z)_\sigma$ is a subgroup of the group of signed permutations of a set with $k+1$ elements. This is the Coxeter group of type $\mathtt{B}_{k+1}$, a finite group of order $2^{k+1}\cdot (k+1)!$.

This implies that the stabiliser $\Gamma_\sigma$ is finite and by assumption, the orders of all its elements are invertible in $R$. It follows from Cauchy's Theorem that the order of $\Gamma_\sigma$ is invertible in $R$ as well, so we can apply \autoref{projectiveGeneral}.
\end{proof}

\begin{remark}
In particular, this implies that $C_{k}(\BAA_n, \BAA_n'; \Q)$ is projective as a $\Q[\SL_n(\Z)]$-module and $C_{k}(\BAA_n, \BAA_n'; \Z)$ is projective as a $\Z[\Gamma]$-module if $\Gamma$ is torsion-free.
\end{remark}

We use the concrete description of $C_{n+1}(\BAA_n, \BAA_n')$ in terms of double-double and double-triple simplices to show the following.

\begin{lemma}
\label{vanishing_coinvariants}
The group  $C_{n+1}(\BAA_n, \BAA_n';\Q)\otimes_{\SL_n(\Z)} \Q$ vanishes for all $n \geq 3$.
\end{lemma}
\begin{proof}
The group $C_{n+1}(\BAA_n, \BAA_n';\Q) $ is generated by all oriented $(n+1)$-dimensional double-triple and double-double simplices of $\BAA_n$. (As noted above, the double-double simplices only occur if $n\geq 4$.) Let $\sigma = \{v_0, \ldots, v_{n+1} \}$ be such an $(n+1)$-simplex, where $\vec v_2, \ldots , \vec v_{n+1}$ is a basis of $\Z^n$. We need to show that for any $q\in \Q$, the element $\sigma\otimes q$ is trivial in $C_{n+1}(\BAA_n, \BAA_n';\Q)\otimes_{\SL_n(\Z)} \Q$.
There are two cases to consider.

First suppose that $\sigma$ is a double-double simplex. Then for suitable choices of signs, we have  $\vec v_0 = \vec v_2 + \vec v_3 $ and $\vec v_1 = \vec v_4 + \vec v_5$. Let $ \phi \colon \Z^n \to \Z^n$ be the automorphism defined by
\begin{equation*}
	\phi(\vec v_2) = \vec v_4, \, \phi(\vec v_4) = \vec v_2,  \, \phi(\vec v_3) = \vec v_5,\, \phi(\vec v_5) = \vec v_3 , \, \phi(\vec v_i) = \vec v_i \text{ for } i>5.
\end{equation*}
The automorphism $\phi$ is contained in $\SL_n(\Z)$ because it acts as an even permutation on the basis $\vec v_2, \ldots , \vec v_{n+1}$.
On the other hand, it acts as an odd permutation on the vertices of $\sigma$, as
\begin{equation*}
\phi((v_0, \ldots, v_{n+1})) = (v_1, v_0, v_4, v_5, v_2, v_3, v_6, \ldots, v_{n+1}).
\end{equation*}
Hence, in $C_{n+1}(\BAA_n, \BAA_n';\Q)\otimes_{\SL_n(\Z)} \Q$, we have $\sigma\otimes q = \phi(\sigma)\otimes q = -\sigma\otimes q$ for any $q\in \Q$. This implies that $\sigma\otimes q$ is trivial.

Next suppose that $\sigma$ is a double-triple simplex. In this case, we can choose signs such that $\vec v_0 = \vec v_2 + \vec v_3 $ and $\vec v_1 = \vec v_2 + \vec v_4$. We define $\psi\colon \Z^n \to \Z^n$ as the automorphism given by
\begin{equation*}
	\psi(\vec v_2) = - \vec v_2, \, \psi(\vec v_3) = \vec v_2+ \vec v_3,  \, \psi(\vec v_4) = - \vec v_4,\, \psi(\vec v_i) = \vec v_i \text{ for } i>4.
\end{equation*}
It is easy to see that $\psi$ has determinant $1$ and hence is contained in $\SL_n(\Z)$.
Noting that $\psi(\vec v_0) = \psi(\vec v_2 + \vec v_3) = \vec v_3$ and $\psi(\vec v_1) = \psi(\vec v_2 + \vec v_4) = -\vec v_1$, one sees that $\psi$ acts as an odd permutation on the vertices of $\sigma$, namely
\begin{equation*}
\psi((v_0, \ldots, v_{n+1})) = (v_3, v_1, v_2, v_0, v_4, v_5, \ldots, v_{n+1}).
\end{equation*}
As before, it follows that $\sigma\otimes q$ is trivial.
\end{proof}

That the codimension-$2$ rational cohomology of $\SL_n(\Z)$ vanishes for $n \geq 3$ is an easy consequence of the previous results:

\begin{proof}[Proof of \autoref{TheoremB}]
Because of Borel--Serre duality (see \autoref{Borel-Serre}), it is sufficient to show that $H_{2}(\SL_n(\Z); \St_n(\Q) \otimes \Q)$ is trivial. \autoref{partial_resolution_Steinberg} and \autoref{projectiveness} give us a partial projective resolution of $\St_n(\Q) \otimes \Q$  as follows:
\begin{equation*}
C_{n+1}(\BAA_n, \BAA_n';\Q) \longrightarrow C_{n}(\BAA_n, \BAA_n';\Q) \longrightarrow C_{n-1}(\BAA_n, \BAA_n';\Q) \longrightarrow \St_n(\Q) \otimes \Q \longrightarrow 0.
\end{equation*}
As this partial resolution can be extended to a projective resolution, it suffices to show that the second homology of the chain complex
\begin{multline*}
\cdots \longrightarrow C_{n+1}(\BAA_n, \BAA_n';\Q) \otimes_{\SL_n(\Z)} \Q \longrightarrow C_{n}(\BAA_n, \BAA_n';\Q) \otimes_{\SL_n(\Z)} \Q \longrightarrow \\
\longrightarrow C_{n-1}(\BAA_n, \BAA_n';\Q) \otimes_{\SL_n(\Z)} \Q \longrightarrow 0
\end{multline*}
vanishes. This is an immediate consequence of \autoref{vanishing_coinvariants}.
\end{proof}

\begin{remark}
Church--Putman \cite[Theorem A]{CP} also proved a vanishing result for the codimension-$1$ cohomology of $\SL_n(\Z)$ with coefficients in rational representations of $\GL_n(\Q)$. The analogous result is true for the codimension-$2$ cohomology.
\end{remark}

\section{Proof of \autoref{TheoremC}}

\label{Sec8}
We now shift attention to the codimension-$1$ cohomology congruence subgroups and prove \autoref{TheoremC}. 

\subsection{Relevant simplicial complexes and connectivity results}

We will deduce our results about congruence subgroups by studying connectivity properties of $\BAA_n^\pm(\F_p)$.  To prove the following result, it is not difficult to adapt the proofs of  \cite[Lemmas 2.35 and 2.43]{MPP}. 

\begin{proposition} \label{quotientBAA}
 For $p$ an odd prime, $\BAA_n /\Gamma_n(p) \cong \BAA_n^\pm(\F_p)$ and $\BAA_n' /\Gamma_n(p) \cong \BAA_n^\pm(\F_p)'$. 
\end{proposition}

An argument identical to \cite[Lemma 3.23]{MPP} gives the following corollary.

\begin{proposition} \label{RelHomologyGamma}
Let $p$ be an odd prime. There is a natural isomorphism  \[ H_1(\Gamma_n(p);\St_n(\Q) ) \cong H_{n}(\BAA_n^\pm(\F_p),\BAA_n^\pm(\F_p)').\]
\end{proposition}

\begin{proof}
\autoref{partial_resolution_Steinberg} states that 
\[C_{n+1}(\BAA_n,\BAA_n') \xrightarrow{\partial_{n+1}} C_{n}(\BAA_n,\BAA_n')\xrightarrow{\partial_{n}}   C_{n-1}(\BAA_n,\BAA_n') \m \St_n(\Q)  \to 0 \] is exact. 
Note that for $p$ odd $\Gamma_n(p)$ is torsion-free. Thus \autoref{projectiveness} implies that the groups $C_k(\BAA_n,\BAA_n')$ are projective $\Z[\Gamma_n(p)]$-modules and hence that this is a partial projective resolution of $\Z[\Gamma_n(p)]$-modules. Therefore, $H_1(\Gamma_n(p);\St_n(\Q) )$ is the homology 
of the sequence \[C_{n+1}(\BAA_n,\BAA_n')_{\Gamma_n(p)} \m  C_{n}(\BAA_n,\BAA_n')_{\Gamma_n(p)} \m C_{n-1}(\BAA_n,\BAA_n')_{\Gamma_n(p)}. \] This sequence agrees with 
\begin{multline*}
C_{n+1}(\BAA_n/\Gamma_n(p),\BAA_n'/\Gamma_n(p)) \m  C_{n}(\BAA_n/\Gamma_n(p),\BAA_n'/\Gamma_n(p)) \m \\
\m  C_{n-1}(\BAA_n/\Gamma_n(p),\BAA_n'/\Gamma_n(p)). 
\end{multline*}
Using \autoref{quotientBAA}, this is exactly \[C_{n+1}(\BAA_n^\pm(\F_p)',\BAA_n^\pm(\F_p)') \m  C_{n}(\BAA_n^\pm(\F_p)',\BAA_n^\pm(\F_p)') \m C_{n-1}(\BAA_n^\pm(\F_p)',\BAA_n^\pm(\F_p)'). \] Thus, the homology of this sequence is $H_{n}(\BAA_n^\pm(\F_p),\BAA_n^\pm(\F_p)')$.
\end{proof}

\begin{proposition}\label{BAtoBAA}
For all $p$, the inclusion $\BA_n^\pm(\F_p) \m \BAA_n^\pm(\F_p)$ induces a surjective map on $\pi_d$, $d\leq n$.
\end{proposition}

\begin{proof}
Fix $d \leq n$ and let $f \colon  S^d \m  \BAA_n^\pm(\F_p)$ be a map that is simplicial with respect to some simplicial structure on $S^d$.  It suffices to show that $f$ is homotopic to a map $\tilde{f} \colon S^d \m  \BAA_n^\pm(\F_p)$ that factors through the inclusion $\BA_n^\pm(\F_p) \hookrightarrow \BAA_n^\pm(\F_p)$.

This can be shown similarly to \hyperref[procedure1]{Procedure 1}, which was used in \autoref{sec_connectivity_BAA} to show that $\BAA_n^m$ is highly-connected. We will follow this procedure very closely and keep the notation as similar as possible to make it easier to follow. We first define the ``bad'' simplices that we want to remove:
A simplex $\sigma$ of $S^d$ is called \emph{$(b,c)$-over-augmented}, $b,c \in \N$, if 
\begin{itemize}
\item $f(\sigma)$ is a 3-additive, double-triple, or double-double simplex,
\item every vertex of $f(\sigma)$ is contained in the additive core,
\item $\sigma$ has exactly $b\geq 0$ vertices $x$ such that $f(x)$ is contained in the additive core of a 3-additive face of $f(\sigma)$,
\item $\dim(\sigma) = c$.
\end{itemize}
We call a simplex \emph{overly augmented} if it is $(b,c)$-over-augmented for some $b,c\geq 0$.
We say that a $(b,c)$-over-augmented simplex $\sigma$ is \emph{better} than a $(b',c')$-over-augmented simplex $\tau$ if $(b,c)< (b',c')$ lexicographically. 
If $f$ has no overly augmented simplices, then its image lies in $\BA_n^\pm(\F_p)$. To obtain such a map, we will successively replace $f$ with homotopic maps that have less such simplices that are maximally over-augmented.

Let $\sigma$ be a $(b,c)$-over-augmented simplex with $(b,c)$ as large as possible lexicographically. We want to remove $\sigma$ from $f$. To do so, we first define, just as in \hyperref[procedure1]{Procedure 1}, a complex $K(\sigma)$ and then verify adapted versions of \autoref{claim_K_in_link}, \autoref{claim_link_highly_conn} and \autoref{claim_not_worse_but_better}.

If $f(\sigma)$ is a double-triple or double-double simplex, it can be written as $\ls v_0, v_1, \ldots v_k \rs$, where $\ls \vec v_2, \ldots, \vec v_k\rs$ is a partial basis. As we assumed that every vertex of $f(\sigma)$ is contained in the additive core, we here have $k=4$ for a double-triple and $k=5$ for a double-double simplex.
 Define $K(\sigma)\coloneqq \Link_{\B_n^\pm(\F_p)}(\ls v_2, \ldots, v_k\rs)$. 
If $f(\sigma)$ is 3-additive, we can write $f(\sigma) = \ls \langle \vec v_1 + \vec v_2 + \vec v_3 \rangle, v_1, v_2, v_3 \rs$, where $\ls \vec v_1, \vec v_2, \vec v_3\rs$ is a partial basis. Let $J\coloneqq \ls \langle \vec v_1 + \vec v_2 \rangle, \langle \vec v_1 + \vec v_3 \rangle, \langle \vec v_2 + \vec v_3 \rangle\rs$. Note that all elements of $J$ are vertices of $\BAA_n^\pm(\F_p)$. We view $J$ as a 0-dimensional simplicial complex and define $K(\sigma) \coloneqq \Link_{\B_n^\pm(\F_p)}(\ls v_1, v_2, v_3 \rs) \ast J$.

As $f$ is simplicial and $\sigma$ is maximally over-augmented, we have $f(\Link_{S^d}(\sigma)) \subseteq \Link_{\BAA_n^\pm(\F_p)}(f(\sigma))$. So to prove the analogue of \autoref{claim_K_in_link}, it suffices to see that $K(\sigma) = \Link_{\BAA_n^\pm(\F_p)}(f(\sigma))$. This can be checked easily just as in \hyperref[procedure1]{Procedure 1}. In \autoref{prop_connecitivity_links}, we describe the links of simplices in $\BAA_n$ and an analogous statement is true for $\BAA_n^\pm(\F_p)$.

To see that $K(\sigma)$ is  $\dim \Link_{S^d}(\sigma)$-connected, note that by a result of Miller--Patzt--Putman the complex $\Link_{\B_n^\pm(\F_p)}(\ls v_0,\ldots,v_l \rs)$ is $(n-l-3)$-connected \cite[Proposition 2.45]{MPP}. Hence, $K(\sigma)$ is $(n-5)$-connected if $f(\sigma)$ is a double-triple simplex, $(n-6)$-connected if $f(\sigma)$ is a double-double simplex and $(n-5+1)=(n-4)$-connected if $f(\sigma)$ is 3-additive. We have $\dim \Link_{S^d}(\sigma) \leq n - \dim(f(\sigma)) - 1$. The claim follows because $f(\sigma)$ is a double-triple, double-double or 3-additive simplex with all vertices contained in the additive core and hence has dimension 4, 5 or 3, respectively.

Consequently, the restriction $f|_{\Star_{S^d}(\sigma)}$ is homotopic to a simplicial map $h \colon  \partial \sigma * \Cone ( \Link_{S^d}(\sigma) ) \to f(\sigma) * K(\sigma)$ that agrees with $f$ on $\partial \sigma * \Link_{S^d}(\sigma)$ and such that $h(\Cone(\Link_{S^d}(\sigma)))\subseteq K(\sigma)$.
We will now verify that $h$ has only simplices that are better than $\sigma$. This is very similar to the proof of \autoref{claim_not_worse_but_better} in \hyperref[procedure1]{Procedure 1}, so we will be slightly briefer here:

Every simplex in $\partial \sigma * \Cone ( \Link_{S^d}(\sigma) )$ is of the form $\tilde{\sigma} \cup \tau$, where $\emptyset \subseteq \tilde{\sigma} \subset \sigma$ is a proper face of $\sigma$. It gets mapped to $h(\sigma') = f(\tilde{\sigma}) \cup  g(\tau)$, where $g(\tau) \subseteq K(\sigma)$.
Let $\sigma' = \tilde{\sigma} \cup \tau$ be a simplex in the domain of $h$ that is $(b',c')$-over-augmented. We need to show that $(b',c')<(b,c)$ lexicographically.
If $f(\sigma)$ is a double-triple or double-double simplex, no vertex of $g(\tau)\subseteq K(\sigma)$ can be contained in the additive core of $h(\sigma')$. As $\sigma'$ is overly augmented, this implies that $b'\leq b$
and that $\tau$ is the empty simplex.
Hence, $c' = \dim(\sigma') < \dim(\sigma) = c$ and we have $(b',c')< (b,c)$.
Next assume that $f(\sigma)$ is 3-additive. In this case, $K(\sigma) \coloneqq \Link_{\B_n^\pm(\F_p)}(\ls v_1, v_2, v_3 \rs) \ast J$.
As $\sigma'$ is overly augmented and no vertex of $\Link_{\B_n^\pm(\F_p)}(\ls v_1, v_2, v_3 \rs)$ can be contained in the additive core of $h(\sigma')$, all vertices of $\tau$ get mapped to $J$. This means that either $\tau$ is the empty simplex and $\sigma' = \tilde{\sigma} \subset \sigma$ or $h(\sigma') = f(\tilde{\sigma}) \cup \ls j \rs$ for some $j \in J$ and $\tilde{\sigma}\subset \sigma$.
In the first case, we have $(b',c')< (b,c)$ with the same argument as in the situation of double-triple or double-double simplices.
In the second case, $f(\sigma)\cup \ls j \rs$ is a double-triple simplex that contains $h(\sigma')$ and has $f(\sigma)$ as its unique 3-additive face. Hence, $j$ cannot be contained in the additive core of a 3-additive face of $h(\sigma')$.
It follows that $b' \leq \dim(\tilde{\sigma})< \dim(\sigma) = b$. In particular, $(b',c')< (b,c)$.

We now replace $\Star(\sigma)$ with $\partial \sigma * \Cone ( \Link_{S^d}(\sigma) )$ and $f$ by the map $f'$ that agrees with $f$ outside $\Star(\sigma)$ and is equal to $h$ on $\partial \sigma * \Cone ( \Link_{S^d}(\sigma) )$. This removes $\sigma$ without introducing any other simplices that are $(b,c)$-over-augmented or worse. Iterating this procedure, we obtain a map that has no overly augmented simplices and hence factors through $\BA_n^\pm(\F_p)$.
\end{proof}
 
\begin{corollary} \label{BAApmConn}
For $p=3$ or $5$, the complex $\BAA_n^\pm(\F_p)$ is $(n-1)$-connected.
\end{corollary}

\begin{proof}
By \autoref{BAtoBAA}, there is a surjection $\pi_d(\BA_n^\pm(\F_p)) \twoheadrightarrow \pi_d(\BAA_n^\pm(\F_p))$ for $d \leq n$. The claim follows because by \cite[Proposition 2.50]{MPP}, the complex $\BA_n^\pm(\F_p)$ is $(n-1)$-connected for $p=3$ or $5$.
\end{proof}

\begin{corollary} \label{surjCor}
For $p=3$ or $5$, there is a surjection $H_1(\Gamma_n(p);\St_n(\Q) ) \m H_{n-1}(\BAA_n^\pm(\F_p)')$.
\end{corollary}
\begin{proof}
This follows from the long exact sequence of the pair $(\BAA_n^\pm(\F_p),\BAA_n^\pm(\F_p)')$, \autoref{RelHomologyGamma} and \autoref{BAApmConn}.
\end{proof}

\begin{proposition} \label{H2}
For $p$ an odd prime,  $H_2(\BAA_2^\pm(\F_p)) \cong \Z$.
\end{proposition}

\begin{proof}
Note that the inclusion $\BA_2^\pm(\F_p) \m \BAA_2^\pm(\F_p)$ is an isomorphism. The claim follows from \cite[Lemma 2.44]{MPP} which identifies $\BA_2^\pm(\F_p)$ with the compactified modular curve of level $p$, a compact surface of genus $(p+2)(p-3)(p-5)/24$. \end{proof}

\begin{remark}
\autoref{H2} shows that $\BAA_n^\pm(\F_p)$ is not always $n$-connected. This may come as a surprise. This fact is not just due to our restriction that the determinant of bases be equal to $\pm 1$ as this condition is vacuous for  $p=3$. If the reader is interested in defining a complex $\BAA_n(\F)$ to determine the relations among the relations in $\St_n(\F)$ for $\F$ an arbitrary field, we expect that extra types of additive simplices will be needed. For example, simplices of the form $\{v_1,v_2,\langle a \vec v_1 + b \vec v_2 \rangle, \langle c \vec v_1 + d \vec v_2 \rangle  \}$ with $ad-bc \neq \pm 1$ might be needed in the $n=2$ case.

\end{remark}

We now describe a model for $\cT_n(\Q)/\Gamma_n(p)$.

\begin{definition}
A \emph{$\pm$-orientation} on a rank $k$ submodule $V \subset \F_p^n$ is an equivalence class of generators of $\wedge^k  V \cong \F_p$ up to sign. 

Let $Gr_k^n(\F_p)^\pm$ denote the set of $\pm$-oriented summands of rank $k$ in $\F_p^n$. We let $\bT_n^\pm(\F_p)$ denote the poset whose elements are all proper nonzero $\pm$-oriented summands of $\F_p^n$ with order induced by proper inclusion. Let $\cT_n^\pm(\F_p)$ denote its geometric realisation.
\end{definition}

Note that the $\pm$-orientations play no role in deciding if summands of different ranks are comparable and differently oriented subspaces of the same rank are never comparable.  The following results are due to Miller--Patzt--Putman \cite{MPP}. 

\begin{proposition}[{\cite[Proposition 3.16]{MPP}}]
For $p$ an odd prime, the natural map $\cT_n(\Z)/\Gamma_n(p) \m \cT_n^{\pm}(\F_p)$ is an isomorphism.
\end{proposition}

\begin{proposition}[{\cite[Lemma 3.15]{MPP}}] \label{Tconn}
For all $p$, the complex $\cT_n^\pm(\F_p)$ is  Cohen--Macaulay of dimension $n-2$.
\end{proposition}

\subsection{Lower bounds on the codimension-$1$ cohomology of certain congruence subgroups}

In this subsection, we use the map-of-poset spectral sequence and the fact that $\BAA_2^\pm(\F_p)$ is not $2$-connected to produce cohomology classes in the codimension-$1$ cohomology of level $3$ and $5$ congruence subgroups.

The following is a categorified version of \autoref{TheoremC}.

\begin{theorem} \label{CatC}

For $p=3$ or $5$ and $n \geq 3$, $H^{{n \choose 2}-1}(\Gamma_n(p)) $ surjects onto  \[ \Z[Gr_2^n(\F_p)^\pm] \otimes \widetilde H_{n-4}(\cT_{n-2}^\pm(\F_p)).\] 
\end{theorem}

\begin{proof}
Since $\Gamma_n(p)$ is torsion-free for $p$ an odd prime, Borel--Serre duality holds with integral coefficients. In particular, \[ H^{{n \choose 2}-1}(\Gamma_n(p)) \cong H_1(\Gamma_n(p);\St_n(\Q)) .\]  Thus, by \autoref{surjCor}, it suffices to produce a surjection \[H_{n-1}(\BAA_n^\pm(\F_p)')  \m \Z[Gr_2^n(\F_p)^\pm] \otimes \widetilde H_{n-4}(\mathcal T_{n-2}^\pm(\F_p)).\] 

Let $\phi \colon\mathbb P(\BAA_n^\pm(\F_p)') \m \mathbb T_n^\pm(\F_p)$ be the map sending a simplex $\sigma = \ls v_0,\ldots,v_k \rs$ to $\langle \vec v_0,\ldots, \vec v_k\rangle_{\F_p}$ with the orientation given by $\vec v_{i_0} \wedge \vec v_{i_1} \wedge \ldots \wedge \vec v_{i_m}$, where $\ls \vec v_{i_0} , \vec v_{i_1} , \ldots , \vec v_{i_m} \rs$ is any maximal partial basis contained in $\sigma$. Observe that this orientation does not depend on the choice of the maximal partial basis in $\sigma$: For example, if $\sigma=\ls \vec v_1+\vec v_2,\vec v_1,\vec v_2,\ldots,\vec v_k\rs$ is $2$-additive, then 
\begin{equation*}
\vec v_1 \wedge \vec v_2 \wedge \ldots \wedge \vec v_k = (\vec v_1+\vec v_2 ) \wedge \vec v_2 \wedge \vec v_3 \wedge \ldots \wedge \vec v_k= (\vec v_1+\vec v_2 ) \wedge \vec v_1 \wedge \vec v_3 \wedge \ldots \wedge \vec v_k.
\end{equation*}
Reordering the vectors of the partial basis introduces a sign but does not change the equivalence class of the orientation. 

 Let $E^r_{a,b}$ denote the associated map-of-poset spectral sequence associated to $\phi$ described in \autoref{SSquillen}. For $V$ a proper nonzero rank $k$ $\pm$-summand of $\F_p^n$, note that $\operatorname{ht}(V)=k-1$, $\phi^{\leq V} \cong \mathbb P(\BAA_k^\pm(\F_p))$, and $\mathbb T_n^\pm(\F_p)_{>V} \cong \mathbb T_{n-k}^\pm(\F_p)$. Applying \autoref{AndyLemma} with $e=n-2$, $d=2$, and $r=1$, we find that for $b \geq 1$, $E^2_{a,b} \cong 0$ unless $a+b=n-1$ or $a+b=n$. See \autoref{MPSS}. 
  \begin{figure}[h!]    \hspace{-0cm} \hspace{-1.2cm} 
\begin{center}  \begin{tikzpicture} \scriptsize
  \matrix (m) [matrix of math nodes,
    nodes in empty cells,nodes={minimum width=3ex,
    minimum height=3ex,outer sep=2pt},
 column sep=3ex,row sep=3ex]{  
  8    & 0 &  &  & & &  & &&  \\   
  7    &\bigstar & 0 && &  &  & &&  \\ 
 6    &  \bigstar  & \bigstar & 0 &  &  &   &&&  \\   
 5    &  0 &  \bigstar  & \bigstar & 0 & &   &&& \\  
4   & 0  & 0 & \bigstar  & \bigstar & 0 &  & && \\            
3    &  0  & 0 & 0 & \bigstar & \bigstar & 0 &&&  \\        
2     & 0 & 0 & 0 & 0&|[draw=red, circle]|\bigstar& \bigstar & 0 && \\       
 1     & 0 & 0 & 0 & 0&0&\clubsuit & 0& 0 & \\      
   0     & \bigstar & \bigstar & \bigstar & \bigstar&\bigstar&\bigstar&0& 0 & 0  \\      
 \quad\strut &     0  &  1  & 2  & 3 & 4 &5 &6 & 7 & 8  \\ &&&&&&&&&& \\}; 
 \draw[thick] (m-1-1.east) -- (m-10-1.east) ;
 \draw[thick] (m-10-1.north) -- (m-10-10.north east) ;
 
  \draw[-stealth, blue]  (m-7-6) -- (m-6-4);
  \draw[-stealth, blue]  (m-7-6) -- (m-5-3);
    \draw[-stealth, blue]  (m-7-6) -- (m-4-2);
  \draw[-stealth, blue]  (m-8-8) -- (m-7-6);
    \draw[-stealth, blue]  (m-9-9) -- (m-7-6);

      \begin{pgfonlayer}{background}
\draw[rounded corners, draw=none, fill=black!60!white, inner sep=3pt,fill opacity=0.25]
   (m-10-7.north east)  -| (m-10-10.south east) 
    -| (m-1-10.north east) -| (m-1-7.south east);
    \end{pgfonlayer}

\end{tikzpicture} \vspace{-2em} 
\end{center}
\caption{ The page $E^2_{a,b}$ when $n=7$. The domains of all subsequent differentials into $E^2_{n-3,2}$ are 0, as are the codomains of all subsequent differentials out of $E^2_{n-3,2}$. Thus  $E^2_{n-3,2} \cong E^\infty_{n-3,2}$.}
\label{MPSS}
\end{figure}
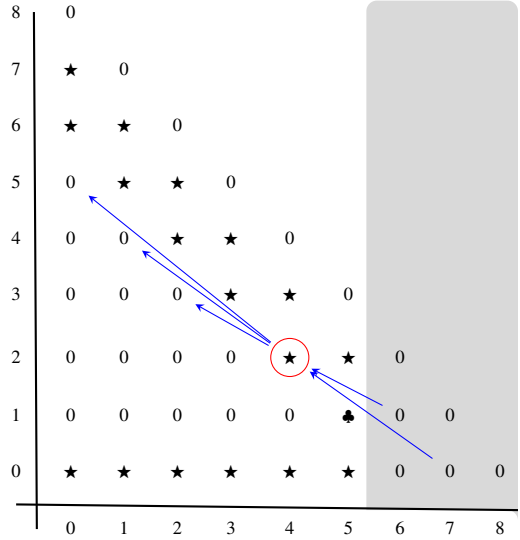

 Since $\cT_n^\pm(\F_p^n)$ is $(n-2)$-dimensional,   $E^2_{a,b} \cong 0$  for $a>n-2$. This region is shaded grey in \autoref{MPSS}. Thus $E^2_{n-3,2} \cong E^\infty_{n-3,2}$ as all higher differentials to or from $E^*_{n-3, 2}$ vanish, as in \autoref{MPSS}.  

Observe that the group $E^\infty_{n-2,1}$ (marked by $\clubsuit$ in \autoref{MPSS}) must vanish,  since $H_1(\BAA_1^\pm(\F_p)) = H_1(\BA_1^\pm(\F_p)) \cong 0$. Thus the abutment of the spectral sequence surjects onto $E^\infty_{n-3,2}$.  All that remains is to identify $E^2_{n-3,2}$ with $[ \Z[Gr_2^n(\F_p)^\pm] \otimes \widetilde H_{n-4}(\mathcal T_{n-2}^\pm(\F_p))$. We will apply \autoref{CharneyLemma}. Observe that the functor $V \mapsto H_2(\phi^{\leq V})$ is supported on vector spaces $V$ of dimension $2$, equivalently, of height 1 in the poset $ \mathbb T_n^\pm(\F_p)$. By  \autoref{CharneyLemma},
$$E^2_{n-3,2} \cong H_{n-3}( \mathbb T_n^\pm(\F_p); H_2(\phi)) \cong\bigoplus_{\operatorname{ht}(V)=1} \widetilde H_{n-4}(\mathbb T_{n-2}^\pm(\F_p); H_2(\BAA_2^\pm(\F_p))
.$$
The set of height-$1$ elements of $\mathbb{T}_n^\pm(\F_p)$ is isomorphic to $Gr_2^n(\F_p)^\pm$, and $H_2(\BAA_2^\pm(\F_p)) \cong \Z$ by \autoref{H2}.  Thus $E^2_{n-3,2} \cong  \Z[Gr_2^n(\F_p)^\pm] \otimes \widetilde H_{n-4}(\mathcal T_{n-2}^\pm(\F_p))$ and the result follows. 
\end{proof}

We now prove \autoref{TheoremC} which gives a numerical lower bound for $H^{{n \choose 2}-1}(\Gamma_n(p)) ;\Q)$.

\begin{proof}[Proof of \autoref{TheoremC}]

We must show that if $p=3$ or $5$, then \[\dim_{\Q} H^{{n \choose 2}-1}(\Gamma_n(p);\Q) \geq p^{n-2 \choose 2} |Gr_2^{n}(\F_p)| \left(\frac{p-1}{2}\right)^{n-2}. \] 
When $n=2$, both sides of the inequality are equal to $1$. Assume $n \geq 3$. 
By \autoref{CatC}, $$ \dim_{\Q} H^{{n \choose 2}-1}(\Gamma_n(p);\Q) \geq |Gr_2^n(\F_p)^\pm| \dim_{\Q} \widetilde H_{n-4}(\cT_{n-2}^\pm(\F_p);\Q).$$ 
Observe that the order of $Gr_2^{n}(\F_p)^\pm$ is $\frac{p-1}{2}$ times the order of $Gr_2^{n}(\F_p)$.  Furthermore, \cite[Page 5]{MPP} contains a proof that $\widetilde H_{n-4}(\cT_{n-2}^\pm(\F_p))$ is a free abelian group of rank at least $p^{n-2 \choose 2} \left( \frac {p-1}{2}\right)^{n-3}$. We deduce  \autoref{TheoremC}. 
\end{proof}

\bibliographystyle{amsalpha}
\bibliography{codim2}

\end{document}